\documentclass[letterpaper,11pt,oneside,reqno]{amsart}
\usepackage{amsfonts,amsmath, amssymb,amsthm,amscd}
\usepackage[height=9.6in,width=5.95in]{geometry}
\usepackage[mpexclude,DIV13]{typearea}
\usepackage{verbatim}
\usepackage{hyperref}
\usepackage{graphicx,psfrag,subfigure,url}
\usepackage[latin1]{inputenc}
\usepackage{latexsym}
\usepackage{lscape}
\usepackage{epsfig}
\include{bibtex}

\usepackage{setspace}

\begin{document}
\bibliographystyle{alpha}
\newcommand{\cn}[1]{\overline{#1}}
\newcommand{\e}[0]{\epsilon}

\newcommand{\maxdyson}{{\rm MAX}}
\newcommand{\mindyson}{{\rm MIN}}
\newcommand{\EE}{\ensuremath{\mathbb{E}}}
\newcommand{\PP}{\ensuremath{\mathbb{P}}}
\newcommand{\var}{\textrm{var}}
\newcommand{\N}{\ensuremath{\mathbb{N}}}
\newcommand{\R}{\ensuremath{\mathbb{R}}}
\newcommand{\Rkle}{\ensuremath{\mathbb{R}^k_{>}}}
\newcommand{\Rklezero}{\ensuremath{\mathbb{R}^k_{>0}}}
\newcommand{\C}{\ensuremath{\mathbb{C}}}
\newcommand{\Z}{\ensuremath{\mathbb{Z}}}
\newcommand{\Q}{\ensuremath{\mathbb{Q}}}
\newcommand{\T}{\ensuremath{\mathbb{T}}}
\newcommand{\Yind}{\ensuremath{\mathcal{L}^{\rm{ind}}}}
\newcommand{\edge}{\textrm{edge}}
\newcommand{\E}[0]{\mathbb{E}}
\newcommand{\bo}{\mathbf}
\newcommand{\OO}[0]{\Omega}
\newcommand{\F}[0]{\mathfrak{F}}
\def \Ai {{\rm Ai}}
\newcommand{\G}[0]{\mathfrak{G}}
\newcommand{\ta}[0]{\theta}
\newcommand{\w}[0]{\omega}
\newcommand{\ra}[0]{\rightarrow}
\newcommand{\vectoro}{\overline}
\newcommand{\ws}{{\rm WS}}
\newcommand{\nc}{{\rm NC}}
\newcommand{\ncf}{{\rm NC}^f}
\newcommand{\ncmi}{{\rm NC}^{-\infty}}
\newcommand{\wxy}{\mathcal{W}_{k;\bar{x},\bar{y}}}
\newcommand{\wxylr}{\mathcal{W}_{k;\bar{x},\bar{y}}^{\ell,r}}
\newcommand{\wxylrprime}{\mathcal{W}_{k;\bar{x}',\bar{y}'}^{\ell,r}}
\newcommand{\ewxy}{\mathcal{E}_{k;\bar{x},\bar{y}}}
\newcommand{\bxyf}{\mathcal{B}_{\bar{x},\bar{y},f}}
\newcommand{\bxyflr}{\mathcal{B}_{\bar{x},\bar{y},f}^{\ell,r}}
\newcommand{\bstar}{\mathcal{B}^{*}_{\bar{x},\bar{y},f}}
\newcommand{\bprime}{\mathcal{B}'_{\bar{x},\bar{y},f}}
\newcommand{\zxyf}{Z_{\bar{x},\bar{y},f}}
\newcommand{\signc}{\Sigma}
\newcommand{\AP}{\mathfrak{a}}
\newcommand{\CM}{\mathfrak{c}}
\newcommand{\CMK}{\textrm{CM}^K_{\ell,r}}
\newcommand{\XYfM}{\textrm{XY}^{f}_M}
\newcommand{\maxone}{\mathcal{M}}
\newcommand{\fext}{\mathcal{F}_{ext}}
\newcommand{\aone}{{\rm A}}
\newcommand{\setm}{{\rm M}}
\newtheorem{theorem}{Theorem}[section]
\newtheorem{partialtheorem}{Partial Theorem}[section]
\newtheorem{conj}[theorem]{Conjecture}
\newtheorem{lemma}[theorem]{Lemma}
\newtheorem{proposition}[theorem]{Proposition}
\newtheorem{corollary}[theorem]{Corollary}
\newtheorem{claim}[theorem]{Claim}
\newtheorem{experiment}[theorem]{Experimental Result}

\def\todo#1{\marginpar{\raggedright\footnotesize #1}}
\def\change#1{{\color{green}\todo{change}#1}}
\def\note#1{\textup{\textsf{\color{blue}(#1)}}}

\theoremstyle{definition}
\newtheorem{rem}[theorem]{Remark}

\theoremstyle{definition}
\newtheorem{com}[theorem]{Comment}

\theoremstyle{definition}
\newtheorem{definition}[theorem]{Definition}

\theoremstyle{definition}
\newtheorem{definitions}[theorem]{Definitions}

\theoremstyle{definition}
\newtheorem{conjecture}[theorem]{Conjecture}

\title{Brownian Gibbs property for Airy line ensembles}

\author[I. Corwin]{Ivan Corwin}
\address{I. Corwin\\
  Microsoft Research New England: Theory group\\
  1 Memorial Drive\\
  Cambridge, MA 02142, USA}
\email{ivan.corwin@gmail.com}

\author[A. Hammond]{Alan Hammond}
\address{A. Hammond\\
  Department of Statistics\\
  University of Oxford\\
  1 South Parks Road\\
  Oxford, OX1 3TG, U.K.}
\email{hammond@stats.ox.ac.uk}

\begin{abstract}
Consider a collection of $N$ Brownian bridges $B_i:[-N,N] \to \R$, $B_i(-N) = B_i(N) = 0$, $1 \leq i \leq N$, conditioned not to intersect. The edge-scaling limit of this system is obtained by taking a weak limit as $N \to \infty$ of the collection of curves scaled so that the point $(0,2^{1/2} N)$ is fixed and space is squeezed, horizontally by a factor of $N^{2/3}$ and vertically by $N^{1/3}$. If a parabola is added to each of the curves of this scaling limit, an $x$-translation invariant process sometimes called the multi-line Airy process is obtained. We prove the existence of a version of this process (which we call the Airy line ensemble) in which the curves are almost surely everywhere continuous and non-intersecting. This process naturally arises in the study of growth processes and random matrix ensembles, as do related processes with ``wanderers'' and ``outliers''. We formulate our results to treat these relatives as well.

Note that the law of the finite collection of Brownian bridges above has the property -- called the Brownian Gibbs property -- of being invariant under the following action. Select an index $1 \leq k \leq N$ and erase $B_k$ on a fixed time interval $(a,b) \subseteq (-N,N)$; then replace this erased curve with a new curve on $(a,b)$ according to the law of a Brownian bridge between the two existing endpoints $\big(a,B_k(a)\big)$ and $\big(b,B_k(b)\big)$, conditioned to intersect neither the curve above nor the one below. We show that this property is preserved under the edge-scaling limit and thus establish that the Airy line ensemble has the Brownian Gibbs property. 

An immediate consequence of the Brownian Gibbs property is a confirmation of the prediction of M. Pr\"{a}hofer and H. Spohn that each line of the Airy line ensemble is locally absolutely continuous with respect to Brownian motion. We also obtain a proof of the long-standing conjecture of K. Johansson that the top line of the Airy line ensemble minus a parabola attains its maximum at a unique point. This establishes the asymptotic law of the transversal fluctuation of last passage percolation with geometric weights. Our probabilistic approach complements the perspective of exactly solvable systems which is often taken in studying the multi-line Airy process, and readily yields several other interesting properties of this process.
\end{abstract}

\maketitle


\section{Introduction}

One-dimensional Markov processes (such as random walks or Brownian motion) conditioned not to intersect form an important class of models which arise in the study of random matrix theory, growth processes, directed polymers, tilings and certain problems in combinatorics and representation theory (see the surveys \cite{PLF,FSreview,KJSurvey,HS}).

The probability distribution for these collections of lines may be analysed using a technique which exploits the non-intersection property: in different guises this tool is the Karlin-McGregor formula, the Lingstrom-Gessel-Viennot formula, and the physics method of free fermions. These methods yield exact expressions (as determinants) for the statistics of such conditioned processes. Asymptotic analysis then gives rise to certain universal scaling limits in the sense of convergence of finite-dimensional distributions and accordingly provides exact expressions for these distributions. Striking examples of this are the works of \cite{BDJ,PS} which give limit theorems for the finite-dimensional distribution of the fluctuations of the height functions for the polynuclear growth (PNG) model in terms of Fredholm determinants involving the extended ${\rm Airy}_2$ kernel (see also \cite{OkResh}). Using exactly solvable methods, \cite{KJPNG} further proved a functional limit theorem and established the existence of a continuous version of a stochastic process with the above family of finite-dimensional distributions which is known as the Airy (or sometimes ${\rm Airy}_2$) process.

Such scaling limits have often been studied by 
analysing  exact (and often determinantal) formulas for finite-dimensional distributions and correlation functions.   It is natural to think that such canonical random processes might also be studied by probabilistic techniques; and this study seems all the more warranted in light of the numerous significant questions about these limiting processes which have remained unanswered by the existing approaches. In this paper, we begin such a probabilistic study and solve a number of these open questions.

The main tool which we will employ is the notion of a {\it Gibbs property}. Let us illustrate this idea in an informal way, so as to avoid introducing unnecessary notation (the specific {\it Brownian Gibbs property} which we will use is given explicitly in Definition \ref{maindefBGP}). In studying random processes (or fields) $X$ from $\Sigma\to \R$ (where $\Sigma$ may be taken to be a discrete lattice, graph, or Euclidean space), it is natural to consider the conditional distribution of $X$ inside a compact subset $C$ of $\Sigma$, given its values on $\Sigma\setminus C$.  There is a class of such processes with the property that these distributions depend only on the values of $X$ on the boundary of $C$ (in the discrete case, we mean the exterior boundary). Furthermore, given these values, the distribution of $X$ on $C$ is determined relative to a reference measure in terms of a Radon Nikodym derivative which is often written in terms of a Hamiltonian, and which is measurable with respect to the sigma-field generated by $X$ on $C$ and its (exterior) boundary. This property, which is often called the Gibbs property, may be regarded as a spatial version of the Markov property. Gibbs properties for random lattice indexed fields are ubiquitous in models of statistical physics (for instance, in Ising or Potts models) and have received extensive treatment since the seminal work of Dobrushin \cite{Dob}.

We dispense with generalities and focus on a simple, though highly relevant, example of the Gibbs property. Fix $N,T \in\N$ and let  $\{X_k: [-T,T] \to \Z \}_{k=1}^{N}$ denote a system of $N$ simple symmetric random walks (starting at time $-T$ and ending at time $T$)
 whose starting and ending points satisfy $X_k(-T)=X_k(T)=-k+1$, and which are conditioned on the non-intersection requirement that $X_{k}(i)>X_{k+1}(i)$ for all $i$ and $k$. Fix two integers $-T\leq a<b\leq T$ and a line index $1\leq k\leq N$. Then, conditioned on the values of $X_k(a)$, $X_k(b)$ and $X_{k\pm 1}(i)$ for $i\in (a,b)$, the law of $\{X_{k}(i):i\in (a,b)\}$ is uniform over the set of simple random walk paths of length $b-a$ between $X_k(a)$ and $X_k(b)$ that intersect neither $X_{k+1}$ nor $X_{k-1}$. (If $k \in \{1,n\}$, there is only one path to be avoided.) In other words, the line ensemble measure is invariant under resampling according to the usual random walk measure subject to the non-intersection condition.

The continuum counterpart of this system is a variant of Dyson's Brownian motion given by several Brownian bridges conditioned not to intersect. The resampling property has an analogue as well: the rule is the same, except that the underlying measure is now an independent system of Brownian bridges between the given endpoints, rather than the uniform law on simple random walk paths. This property, which we call the {\it Brownian Gibbs property} and introduce formally in Definition \ref{maindefBGP}, is easily understood and in many ways unsurprising.

The main contribution of this article is the observation of Theorem \ref{mainthm2}, that this Brownian Gibbs property is preserved in the edge scaling limits of non-intersecting line-ensembles (such as the random walk and Brownian bridge examples which we have seen). From this perspective, the observation seems fairly intuitive; however, the processes encountered in these scaling limits arise in many other perspectives from which this Gibbs property is neither straightforward nor intuitive -- in fact, this observation has not previously appeared in the literature.

Before discussing the consequences of the Brownian Gibbs property, we briefly discuss the two main difficulties in the proof of the result. Each difficulty derives from the edge scaling limit results for such systems having only been proved in terms of finite-dimensional distributions. The first problem is whether the limiting consistent family of finite-dimensional distributions, which is called the multi-line Airy process \cite{PS}, has a version which is supported on continuous, non-intersecting curves. Restricting to just the top line, the existence of a continuous version was shown in \cite{KJPNG} by the exactly solvable systems approach involving asymptotic analysis of Fredholm determinant expressions. The second difficulty is to show convergence in distribution of the finite system of Brownian bridges to this limiting ensemble. Once these two difficulties are settled, a coupling argument serves to prove that the Brownian Gibbs property is maintained in the limit.

In fact, we resolve both difficulties simultaneously by proving a functional limit theorem for the Brownian bridge line ensembles (which is in essence tightness for this family of curves), thus showing the existence of the {\it Airy line ensemble}. In proving such tightness (as well as the non-intersecting nature of the limit), one cannot appeal entirely to the finite-dimensional distributions, as they pertain only to a finite number of deterministic times and not random exceptional times at which there may be a small gap between lines or a large modulus of continuity for some line. The proof must also handle the squeezing of space that occurs under edge scaling. The determinantal approach seeming to fail here, we appeal to a probabilistic perspective. We employ the Brownian Gibbs property of $N$ Brownian bridges to show that, focusing on the top $k$ curves in a rescaled window of size $[-T,T]$, the following three events occur with high probability: (i) the collection of curves remains uniformly absolutely continuous with respect to $k$ Brownian bridges; (ii) the minimal gap between consecutive curves remains uniformly bounded from below; (iii) the top curve and the $k$-th curve remain uniformly bounded from above and below.

That the scaling limit has the Brownian Gibbs property has a number of significant consequences. In particular, in Proposition \ref{propabscon} we show that the top line of the Airy line ensemble (often known as the Airy process) is absolutely continuous with respect to Brownian motion (with diffusion parameter 2) on any fixed interval (this is also true for every other line). This leads to a  proof of Conjecture 1.5 of \cite{KJPNG} (stated here as Theorem \ref{propjoh}), a long-standing claim that the Airy process minus a parabola achieves its maximum at a unique point. As we explain in Section \ref{absconsec}, this conjecture leads to Theorem \ref{maxdist}, which identifies and proves the convergence of the law of the endpoint of a geometrically weighted ground state directed polymer (i.e., of the point-to-line maximizing path in last passage percolation). Appealing to the Brownian absolute continuity result we prove herein, \cite{MQR} has employed the Airy process continuum statistics of \cite{CQR} to give  and prove an exact formula for the law of this endpoint.  Proposition \ref{propabscon} also extends results of \cite{Hagg} and proves a functional central limit theorem for the Airy process on short time scales. In the process of proving our main theorem, we record several results of independent interest (see Section \ref{absconsec}), including the proof of Conjecture 1.21 of \cite{AvM}.

The extensive literature on non-intersecting line ensembles suggests many directions in which to pursue the approach introduced here. The non-intersecting random walkers' model described above goes by the name {\it vicious walkers} and was introduced by de Gennes \cite{Gen} as a model of directed fibrous structures in $1+1$ dimensions. This model was then studied  in \cite{fisher,HF}. This area has remained of interest due to its connections with symmetric function theory (specifically Schur functions and enumeration of Young tableaux as in \cite{GOV}), discrete analogues of random matrix theory (for instance \cite{NF}), and 2D Yang-Mills theory \cite{FMS,HT}. The non-intersecting Brownian line ensembles that we consider are variants of Dyson Brownian motion \cite{dyson, DG, NM} and can be considered as continuum versions of the vicious walkers' model.
Part of the attention which they attract is due to the description that they offer of random matrix eigenvalue processes \cite{KJnonint}.

Within the above literature are a number of interesting and important scaling limits and perturbations to the basic model considered in this paper. Perturbations to the first few Brownian particles result in perturbed Airy-like limiting line ensembles \cite{ADvM,AFvM} (see our Section \ref{airylikeSEC}). Scaling limits in the vicinity of the bulk of the Brownian bridges lead to the Dyson sine process, and large-scale perturbations (such as separating half of the starting points and/or ending points) lead to perturbations in the limiting processes such as the Pearcy \cite{TWpearcy} or tacnode processes \cite{BDTac}. Likewise, limiting processes have been derived when the Brownian bridges are replaced by alternative path measures such as Brownian excursions \cite{TWex}, or Bessel processes \cite{KTBes}.

The approach developed in this paper may be extended to study line ensembles with a Gibbs property which penalises but does not exclude crossing of curves with consecutive labels. In an upcoming article \cite{CH2}, we employ our probabilistic techniques to construct the {\it KPZ line ensemble} (related to the multi-layer extension of the stochastic heat equation constructed in \cite{OConWar}) and prove that is has such a ``soft'' Brownian Gibbs property. The top labeled curve of this line ensemble is the fixed time Hopf-Cole solution to the KPZ equation with narrow-wedge initial data (see \cite{ACQ}). As a result of the Gibbs property we are able to prove that this solution to the KPZ equation is locally absolutely continuous with respect to Brownian motion (analogously to Proposition \ref{propabscon}).

In the present paper, non-intersecting Brownian paths play the key role; in \cite{CH2}, this role is assumed by diffusions associated to the quantum Toda lattice Hamiltonian. O'Connell \cite{OCon} discovered these diffusions and their relationship to directed polymers (and hence also to the KPZ equation -- see \cite{AKQ,QM}).

Returning to discrete line ensembles, there are a variety of examples coming from the study of growth processes (such as the various PNG line ensembles \cite{PS,KJPNG}) which have the same large $N$ limits as in the Brownian case \cite{IS,BP,BFP,CFP1,OkResh,KJnonint}. Measures on rhombus tilings are closely related to representation theory as well as to vicious walkers. In that setting, \cite{BS} relates Gibbs properties of discrete sine line ensembles to the associated determinant kernel. The papers \cite{BO,BG} construct infinite-dimensional Gibbs line ensembles corresponding to the bulk scaling limits of these tiling related line ensembles.

Rhombus tilings also represent perfect matchings for the honeycomb lattice. In \cite{KOS}, it is shown that models based on perfect matchings (on any weighted doubly-periodic bipartite graph $G$ in the plane) are exactly solvable in a rather strong sense. The authors of \cite{KOS} not only derive explicit formulas for the surface tension; they classify the Gibbs measures on tilings and explicitly compute the local probabilities in each of them. These results are a generalization of \cite{CKP} where similar results for $G = \Z^2$ with constant edge weights were obtained.

\subsection{Outline}
Subsection \ref{lineensembledef} contains the definition of a line ensemble and the Brownian Gibbs property; Subsection \ref{twohelp} contains the statement of the strong Gibbs property and monotonicity properties for Brownian Gibbs line ensembles; Subsection \ref{Airylineensemblesec} contains the definitions of the edge-scaled Dyson line ensemble.

The main results of this paper are contained in Section \ref{mainresec}. Subsection \ref{secunique} contains a uniqueness conjecture regarding the Airy line ensemble; Subsection \ref{hypsec} contains a more general formulation of the main result and a set of general hypotheses under which it holds; Subsection \ref{airylikeSEC} records how many other line ensembles satisfy these more general hypotheses.

Section \ref{absconsec} presents some interesting consequences of our main results. Subsection \ref{subsubabs} includes a proof that any line of the Airy line ensemble minus a parabola has increments which are absolutely continuous with respect to the Brownian bridge (with diffusion parameter 2); In Subsection \ref{uniquenesssec}, this absolute continuity is used to prove the uniqueness of the location at which the maximum of the top line of the Airy line ensemble minus a parabola is attained. In Subsection \ref{transec} this uniqueness result implies that this location describes the endpoint of a directed polymer model. The section also contains the proof of a conjecture of Adler and van Moerbeke.

The main technical results of the paper are contained in the estimates of Proposition \ref{propacceptprob} which is proved in Section \ref{technicalSec}. Finally, Section \ref{proofslemmas} contains a proof of the strong Gibbs property Lemma \ref{stronggibbslemma} and the monotonicity Lemmas~\ref{monotonicity} and \ref{lemmonotonetwo}.

\subsection{Acknowledgments}
This project was initiated at the 2010 Clay Mathematics Institute Summer School in Buzios, Brazil. The authors also thank the Mathematical Science Research Institute, the Fields Institute and the Mathematisches Forschungsinstitut Oberwolfach for their hospitality and support, as much of this work was completed during stays at these institutes. We thank Jinho Baik, Jeremy Quastel and Herbert Spohn for their input and interest. We also thank our referee for a thorough reading of this work and many useful comments. A.H. would like to thank Scott Sheffield for drawing attention to a talk in 2006 in which Andrei Okounkov proposed problems closely related to the discussion in Section~\ref{secunique} and for interesting ensuing conversations, and Neil O'Connell and Jon Warren for useful early discussions regarding approaches to proving the results in this article. I.C. recognizes support and travel funding from the NSF through the PIRE grant OISE-07-30136 and grant DMS-1208998; as well as support from the Clay Mathematics Institute through a Clay Research Fellowship and Microsoft Research through the Schramm Memorial Fellowship. A.H. was supported principally by EPSRC grant EP/I004378/1.

\section{Definitions, notations and basic lemmas}\label{defnotsec}

We begin by providing the necessary definitions and notations to state our main results. We also include the statement of our two primary tools utilized in proving the main result -- the strong Gibbs property and the monotonicity properties of line ensembles. We also include some basic facts about Brownian motions and Brownian bridges.

\subsection{Line ensembles and the Brownian Gibbs property}\label{lineensembledef}
In order to state our main result we must introduce the concept of a line ensemble and the Brownian Gibbs property.

\begin{definition}\label{maindef}
Let $\Sigma$ be a (possibly infinite) interval of $\Z$, and let $\Lambda$ be an interval of $\R$. 
Consider the set $X$ of continuous functions $f:\Sigma\times \Lambda \rightarrow \R$ endowed with the topology of uniform convergence on compact subsets of $\Sigma\times\Lambda$. Let $\mathcal{C}$ denote the sigma-field  generated by Borel sets in $X$.

A {\it $\Sigma$-indexed line ensemble} $\mathcal{L}$ is a random variable defined on a probability space $(\Omega,\mathcal{B},\PP)$, taking values in $X$ such that $\mathcal{L}$ is a $(\mathcal{B},\mathcal{C})$-measurable function. Intuitively, $\mathcal{L}$ is a collection of random continuous curves (even though we use the word ``line'' we are referring to continuous curves), indexed by $\Sigma$, each of which maps $\Lambda$ into $\R$. We will often slightly abuse notation and write $\mathcal{L}:\Sigma\times \Lambda \rightarrow \R$, even though it is not $\mathcal{L}$ which is such a function, but rather $\mathcal{L}(\omega)$ for each $\omega \in \Omega$. Furthermore, we write $\mathcal{L}_i:=(\mathcal{L}(\omega))(i,\cdot)$ for the line indexed by $i\in\Sigma$. Given a $\Sigma$-indexed line ensemble $\mathcal{L}$, and a sequence of such ensembles $\big\{ \mathcal{L}^N: N \in \N \big\}$,  a natural notion of convergence is the weak-* convergence of the measure on $(X,\mathcal{C})$ induced by $\mathcal{L}^N$, to the measure induced by $\mathcal{L}$; we call this notion {\it weak convergence as a line ensemble} and denote it by $\mathcal{L}^N\Rightarrow \mathcal{L}$. In order words, this means that for all bounded continuous functionals $f$, $\int d\PP(\omega) f(\mathcal{L}^{N}(\omega)) \to \int d\PP(\omega) f(\mathcal{L}(\omega))$ as $N\to \infty$. A line ensemble is {\it non-intersecting} if, for all $i<j$, $\mathcal{L}_i(r)>\mathcal{L}_j(r)$ for all $r\in \Lambda$. All statements are to be understood as being almost sure with respect to $\PP$.
\end{definition}

We turn now to formulating the Brownian Gibbs property. As a matter of convention, all Brownian bridges will have diffusion parameter 1 unless otherwise noted.
\begin{definition}\label{maindefBGP}
Let $k \in \N$. A point $\vectoro{x}=(x_{1},\ldots,x_{k} \big) \in \R^k$ is called a $k$-decreasing list if $x_i > x_{i+1}$ for $1 \leq i \leq k-1$. We write $\Rkle \subseteq \R^k$ for the set of $k$-decreasing lists. Let $\vectoro{x}=(x_{1},\ldots,x_{k} \big)$ and  $\vectoro{y}=(y_{1},\ldots, y_{k} \big)$ be two $k$-decreasing lists. Let $a,b \in \R$ satisfy $a < b$, and  let $f,g:[a,b] \to \R^*$ (where $\R^*=\R\cup\{-\infty,+\infty\}$) be two given continuous functions that satisfy $f(r)>g(r)$ for all $r\in[a,b]$ as well as the boundary conditions $f(a)>x_{1}$, $f(b)>y_{1}$ and $g(a)<x_{k}$, $g(b)<y_{k}$.

The {\it $(f,g)$-avoiding Brownian line ensemble on the interval $[a,b]$ with entrance data  $\vectoro{x}$ and exit data $\vectoro{y}$} is a line ensemble $\mathcal{Q}$ with $\Sigma=\{1,\ldots, k\}$, $\Lambda=[a,b]$ and with the law of $\mathcal{Q}$ equal to the law of $k$ independent Brownian bridges $\{B_i:[a,b] \to \R\}_{i=1}^{k}$ from $B_i(a) = x_i$ to $B_i(b) = y_i$ conditioned on the event that $f(r) > B_1(r)>B_2(r)>\cdots >B_k(r) > g(r)$ for all $r \in [a,b]$. Note that any such line ensemble $\mathcal{Q}$ is necessarily non-intersecting.

Now fix an interval $\Sigma\subseteq \Z$ and $\Lambda\subseteq \R$ and let $K=\{k_1,k_1+1,\ldots,k_2-1, k_2\} \subset \Sigma$ and $a,b \in \Lambda$, with $a <b$. Set $f=\mathcal{L}_{k_1-1}$ and $g=\mathcal{L}_{k_2+1}$ with the convention that if $k_1-1\notin\Sigma$ then $f\equiv +\infty$ and likewise if $k_2+1\notin \Sigma$ then $g\equiv -\infty$. Write $D_{K,a,b} = K \times (a,b)$ and $D_{K,a,b}^c = (\Sigma \times \Lambda) \setminus  \Lambda_{K,a,b}$. A $\Sigma$-indexed line ensemble $\mathcal{L}:\Sigma\times \Lambda \to \R$ is said to have the {\it Brownian Gibbs property} if
\begin{equation*}
\textrm{Law}\Big(\mathcal{L} \big\vert_{D_{K,a,b}} \textrm{conditional on } \mathcal{L} \big\vert_{D_{K,a,b}^c}\Big) = \textrm{Law}(\mathcal{Q}),
\end{equation*}
where $\mathcal{Q}_{i}=\tilde{\mathcal{Q}}_{i-k_1+1}$ and $\tilde{\mathcal{Q}}$ is the $(f,g)$-avoiding Brownian line ensemble on $[a,b]$ with entrance data $\big(\mathcal{L}_{k_1}(s),\ldots,\mathcal{L}_{k_2}(s)\big)$ and exit data  $\big(\mathcal{L}_{k_1}(t),\ldots,\mathcal{L}_{k_2}(t)\big)$. Note that $\tilde{\mathcal{Q}}$ is introduced because, by definition, any such $(f,g)$-avoiding Brownian line ensemble is indexed from $1$ to $k_2-k_1+1$, but we want $\mathcal{Q}$ to be indexed from $k_1$ to $k_2$.
\end{definition}

\begin{definition}\label{WBdef}
Let $k\in \N$, $a <b$, and $\bar{x},\bar{y}\in \Rkle$. Write $\wxy^{a,b}$ for the law of $k$ independent Brownian bridges
$B_i:[a,b] \to \R$, $1 \leq i \leq k$, that satisfy $B_i(a) = x_i$ and $B_i(b) = y_i$. Write $\ewxy^{a,b}$ for the expectation under $\wxy^{a,b}$.

Let  $f:[a,b] \to \R$ be a measurable function such that $x_{k}>f(a)$ and $y_k>f(b)$.
Define the non-crossing event on an interval $A\subset [a,b]$ by
\begin{equation*}
\ncf_A =\Big\{ \, \textrm{for all } r\in A,  B_{i}(r) > B_{j}(r) \textrm{ for all } 1\leq i<j\leq k \textrm{ and  $B_k(r) > f(r)$} \Big\}.
\end{equation*}

The conditional measure $\wxy^{a,b} \big( \cdot \big\vert \ncf_{[a,b]} \big)$ is the $(\infty,f)$-avoiding line ensemble on $[a,b]$  with entrance data $\bar{x}$ and exit data $\bar{y}$; it will be denoted by $\bxyf^{a,b}(\cdot)$.

We define the {\it acceptance probability} as
\begin{equation*}
\AP(a,b,\bar{x},\bar{y},f)= \wxy^{a,b}(\ncf_{[a,b]}).
\end{equation*}
Note that this implies that
\begin{equation*}
\bxyf^{a,b} (E) = \frac{\wxy^{a,b}(E\cap \ncf_{[a,b]})}{\AP(a,b,\bar{x},\bar{y},f)}.
\end{equation*}
\end{definition}

\subsection{Two helpful lemmas}\label{twohelp}
The following two sets of lemmas will be essential to the proof of our main results. The proofs of these results appear at the end of the paper, in Section \ref{proofslemmas}. We end this subsection with a few useful tidbits about Brownian bridges.

\subsubsection{Strong Gibbs property}
In order to introduce the strong Gibbs property for Brownian Gibbs line ensembles, we first introduce the concept of a stopping domain.

\begin{definition}\label{defstopdom}
Consider a line ensemble $\mathcal{L}: \{1,\ldots, N\}\times [a,b]\to \R$. 
For $a<\ell<r<b$, and $1\leq k\leq N$ denote the sigma-field generated by $\mathcal{L}$ outside $\{1,\ldots,k\}\times [a,b]$ by
\begin{equation*}
\mathcal{F}_{ext}(k,\ell,r)=\sigma\Big\{\mathcal{L}_1,\ldots, \mathcal{L}_k \textrm{ on } [a,b]\setminus(\ell,r), \textrm{ and } \mathcal{L}_{k+1},\ldots,\mathcal{L}_N \textrm{ on } [a,b]\Big\}.
\end{equation*}

The random variable $(\mathfrak{l},\mathfrak{r})$ is called a {\it stopping domain} for lines $\mathcal{L}_1,\ldots,\mathcal{L}_k$ if for all $\ell<r$,
\begin{equation*}
\{\mathfrak{l} \leq \ell , \mathfrak{r}\geq r\} \in \mathcal{F}_{ext}(k,\ell,r).
\end{equation*}
In other words, the domain is determined by the information outside of it. We will generally assume that, when discussing a stopping domain, it is for the top $k$ indexed line in the line ensemble (and will not mention this).
\end{definition}

We will make use of a version of the strong Markov property where the concept of stopping domain introduced in Definition \ref{defstopdom} plays the role of stopping time.

Let $C^k(\ell,r)$ denote the set of continuous functions $(f_1,\ldots, f_k)$ with each $f_i:[\ell,r]\to \R$. Define
\begin{equation*}
C^k = \left\{ (\ell,r,f_1,\ldots, f_k): \ell<r \textrm{ and } (f_1\ldots, f_k)\in C^k(\ell,r)\right\}.
\end{equation*}
Let $bC^k$ denote the set of Borel measurable functions from $C^k\to \R$.



\begin{lemma}\label{stronggibbslemma}
Consider a line ensemble $\mathcal{L}: \{1,\ldots, N\}\times [a,b]\to \R$ which has the Brownian Gibbs property.
 Write $\PP$ and $\EE$ as the probability measure and expectation on $\mathcal{L}$.
 Fix $k\in \{1,\ldots, N\}$. For all random variables $(\mathfrak{l},\mathfrak{r})$ which are stopping domains for lines $\mathcal{L}_1,\ldots, \mathcal{L}_k$, the following {\it strong Brownian Gibbs property} holds: for all $F\in bC^k$,  $\PP$ almost surely,
\begin{equation*}\label{strongeqn}
\EE\Big[F \big( \mathfrak{l},\mathfrak{r},\mathcal{L}_{1}\big\vert_{(\mathfrak{l},\mathfrak{r})},\ldots, \mathcal{L}_{k}\big\vert_{(\mathfrak{l},\mathfrak{r})} \big) \Big\vert \mathcal{F}_{ext}(k,\mathfrak{l},\mathfrak{r})\Big] = \bxyf^{\mathfrak{l},\mathfrak{r}} \big[F(\mathfrak{l},\mathfrak{r},B_1,\ldots, B_k) \big],
\end{equation*}
where $\bar{x} = \{\mathcal{L}_i(\mathfrak{l})\}_{i=1}^{k}$, $\bar{y} = \{\mathcal{L}_i(\mathfrak{r})\}_{i=1}^{k}$, $f(\cdot)=\mathcal{L}_{k+1}(\cdot)$ (or $-\infty$ if $k=N$), $\bxyf^{\mathfrak{l},\mathfrak{r}}$ is given in Definition~\ref{WBdef}.
\end{lemma}
The proof of this lemma is given in Section \ref{proofslemmas}. The main message of the lemma is that the distribution of the value of a line ensemble inside a stopping domain is entirely determined by the boundary data and specified according to the non-intersecting Brownian bridge measure.

\subsubsection{Monotonicity results}
The following lemmas demonstrate certain typies of monotonicity which exist between non-intersecting Brownian bridge measures.

\begin{lemma}\label{monotonicity}
Fix $k\in \N$, $a<b$ and two measurable functions $f,g:[a,b]\rightarrow \R\cup\{-\infty\}$ such that for all $s\in [a,b]$, $f(s)\leq g(s)$. Let $\bar{x},\bar{y}\in \Rkle$ be two k-decreasing lists such that $x_k \geq g(a)$ and $y_k \geq g(b)$.
Recalling Definition~\ref{WBdef}, set $\PP^{k}_f = \wxy^{a,b} \big( \cdot \big\vert \nc^{f}_{[a,b]} \big)$, and likewise define $\PP^{k}_g$. Then there exists a coupling of $\PP^{k}_f$ and $\PP^{k}_g$ such that almost surely $B^f_i(s)\leq B^g_i(s)$ for all $i\in \{1,\ldots, k\}$ and all $s\in [a,b]$.
\end{lemma}

\begin{lemma}\label{lemmonotonetwo}
Fix $k\in \N$, $a<b$, a measurable function $f:[a,b]\rightarrow \R\cup\{-\infty\}$ and a measurable set $A\subseteq [a,b]$. Consider two pairs of k-decreasing lists $\bar{x},\bar{y}$ and $\bar{x}',\bar{y}'$ such that $x_k,x'_k\geq f(a)$, $y_k,y'_k \geq f(b)$ and $x_i' \geq x_i$ and $y_i' \geq y_i$ for each $1 \leq i \leq k$.
Then the laws $\wxy^{a,b} \big( \cdot \big\vert \ncf_A \big)$ and $\mathcal{W}^{a,b}_{k;\bar{x}',\bar{y}'} \big( \cdot \big\vert \ncf_A \big)$ may be coupled so that, denoting by $B_i$ and $B'_i$ the curves defined under the respective measures, $B_i'(s) \geq B_i(s)$ for each $1 \leq i \leq k$ and for all $s \in [a,b]$.
\end{lemma}

\subsubsection{Brownian bridge properties}\label{BBD}

The following decomposition is closely related to the L\'{e}vy-Ciesielski construction of Brownian motion discussed in \cite{McKean} and can be proved by checking that the covariance of the constructed process $B(s)$ coincides with that of a Brownian bridge.

\begin{lemma}\label{browniandecomp}
Fix $j\in \N$, $T>0$ and consider a sequence of times $0=t_0<t_1<\cdots < t_j=T$. Define a sequence of independent centered Gaussian random variables $\{N_i\}_{i=1}^{j-1}$ so that
\begin{equation*}
\EE[N_i^2] = \frac{(t_i-t_{i-1})(T-t_{i})}{(T-t_{i-1})}.
\end{equation*}
For each $i\in \{1,\ldots,j-1\}$ define an interpolation function
\begin{equation*}
I_i(s) =
\begin{cases}
0 & 0\leq s\leq t_{i-1}\\
\frac{s-t_{i-1}}{t_i-t_{i-1}} N_i & t_{i-1}\leq s\leq t_i\\
\frac{T-s}{T-t_i}N_i & t_i\leq s\leq T,
\end{cases}
\end{equation*}
and let $I_j\equiv 0$.
Define a sequence of independent Brownian bridges $\{B_i\}_{i=1}^j$ such that $B_i:[0,t_i-t_{i-1}]\rightarrow \R$ with the property that $B_i(0)=B_i(t_i-t_{i-1})=0$ and let $m(s) = \max \big\{i:t_i<s \big\}$. Then the random function $B:[0,T]\rightarrow \R$,
\begin{equation*}
B(s) = \sum_{i=1}^{m(s)+1} I_i(s) +   B_{m(s)+1}(s-t_{m(s)})\, ,
\end{equation*}
is equal in law to a Brownian bridge $B'$ on $[0,T]$ (i.e., with $B'(0)=B'(T)=0$ and $\EE[B'(s)^2] = \tfrac{s(T-s)}{T}$).
\end{lemma}


\begin{figure}
\begin{center}
\includegraphics[width=0.5\textwidth]{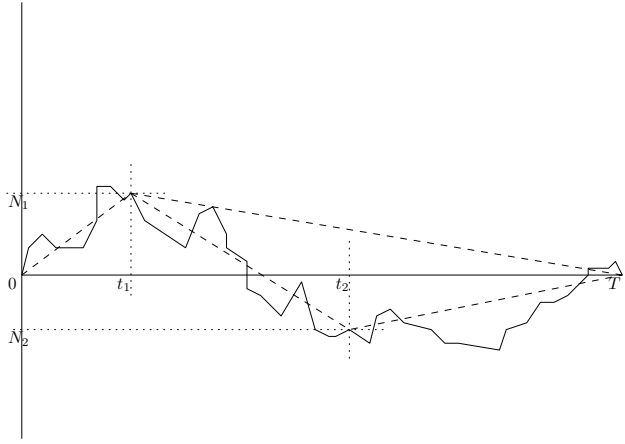}
\end{center}
\caption{Illustrating Lemma \ref{browniandecomp}. A Brownian bridge may be constructed by fixing its value at an intermediate time $t_1$ according to a suitably scaled Gaussian random variable $N_1$, and then inserting a Brownian bridge (plus linear shift) between $(0,0)$ and $(t_1,N_1)$ and a second Brownian bridge (plus linear shift) between $(t_1,N_1)$ and $(T,0)$. This construction can be performed iteratively and results in a decomposition involving a collection of independent Gaussian random variables, and independent Brownian bridges. The case illustrated corresponds to two iterates of this construction.
%
%
%
}\label{figbrbrconst}
\end{figure}

\begin{corollary}\label{cor:touch}
Fix a continuous function $f:[0,1]\rightarrow \R$ such that $f(0)>0$ and $f(1)>0$. Let $B$ be a standard Brownian bridge on $[0,1]$. Define two events: $C=\{\exists \, t\in (0,1): B(t)>f(t)\}$ (crossing) and $T=\{\exists \, t\in (0,1):B(t)=f(t)\}$ (touching). Then $\PP(T\cap C^{c})=0$.
\end{corollary}
\begin{proof}
From Lemma \ref{browniandecomp} (with the choice $j=2$, $T=1$ and $t_1=1/2$), we can decompose $B$ into two independent Brownian bridges $B_i:[0,1/2]\rightarrow \R$ for $i=1,2$, and an independent centered Gaussian random variable  $N_1$ with $\EE[N_1^2]=1/4$. Let $E=T\cap C^{c}$. Then, conditioned on $B_1$ and $B_2$, the event $E$ holds only for a particular (though random) value of $N_1$. However, due to independence and since $N_1$ is Gaussian, the probability it takes a given value is zero, thus proving the corollary.
\end{proof}

Although the following result may be considered standard, we include it for the reader's convenience.
\begin{corollary}\label{coropen}
Consider the space of continuous functions from $[0,1]$ to $\R$ endowed with the uniform topology and let $U$ be an open subset which contains a function $f$ such that $f(0)=f(1)=0$. Let $B:[0,1]\rightarrow \R$ be a standard Brownian bridge. Then $\PP(B[0,1] \subseteq U)>0$.
\end{corollary}
\begin{proof}
Set $E=\{B[0,1] \subseteq U\}$. The set $U$ being open, it contains a piecewise linear function $f$ such that $f(0)=f(1)=0$. Moreover, there is a $\delta>0$ such that $\{g:|g-f|\leq \delta\}\subseteq U$. Assume that $f$ has a discontinuous derivative at times $t_1<\cdots t_{j-1}$ for some $j$ (also set $t_0=0$ and $t_j=1$). Using the decomposition of Lemma \ref{browniandecomp} (into $N_i$ and $B_i$) we have
$$
\PP[E] \geq \PP \Big( \bigcap_{i=1}^{j-1} \big\{|B(t_i)-f(t_i)|<\delta/2\big\}\Big) \PP \Big( \max_{1\leq i\leq j}\max_{s\in [0,t_{i}-t_{i-1}]} |B_i(s)| \leq \delta/2 \Big).
$$
Since $\{B(t_i)\}_{i=1}^{j-1}$ are jointly Gaussian, the first term on the right can be bounded from below by some $\eta=\eta(\delta,t_1,\ldots,t_{j-1})$, while the independence of the bridges $B_{i}(\cdot)$ imply that the second term can be bounded below by
$$
\prod_{i=1}^{j} \PP \Big( \max_{s\in [0,t_{i}-t_{i-1}]} |B_i(s)|  \leq \delta/2 \Big) >\eta'>0
$$
for  some other $\eta'=\eta'(\delta,t_1,\ldots,t_{j-1})$. This last fact follows because, by Lemma~\ref{lembridgemod}, the maximum absolute value of a Brownian bridge has  a continuous distribution on $(0,\infty)$. From the above two bounds, it follows that $\PP[E]>\eta\eta'>0$, as desired.
\end{proof}
We also record a fact about Brownian bridge which will be useful on several occasions.
\begin{lemma}\label{lembridgemod}
Let $B:[0,T] \to \R$, $B(0)=B(T)=0$, be a Brownian bridge.
Let $M^+=\sup \big\{ B(t): 0 \leq t \leq T \big\}$.  Then, for $r > 0$,
$$
\PP \Big( M^+ > r \Big) = \exp \big\{ - \tfrac{2r^2}{T} \big\}.
$$
\end{lemma}
\begin{proof}
The formula appears as $(3.40)$ in Chapter $4$ of  \cite{Karat-Schreve}.
\end{proof}

The following result is also used. Recall that (specializing the notation from Definition~\ref{WBdef}) we write the measure of one Brownian bridge $B:[a,b] \to \R$ with constraints $B(a) = x$ and $B(b) = y$ as $\mathcal{W}^{a,b}_{1;x,y}$.
\begin{lemma}\label{lembrownbridge}
Let $M,\delta > 0$. Then
$$
\mathcal{W}_{1;\delta,M}^{0,1} \Big( B(s) > 0 \, \, \forall \, \, s \in [0,1]   \Big)
 \leq 4 \big(2/\pi\big)^{1/2} \big( \delta M \big)^{1/2}.
$$
\end{lemma}
\begin{proof}
Note that $B:[0,1] \to \R$ under $\mathcal{W}_{1;\delta,M}^{0,1}$
may be represented $B(s) = \delta + (M - \delta)s +  B'(s)$, where $B':[0,1] \to \R$, $B'(0) = B'(1) = 0$, is   standard Brownan bridge.
Since  $\delta + (M - \delta)s \leq 2 \delta$ for $s \in [0,\delta M^{-1}]$,
\begin{equation}\label{eqlembbone}
\mathcal{W}_{1;\delta,M}^{0,1} \Big( B(s) > 0 \, \, \forall \, \, s \in [0,1]   \Big)
\leq
\mathcal{W}_{1;0,0}^{0,1} \Big( B(s) > - 2 \delta \, \, \forall \, \, s \in \big[ 0, \delta M^{-1} \big] \Big).
\end{equation}
We further write $\mathcal{W}_{1;0,*}^{0,1}$ for Brownian motion $B:[0,1] \to \R$, $B(0) = 0$,
(the $*$ indicating that the right-hand endpoint value is unspecified).
Note then that
\begin{eqnarray}
 & & \mathcal{W}_{1;0,0}^{0,1} \Big( B(s) > - 2 \delta \, \, \forall \, \, s \in \big[ 0, \delta M^{-1} \big] \Big) \nonumber \\
 & \leq &  \mathcal{W}_{1;0,*}^{0,1} \Big( B(s) > - 2 \delta \, \, \forall \, \, s \in \big[ 0, \delta M^{-1} \big]  \Big\vert B(1) > 0 \Big) \nonumber \\
 & \leq &  2 \mathcal{W}_{1;0,*}^{0,1} \Big( B(s) > - 2 \delta \, \, \forall \, \, s \in \big[ 0, \delta M^{-1} \big]  \Big). \label{eqlembbtwo}
\end{eqnarray}
The first inequality follows because standard Brownian bridge is stochastically dominated by Brownian motion conditioned to have positive end-value. The second inequality follows from the fact that $\PP(B(1)>0)=1/2$. Note further that
\begin{eqnarray}
& & \mathcal{W}_{0,*}^{1;0,1} \Big( B(s) \leq - 2 \delta \, \, \textrm{ for some } \, \, s \in \big[ 0, \delta M^{-1} \big]  \Big) \nonumber \\
& =  & 2 \mathcal{W}_{1;0,*}^{0,1} \Big( B \big( \delta M^{-1} \big) \leq - 2 \delta \Big) \nonumber \\
 & = & 2 \mathbb{P} \Big( N \geq 2 \big( \delta M \big)^{1/2} \Big) \geq 1 - 2  \big(2/\pi\big)^{1/2} \big( \delta M \big)^{1/2}. \label{eqlembbthree}
\end{eqnarray}
The first equality depended on the reflection principle. In the second, $N$ under $\mathbb{P}$ has the law of a standard normal random variable; the inequality is due to the density of this distribution being uniformly bounded above by $(2\pi)^{-1/2}$.

Combining (\ref{eqlembbone}),  (\ref{eqlembbtwo}) and  (\ref{eqlembbthree}), we obtain
the statement of the lemma.
\end{proof}

\subsection{Airy line ensemble}\label{Airylineensemblesec}
We now define the {\it Dyson} and {\it edge-scaled Dyson} line ensembles which form the sequence of line ensembles whose limit we will consider.
\begin{figure}
\begin{center}
\includegraphics[width=.8\textwidth,height=.4\textwidth]{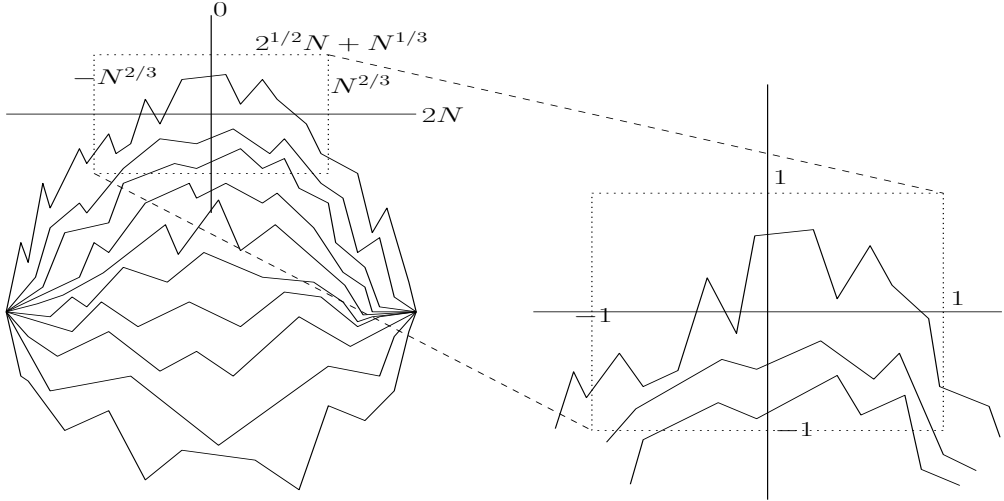}
\end{center}
\caption{An illustration in the left-hand sketch of $N$ Brownian bridges and the scaling window of width $N^{2/3}$ and height $N^{1/3}$. Scaling yields the edge-scaled Dyson line ensemble depicted on the right.}\label{BMfig}
\end{figure}

\begin{definition}\label{Dn}
For each $N\in \N$, define a Brownian bridge line ensemble $B^N:\{1,\ldots,N\}\times [-N,N]\rightarrow \R$ such that $B^N=\{B^N_1,\ldots,B^N_N\}$ is equal in law to the limit (as $\e$ goes to zero) of the law of $N$ Brownian bridges $\{\tilde{B}^N_{1},\ldots, \tilde{B}^N_{N}\}$ with $\tilde{B}^N_{i}(-N)=\tilde{B}^N_{i}(N)=-i\e$ and conditioned on $\tilde{B}^N_{i}(t)> \tilde{B}^N_{i+1}(t)$ for each $1 \leq i \leq N-1$ and $t \in \big[ -N,N \big]$. Clearly this line ensemble is continuous and non-intersecting, and furthermore it has the Brownian Gibbs property.\footnote{One way of seeing this is as follows: For a fixed $\delta$ observe that as the starting and ending points go to zero, the distributions of the height of the $N$ lines  at $\pm(N-\delta)$ converge to a non-trivial limit which can be explicitly calculated via the Karlin-McGregor formula \cite{KM}. The resulting ensemble on the interval $[-N+\delta,N-\delta]$ with this non-trivial entrance and exit law is continuous and non-intersecting and has the Brownian Gibbs property. As $\delta$ goes to zero this procedure yields a consistent family of measures which one identifies as the desired line ensemble with starting and ending height all identically zero.}

We define the $N$-th {\it edge-scaled Dyson line ensemble} $\mathcal{D}^N$ (on the same probability space as $B^N$) to be the ordered list $\big(\mathcal{D}^N_{1},\ldots,\mathcal{D}^N_N \big)$ where $\mathcal{D}^N_i : \big[ -N^{1/3},N^{1/3} \big] \to \R$ are rescaled Dyson lines given by
\begin{equation}\label{Ykn}
\mathcal{D}^N_i(t) = N^{-1/3} \Big(B_i\big(N^{2/3}t\big) - 2^{1/2} N \Big).
\end{equation}
We write $\PP^N$ for the probability measure associated with the $N$-th edge-scaled Dyson line ensemble.

We define $\mathcal{D}^N(k,A)$ for $k\geq 1$ and $A\subseteq[-N^{1/3},N^{1/3}]$ to be the set of top $k$ lines of $\mathcal{D}^N$ at times given by $A$. In practice, we will consider $A=[-T,T]$ or $A=\{t_1,\ldots, t_m\}$. The law $\mathcal{D}^{N}(N,A)$ represents the entire edge-scaled Dyson line ensemble at times given by $A$.
\end{definition}

One should note that the scaling in (\ref{Ykn}) involves centering at $2^{1/2}N$. This is due to the fact (cf. \cite{ADvM}) that for $\alpha\in [-1,1]$,
\begin{equation*}
\lim_{N\to \infty} \frac{B_1(\alpha N)}{N} = \sqrt{2(1-\alpha^2)}.
\end{equation*}

If we fix $A=\{t_1\}$ then $\mathcal{D}^{N}(N,A)$ is a determinantal point process consisting of $N$ distinct points. It is natural to
add a parabola from the edge-scaled Dyson line ensemble because the limiting shape for the edge has non-zero concavity. Call $\tilde{\mathcal{D}}^{N}_i(t) := 2^{1/2}\mathcal{D}^{N}_i(t) +t^2$ and similarly extend the above introduced notation. Though we will not make extensive use of this parabolically shifted line ensemble, it is presently convenient. The point process given by $\tilde{\mathcal{D}}^{N}(N,A)$ may be written in terms of a correlation kernel $K_N$; as $N$ goes to infinity, this kernel converges in the trace-class norm (see \cite{BS:book} for a definition) to a limiting kernel, known as the Airy kernel $K_{\Ai}$. The general theory of determinantal point processes (cf. \cite{Sosh}) implies that, for any fixed $k$, and for $A=\{t_1 \}$, the probability measure on points in $\tilde{\mathcal{D}}^{N}(k,A)$ converges to a limiting measure on $k$ distinct points. Likewise, for $A=\{t_1,\ldots, t_m\}$ and for any fixed $m\geq 1$, the joint probability measure on points in $\tilde{\mathcal{D}}^{N}(k,A)$ converges to a limiting measure on $k$ distinct points at times in $A$.

These measures are called the {\it finite-dimensional distributions of the multi-line Airy process}. As the finite set of times $A$ is augmented, and likewise as $k$ increases, these finite-dimensional distributions form a consistent family, so that Kolmogorov's consistency theorem implies the existence of a stochastic process with these marginal distributions. However, much as in the construction of Brownian motion from its finite-dimensional distributions, this implies neither continuity nor any other regularity properties of the process thus constructed. Johansson \cite{KJPNG} considered the top line (the case that $k=1$) and proved that there exists a {\it continuous version} of the  above stochastic process on any interval $A=[-T,T]$. This should be considered as analogous to proving that there exists a continuous version of the Brownian motion, which is initially only specified by its finite-dimensional distributions. In the process of proving our main result (that the Brownian Gibbs property for $\mathcal{D}^N$ passes over to the $N \to \infty$ limit), we will need to show also that there exists a version of the multi-line Airy process supported on continuous, non-intersecting curves for any interval $A=[-T,T]$. We will call this the {\it Airy line ensemble}.

The exact nature of the finite-dimensional distributions of the multi-line Airy process are, in fact, unnecessary for the statement and proof of our results.

\section{Main results}\label{mainresec}

Our first result, Theorem~\ref{mainthm}, proves the existence of a continuous, non-intersecting Airy line ensemble which, after subtracting a parabola and scaling down by a factor of $2^{1/2}$, has the Brownian Gibbs property and is the edge scaling limit for non-intersecting Brownian bridges. We then formulate some hypotheses on sequences of line ensembles under which we will prove existence of continuous, non-intersecting limiting line ensembles with Brownian Gibbs properties. Finally, we discuss some of the more general Airy-like line ensembles which satisfy these hypotheses. In Section \ref{absconsec} we turn to some of the consequences of our results.

\begin{theorem}\label{mainthm}
There exists a unique continuous non-intersecting $\N$-indexed line ensemble which has finite-dimensional distributions given by the multi-line Airy process. We call this the {\it $\N$-indexed Airy line ensemble} and denote it by $\mathcal{A}:\N\times \R\rightarrow \R$. Moreover, the $\N$-indexed line ensemble $\mathcal{L}:\N\times\R\rightarrow \R$ given $\mathcal{L}_i(x):=2^{-1/2}\big(\mathcal{A}_i(x)-x^2\big)$
for each $i \in \N$, has the Brownian Gibbs property.

Additionally, for any $k\geq 1$ and $T>0$ the line ensemble $\mathcal{D}^N(k,[-T,T])$ converges weakly (Definition \ref{maindef}) as $N \to \infty$ to the line ensemble given by $\mathcal{L}$ restricted to $\{1,\ldots,k\}\times [-T,T]$.
\end{theorem}

\begin{proof}
This follows immediately from the more general results of Section \ref{hypsec}. Proposition \ref{hypsat} shows that the edge-scaled Dyson line ensemble satisfies the hypotheses of Theorem \ref{mainthm2} which in turn proves the above result. The weak convergence result follows from Proposition \ref{proptightnessH} which is shown along the way to proving Theorem \ref{mainthm2}.
\end{proof}

\subsection{Some remarks about the proof of Theorem~\ref{mainthm}}
The above proof appeals to a more general set of results which we will soon present in Theorem \ref{mainthm2}. However, we now briefly explain the approach by which we will derive these results. We start by considering the edge-scaled Dyson line ensemble restricted to the top $k$ lines in an interval $[-T,T]$ (i.e., $\mathcal{D}^N(k,[-T,T])$). We will prove Theorem~\ref{mainthm} (and Theorem \ref{mainthm2}) by showing that, as $N$ tends to infinity, this sequence converges weakly as a line ensemble to a limiting line ensemble which is continuous, non-intersecting, has finite-dimensional distributions given by the multi-line Airy process minus a parabola and scaled down by $2^{1/2}$ (all of which is shown in Proposition \ref{proptightnessH}), and has the Brownian Gibbs property (shown in Proposition \ref{propBrownianGibbsH}). Once these facts are available, consistency of measures with respect to the interval $[-T,T]$ and with respect to $k$ yields Theorem~\ref{mainthm}.

The key to proving Propositions \ref{proptightnessH} and \ref{propBrownianGibbsH} is Proposition \ref{propacceptprob} which should be considered to be the main technical component of this article. This proposition shows that the $k$ lines on the interval $[-T,T]$ remain  sufficiently well-behaved uniformly in high $N$. This is measured in two ways.

The first measure is the {\it acceptance probability} (see Definition \ref{WBdef}). Roughly speaking, the acceptance probability for a line ensemble on an interval $[a,b]$ is the probability that the following operation is accepted: remove the top $k$ curves on the entire interval $[a,b]$, and redraw them with the same starting and ending points according to the law of independent Brownian bridges. The outcome is accepted if there is no point of intersection between the resampled curves or with the $(k+1)$-st curve. The second measure is the minimal gap between any two of the top $k$ lines on an interval $[a,b]$.

Proposition \ref{propacceptprob} shows that both the acceptance probability and the minimal gap stay bounded from below with high probability as $N \to \infty$, and that the top $k$ curves stay bounded between $\pm M$ for $M$ large enough, with a uniformly high probability. The first estimate shows that the top $k$ lines on $[-T,T]$ are uniformly equicontinuous in $N$ with high probability; alongside the known finite-dimensional convergence for $\mathcal{D}^N(k,\{t_1,\ldots, t_m\})$ (see Section \ref{airylikeSEC}), this yields the necessary tightness to prove the weak convergence statement of Proposition \ref{proptightnessH}.  The minimal gap estimate shows that the limiting measure is supported on non-intersecting lines.

That the limiting line ensemble has the Brownian Gibbs property is then shown by using the Skorohod representation theorem to couple the line ensembles for all $N$ so as to have uniform convergence of all $k$ lines on the interval $[-T,T]$. We reformulate the Brownian Gibbs property in terms of a resampling procedure and couple the Brownian bridges used for this resampling. Given these two layers of coupling, we show that the limiting line ensemble inherits the invariance in law under resampling from the finite $N$ ensemble and thus has the Brownian Gibbs property -- hence Proposition \ref{propBrownianGibbsH}.

We have not mentioned here how we will go about proving Proposition~\ref{propacceptprob}. The proof appears Section~\ref{technicalSec} which begins by discussing the ideas involved.

\subsection{Possible uniqueness results for Brownian Gibbs measures}\label{secunique}

Recall that a Gibbs measure is said to be extremal if it may not be written non-trivially as a sum of two Gibbs measures. The $\N$-indexed line ensemble $\mathcal{L}$ appearing in Theorem~\ref{mainthm} is an example of a Brownian Gibbs measure which presumably is extremal, as are
its affine shifts $\mathcal{L}^{(x,y)} = \mathcal{L}(x + \cdot) + y$ for $(x,y) \in \R^2$. Each element of the family $\big\{ \mathcal{L}^{(0,y)}: y \in \R \big\}$ also enjoys the property of being statistically invariant under horizontal shifts in the argument once the parabolic shift is removed. To the best of our knowledge, it was Scott Sheffield who first raised the possibility that, in fact, this family of measures exhausts the list of such extremal Gibbs measures:
\begin{conjecture}\label{conj}
We say that an $\N$-indexed line ensemble $\mathcal{A}$ is $x$-invariant if $\mathcal{A} \big( s + \cdot \big)$ is equal in distribution to $\mathcal{A}$ for each $s \in \R$. The set of extremal Brownian Gibbs $\N$-indexed line ensembles $\mathcal{L}$ which have the property that $\mathcal{A}$ (given by $\mathcal{A}_i(t) = 2^{1/2} \mathcal{L}_i(t) + t^2$ for $i \in \N$) is $x$-invariant is  $\big\{ \mathcal{L}^{(0,y)}: y \in \R \big\}$, where $\mathcal{L}$ appears in Theorem \ref{mainthm}.
\end{conjecture}
Beyond its intrinsic interest,  this conjecture may be worth investigating in light of its possible use as an invariance principle for deriving convergence of systems to the Airy line ensemble; indeed, we understand that Andrei Okounkov suggested problems in this direction in a talk in 2006. As such, the characterization could serve as a route to universality results. (See also \cite{BaikSuidan} for some partial universality results using a different approach involving the Koml{\'o}s-Major-Tusn{\'a}dy coupling of random walks with Brownian motion.)

\subsection{General hypotheses and results}\label{hypsec}
We now formulate some general hypotheses under which we will prove existence of continuous non-intersecting limiting line ensembles with the Brownian Gibbs property.

\begin{definition}\label{hypDef}
Fix $k\geq 1$ and $T>0$. Consider a sequence $\{k_i\}_{i\geq 1}$ and $\{T_i\}_{i\geq 1}$ such that there exists an $N_0$ such that for all $N \geq N_0$, $k_{N}\geq k+1$ and $T_N  \geq T +1$. A sequence of line ensembles $\{\mathcal{L}^N\}_{N=1}^{\infty}$, $\mathcal{L}^N:\{1,\ldots, k_N\}\times [-T_N,T_N]\rightarrow \R$ satisfies Hypothesis $(H)_{k,T}$ if it satisfies the following three hypotheses:
\begin{itemize}
\item $(H1)_{k,T}$: for each $N$, $\mathcal{L}^N$ is a non-intersecting line ensemble with the Brownian Gibbs property.
\item $(H2)_{k,T}$:
for every finite set $S \subseteq [-T,T]$, the joint distribution of $\big\{ \mathcal{L}_i^N(s): 1 \leq i \leq k , s \in S \big\}$ converges weakly.
\item $(H3)_{k,T}$: for all $\e > 0$ and all $t\in [-T,T]$ there exists $\delta>0$ and $N_1\geq N_0$ such that for all $N \geq N_1$,
\begin{equation*}
\PP^N \Big( \min_{1\leq i\leq k-1} \big\vert\mathcal{L}^N_i(t)-\mathcal{L}^N_{i+1}(t)\big\vert<\delta \Big) < \e.
\end{equation*}
\end{itemize}
If one requires the convergence in  $(H2)_{k,T}$ to be only for singletons $S=\{s\}$ we write $(H2')_{k,T}$ instead; in that case we will refer to all three hypotheses $(H1)_{k,T}$, $(H2')_{k,T}$ and $(H3)_{k,T}$ as $(H')_{k,T}$.
\end{definition}

\begin{definition}
Fix $k\geq 1$ and $T>0$. Consider a sequence of line ensembles $\{\mathcal{L}^N\}_{N=1}^{\infty}$ satisfying Hypothesis $(H')_{k,T}$. Define the minimal gap between the first $k$ lines on the interval $[a,b]\subseteq [-T,T]$ to be
\begin{equation*}
M_{k,a,b}^N = \min_{1 \leq i \leq k - 1} \min_{s \in [a,b]} \big\vert \mathcal{L}^N_i(s)  - \mathcal{L}^N_{i+1}(s)  \big\vert.
\end{equation*}
Recall also the acceptance probability $\AP(a,b,\bar{x},\bar{y},f)$ given in Definition \ref{WBdef}.
\end{definition}

We can now state the article's main technical component. Note that the reason we require hypotheses $(H)_{k,T+2}$ is that we need an extra buffer region around $[-T,T]$ in order to establish sufficient control over the lines in that interval.

\begin{proposition}\label{propacceptprob}
Fix $k\geq 1$ and $T>0$. Consider a sequence of line ensembles $\{\mathcal{L}^N\}_{N=1}^{\infty}$ satisfying Hypothesis $(H')_{k,T+2}$. There exists a stopping domain $\big( \mathfrak{l}^N,\mathfrak{r}^N \big)$
with $- T-2  \leq \mathfrak{l}^N \leq -T$ and $T \leq \mathfrak{r}^N \leq T+2$ (almost surely) such that the following holds.
For all $\epsilon > 0$, there exists $\delta = \delta(k,T)$ and $N_0 = N_0(k,T)$, such that, for $N \geq N_0$,
\begin{equation}\label{APpropeqn}
\PP^N \bigg(\AP\Big(\mathfrak{l}^N,\mathfrak{r}^N,\{\mathcal{L}^N_{i}\big(\mathfrak{l}^N\big)\}_{i=1}^{k},\{\mathcal{L}^N_{i}\big(\mathfrak{r}^N\big)\}_{i=1}^{k}, \mathcal{L}^N_{k+1}(\cdot)\Big) < \delta \bigg) < \epsilon,
\end{equation}
and
\begin{equation}\label{GAPpropeqn}
\PP^N \Big( M_{k,-T,T}^N  < \delta \Big) < \epsilon.
\end{equation}
Additionally we have that for all $\epsilon>0$ there exists $M>0$ such that, for each $N \in \N$,
\begin{equation}\label{maxmineqn}
\PP^N\Big(-M \leq \mathcal{L}^N_i(t) \leq M \textrm{ for all } t\in[-T-2,T+2], 1\leq i\leq k \Big) \geq 1-\epsilon.
\end{equation}
\end{proposition}

The proof of this proposition is fairly involved and will be given in Section \ref{proof:propacceptprob}.
We mention in passing that the use of the stopping domain  $\big( \mathfrak{l}^N,\mathfrak{r}^N \big)$
is a technical device needed for the proposition's proof. At least for line ensembles satisfying Hypothesis $(H')_{k,T+2}$, it is a consequence after the fact of the next proposition that~(\ref{APpropeqn}) also holds when $\mathfrak{l}^N$ and  $\mathfrak{r}^N$ are replaced by deterministic times.

Given Proposition~\ref{propacceptprob}, we can prove the following two propositions:

\begin{proposition}\label{proptightnessH}
Fix $k\geq 1$ and $T>0$. Consider a sequence of line ensembles $\{\mathcal{L}^N\}_{N=1}^{\infty}$ satisfying Hypothesis $(H)_{k,T+2}$. Then $\{\mathcal{L}^N_{i}\}_{i=1}^{k}$ converges weakly as $N \to \infty$  to a unique limit, which is the continuous non-intersecting line ensemble $\mathcal{L}^{\infty}:\{1,\ldots,k\}\times [-T,T]\rightarrow \R$ whose finite-dimensional distributions coincide with the limiting distributions ensured by Hypothesis $(H2)_{k,T+2}$.
\end{proposition}

\begin{proposition}\label{propBrownianGibbsH}
The line ensemble $\mathcal{L}^{\infty}$ specified in Proposition \ref{proptightnessH} has the Brownian Gibbs property.
\end{proposition}

Combining these propositions yields the following general theorem:

\begin{theorem}\label{mainthm2}
Consider a sequence of line ensembles $\{\mathcal{L}^N\}_{N=1}^{\infty}$ satisfying Hypothesis $(H)_{k,T}$ for all $k\geq 1$ and all $T>0$. Then there exists a unique continuous non-intersecting $\N$-indexed line ensemble $\mathcal{L}:\N\times \R\to \R$ which has finite-dimensional distributions given by those ensured by Hypothesis $(H2)_{k,T}$ and which has the Brownian Gibbs property.
\end{theorem}
\begin{proof}
We have proved the above result on any finite interval $[-T,T]$ for the top $k$ lines. We can conclude the existence and uniqueness of the infinite ensemble via consistency. The Brownian Gibbs property depends on only a finite number of lines and a finite interval of times, so that it extends to the full $\N$-indexed line ensemble.
\end{proof}

Accepting Proposition \ref{propacceptprob}, we may now prove Propositions \ref{proptightnessH} and \ref{propBrownianGibbsH}.

\begin{proof}[Proof of Proposition \ref{proptightnessH}]
Hypothesis $(H2)_{k,T+2}$ ensures convergence of finite-dimensional distributions. To prove weak convergence of line ensembles, it suffices (in light of Theorem 8.1 of \cite{Bill}) to  prove tightness. (We must then additionally prove the non-intersection property of the limiting ensemble.) The basic idea is to use the estimates of Proposition \ref{propacceptprob}. In particular, tightness relies on the fact that having a strictly positive acceptance probability implies that the actual lines in consideration look sufficiently like Brownian bridges that their moduli of continuity go to zero in probability. The proof must cope with the fact that our estimate for acceptance probability is not at the deterministic times $\pm T$ but rather in a random stopping domain $[\mathfrak{l}^N,\mathfrak{r}^N]$ with $-T-2\leq \mathfrak{l}^N \leq -T$ and $T\leq \mathfrak{r}^N \leq T+2$. This complication is resolved through a careful application of the strong Gibbs property which reduces the problem to a single claim which is then proved via a coupling argument.

The tightness criteria for $k$ continuous functions is the same as for a single function. For an interval $[a,b]$, define the $k$-line modulus of continuity
\begin{equation}\label{eqkline}
w_{a,b}(\{f_1,\ldots, f_k\},r) = \sup_{1\leq i\leq k} \sup_{\substack{s,t\in [a,b]\\ |s-t|<r}} \big\vert f_i(s)-f_i(t)\big\vert.
\end{equation}
Define the event
\begin{equation*}
W_{a,b}(\rho,r) = \big\{w_{a,b}(\{f_1,\ldots, f_k\},r) \leq \rho\big\}.
\end{equation*}
As an immediate generalization of Theorem 8.2 of \cite{Bill},
a sequence of probability measures $P_N$ on $k$ functions $f=\{f_1,\ldots f_k\}$ on the interval $[a,b]$ is tight if  the one-point (single $t$) distribution is tight and if for each positive $\rho$ and $\eta$ there exists a $r>0$ and an integer $N_0$ such that
$$P_N\big(W_{a,b}(\rho,r)\big)\geq 1-\eta, \qquad N \geq N_0.$$
Note that even when the line ensemble measures replacing $P_N$ are on more than $k$ lines, $W_{a,b}(\rho,r)$ will still refer to the top $k$ lines on the interval $[a,b]$.

Turning to our case, the one-point distribution is clearly tight because we have one-point convergence by means of Hypothesis  $(H2')_{k,T+2}$. We wish to prove further that for all $\rho$ and $\eta$ positive, we can choose $r>0$ small enough and $N_0$ large enough so that $\PP^N(W_{-T,T}(\rho,r))\geq 1-\eta$ for all $N \geq N_0$.

Observe that, as we are assuming Hypothesis $(H)_{k,T}$, we may apply Proposition \ref{propacceptprob}. This ensures the almost sure existence of a stopping domain $(\mathfrak{l}^N,\mathfrak{r}^N)$  which contains $[-T,T]$, is contained in $[-T-2,T+2]$ and is such that for all $\e>0$ there exists $\delta>0$ and $N_1$ such that for all $N \geq N_1$,
\begin{equation}\label{revapeqn}
\PP^N\Big(\AP(\mathfrak{l}^N,\mathfrak{r}^N) \geq  \delta \Big) \geq  1-\epsilon
\end{equation}
where we have abbreviated
\begin{equation*}
\AP(\mathfrak{l}^N,\mathfrak{r}^N)= \AP\Big(\mathfrak{l}^N,\mathfrak{r}^N,\{\mathcal{L}^N_{i}\big(\mathfrak{l}^N\big)\}_{i=1}^{k},\{\mathcal{L}^N_{i}\big(\mathfrak{r}^N\big)\}_{i=1}^{k}, \mathcal{L}^N_{k+1}(\cdot)\Big).
\end{equation*}
Observe that
\begin{equation*}
\PP^N\big(W_{-T,T}(\rho,r)\big) \geq \PP^N \Big(W_{-T,T}(\rho,r),  \AP(\mathfrak{l}^N,\mathfrak{r}^N)\geq \delta, S_{k,M} \Big)
\end{equation*}
where
\begin{equation*}
S_{k,M}= \Big\{\mathcal{L}^N_1(\{\mathfrak{l}^{N},\mathfrak{r}^{N}\}) \leq M,\, \mathcal{L}^N_k(\{\mathfrak{l}^{N},\mathfrak{r}^{N}\}) \geq -M\Big\}.
\end{equation*}
We claim that for any $\rho,\eta>0$ there exists $r>0$, $\delta\in (0,1)$, $M>0$ and $N_0$ such that, for $N \geq N_0$,
\begin{equation}\label{lhsabove}
\PP^N\Big(W_{-T,T}(\rho,r) , \AP(\mathfrak{l}^N,\mathfrak{r}^N)\geq \delta, S_{k,M} \Big) > 1-\eta.
\end{equation}

Note that both $\{\AP(\mathfrak{l}^N,\mathfrak{r}^N)\geq \delta\}$ and $S_{k,M}$ are measurable with respect to $\mathcal{F}_{ext}(k,\mathfrak{l}^N,\mathfrak{r}^N)$. With ${\bf 1}$ denoting the indicator function, we can thus write the left-hand side of (\ref{lhsabove}) as
\begin{equation}\label{condexp}
\EE\Big[{\bf 1}\big\{\AP(\mathfrak{l}^N,\mathfrak{r}^N)\geq \delta, S_{k,M}\big\}\,\EE\big[{\bf 1}\big\{W_{-T,T}(\rho,r)\big\} \big\vert \mathcal{F}_{ext}(k,\mathfrak{l}^N,\mathfrak{r}^N)\big] \Big],
\end{equation}
where the sigma-field $\mathcal{F}_{ext}(k,\mathfrak{l}^N,\mathfrak{r}^N)$ is defined in (\ref{defstopdom}).

Observe that by applying the strong Gibbs property given in Lemma \ref{stronggibbslemma}, $\PP$ almost surely
\begin{equation}\label{strongapp}
\EE\big[{\bf 1}\big\{ W_{-T,T}(\rho,r)\big\} \big\vert \mathcal{F}_{ext}(k,\mathfrak{l}^N,\mathfrak{r}^N)\big] = \mathcal{B}_{\bar{x},\bar{y},f}^{\mathfrak{l}^N,\mathfrak{r}^N}\big[W_{-T,T}(\rho,r) \big]
\end{equation}
where, for $1\leq i \leq k$, $x_i= \mathcal{L}^N_i(\mathfrak{l}^{N})$, $y_i= \mathcal{L}^N_i(\mathfrak{r}^{N})$, $f(\cdot)=\mathcal{L}^N_{k+1}(\cdot)$ and where $W_{-T,T}$ is a function of $\mathcal{L}_1,\ldots, \mathcal{L}_k$ in the left-hand side and $B_1,\ldots, B_k$ in the right-hand side (i.e., these correspond to the choice of the functions $f_1,\ldots, f_k$ in the definition of $W_{-T,T}$).

\noindent {\bf Claim:} let $\rho, \eta > 0$, $\delta\in (0,1)$ and $M>0$. There exists $r>0$  such that, for all $a\in [-T-2,-T]$, $b\in [T,T+2]$, and $\bar{x},\bar{y} \in \R^k_>$, $f:[a,b] \to \R$ satisfying $x_1,y_1\leq M$, $x_k,y_k\geq -M$ and $\AP(a,b,\bar{x},\bar{y},f)~\geq~\delta$,
\begin{equation*}
\bxyf^{a,b}\big[W_{-T,T}(\rho,r)\big] \geq 1-\eta/2.
\end{equation*}
As noted before, $W_{-T,T}(\rho,r)$ is an event depending on the modulus of continuity of the random $k$ lines between $[-T,T]$ specified by the measure $\wxy^{a,b}$ (recall $\wxy^{a,b}$ and $\ewxy^{a,b}$ are given in Definition \ref{WBdef} and denote the measure and expectation of $k$ Brownian bridges on $[a,b]$ with starting values $\bar{x}$ and ending values $\bar{y}$).

Let us assume the claim for the moment. By choosing $r$ small enough (depending on $\rho,\eta,\delta,M$), and using (\ref{strongapp}) we may bound (\ref{condexp}) as at least
\begin{equation}\label{claimapp}
(1-\eta/2) \EE\Big[{\bf 1}(\AP(\mathfrak{l}^N,\mathfrak{r}^N)\geq \delta, S_{k,M})\Big] = (1-\eta/2)\PP^N(\AP(\mathfrak{l}^N,\mathfrak{r}^N)\geq \delta, S_{k,M}).
\end{equation}
Since we may apply Proposition \ref{propacceptprob} it follows from (\ref{revapeqn}) and (\ref{maxmineqn}) that by choosing $M$ large enough and $\delta$ small enough,
$$\PP^N\big(\AP(\mathfrak{l}^N,\mathfrak{r}^N)\geq \delta, S_{k,M}\big) \geq 1-\eta/2.$$
For these values of $\delta,M$, by taking $r$ to be small enough so that the above bound for (\ref{condexp}) holds, we can bound (\ref{lhsabove}) by $(1-\eta/2)^2 \geq 1-\eta$ which is exactly as desired to prove tightness.

Note that the tightness implies almost sure continuity of the limiting line ensemble. That said, it does not provide the non-intersection condition. However, this follows immediately from (\ref{GAPpropeqn}) of Proposition \ref{propacceptprob}.

We now finish by proving the claim. First note that if we replace $W_{-T,T}(\rho,r)$ by $W_{a,b}(\rho,r)$ in the statement of the claim, then by containment of events it follows that proving this modified version of the claim implies the stated claim. We will make such a replacement and prove the modified version.

Consider $\{\tilde{B}_i\}_{i=1}^k$ distributed as independent Brownian bridges on $[0,1]$ such that $\tilde{B}_i(0)=0$ and $\tilde{B}_i(1)=0$. We will use these as the basis for a coupling proof of the claim. For each $\tilde{r}$ associate the random  modulus of continuity $w_{0,1}(\tilde{B}_i,\tilde{r})$. The process $\tilde{B}_i$ having bounded sample paths, for each $\tilde{r}$ this random variable is supported on $[0,\infty)$.
Observe that from the standard Brownian bridges we can construct the Brownian bridges $B_i$ on $[a,b]$ of $\wxy^{a,b}$  by setting
\begin{equation*}
B_{i}(t)= ( b-a)^{1/2} \tilde{B}_i\left(\frac{t-a}{b-a}\right) +  \left(\frac{b-t}{b-a}\right) x_i +  \left(\frac{t-a}{b-a}\right) y_i.
\end{equation*}
The $k$ line modulus of continuity $w_{a,b}(\{B_1,\ldots, B_k\},(b-a)\tilde{r})$ may then be bounded by
\begin{equation}\label{modcontr}
w_{a,b}\Big(\{B_1,\ldots, B_k\},(b-a)\tilde{r}\Big) \leq \sup_{1\leq i\leq k} \left((b-a)^{1/2} w_{0,1}(\tilde{B}_i,\tilde{r}) + |x_i-y_i| \tilde{r} \right).
\end{equation}
By assumption $|x_i-y_i|\leq 2M$. Since $T>0$ and we have assumed $a\in [-T-2,-T]$ and $b\in [T,T+2]$ it follows that $b-a\in (c,c^{-1})$ for some constant $c = c_T >0$. Set $\tilde{r} = rc$. By using (\ref{modcontr}) and the fact that the modulus of continuity for $r'<r$ is bounded above by the modulus of continuity for $r$ it follows that
\begin{equation}\label{allabeqn}
w_{a,b}\Big(\{B_1,\ldots, B_k\},r\Big) \leq  \sup_{1\leq i\leq k} \left(c^{-1/2} w_{0,1}(\tilde{B}_i,rc)\right) + 2Mrc.
\end{equation}
Now observe that conditioning $\wxy^{a,b}$ on any event $E$ such that $\wxy^{a,b}(E)>\delta$ is equivalent to conditioning the measure of the $k$ Brownian bridges $\{\tilde{B}_i\}_{i=1}^k$ on some other event $\tilde{E}$ also of measure at least $\delta$. The random variables $w_{0,1}(\tilde{B}_{i},rc)$ are supported on $[0,\infty)$ and converge to zero as $r$ goes to zero (since the Brownian bridges are continuous almost surely), so that, by choosing $r$ small enough (with all of the other variables fixed) we can be assured that, conditioned on the event $\tilde{E}$,
\begin{equation*}
\sup_{1\leq i\leq k} \left(c^{-1/2} w_{0,1}(\tilde{B}_i,rc)\right) +2Mrc \leq \rho
\end{equation*}
with probability at least $1-\eta/2$. Since $\wxy^{a,b}(\ncf_{a,b})=\AP(a,b,x,y,f)\geq \delta$ the above reasoning applies presently since we are conditioning on $\ncf_{a,b}$. By the uniformity over $a,b$ in the specified ranges, the claim follows.
\end{proof}

\begin{proof}[Proof of Proposition \ref{propBrownianGibbsH}]
By Proposition \ref{proptightnessH}, we know that the limit of the probability measure on $\mathcal{L}^N$ is supported on the space of $k$ continuous, non-intersecting curves on $[-T,T]$. It follows (since that space is separable) that we may apply the Skorohod representation theorem (see \cite{Bill} for instance) to conclude that there exists a probability space such that all of the $\{\mathcal{L}_j^N\}_{j=1}^{k}$ as well as a limiting ensemble $\{\mathcal{L}_j^{\infty}\}_{j=1}^{k}$ are defined on the space with the correct marginals, and with the property that for every $\omega\in \Omega$ and $j\in \{1,\ldots, k\}$, $\mathcal{L}_j^N(\omega)\rightarrow \mathcal{L}_j^{\infty}(\omega)$ in the uniform topology.

We will presently show that for any fixed line index $i$ and any two times $a,b\in [-T,T]$ with $a<b$, the law of $\{\mathcal{L}_j^{\infty}\}_{j=1}^{k}$ is unchanged if $\mathcal{L}^{\infty}_i$ is resampled between times $a$ and $b$ according to a Brownian bridge conditioned to avoid the $(i-1)$-st and $(i+1)$-st lines. This property is equivalent to the Brownian Gibbs property. The argument used for one line clearly works for several lines. Also, we assume (for simplicity) that $i\not\in \big\{ 1 , k \big\}$ (i.e.,  the curve in question is neither the highest nor the lowest line).

We show this by coupling the resampling procedure for all values of $N$. In order to do this, we perform the resampling in two steps. We first choose a Brownian bridge and paste it in to agree with the starting and ending points; then we check if it intersects the lines above and below and, if it does not, we accept it. In order to couple this procedure over various values of $N$ we fix a single collection of sampling Brownian bridges. Towards this end, fix a sequence of Brownian bridges $B_\ell:[a,b]\rightarrow \R$ such that $B_\ell(a)=B_\ell(b)=0$; our probability space may be augmented to accommodate these independently of existing data. Define the $\ell$-th resampling of line $i$ to be
$$\mathcal{L}_i^{N,\ell}(t)= B_\ell(t) + \frac{b-t}{b-a} \mathcal{L}_i^N(a) + \frac{t-a}{b-a} \mathcal{L}_i^N(b)$$
where the purpose of the last two terms on the right is to add the necessary affine shift to the Brownian bridge to make sure that $\mathcal{L}_i^{N,\ell}(a)=\mathcal{L}_i^{N}(a)$ and $\mathcal{L}_i^{N,\ell}(b)=\mathcal{L}_i^{N}(b)$. Now define
\begin{equation}\label{eqavoid}
\ell(N) = \min\left\{\ell \in \N: \mathcal{L}_{i-1}^N(t) > \mathcal{L}_i^{N,\ell}(t) > \mathcal{L}_{i+1}^N(t)
\,\, \forall \, t \in [a,b] \right\},
\end{equation}
the index of the first accepted (non-intersecting) Brownian bridge resampling. Write $\{\mathcal{L}_j^{N,\rm{re}}\}_{j=1}^{k}$ for the line ensemble with the $i^{\rm{th}}$ line replaced by $\mathcal{L}_{i}^{N,\ell(N)}$. (Here, $re$ stands for resampled.) Given this notation, another way of stating the Brownian Gibbs property for $\mathcal{L}^N$ is that
\begin{equation}\label{restarstar}
\{\mathcal{L}_j^N\}_{j=1}^{k} \stackrel{(\rm{law})}{=} \{\mathcal{L}_j^{N,\rm{re}}\}_{j=1}^{k}.
\end{equation}

We wish to show that this same property holds for the limiting line ensemble $\{\mathcal{L}_j^{\infty}\}_{j=1}^{k}$, because that will imply the Brownian Gibbs property. Similarly to~(\ref{eqavoid}), we set $\ell(\infty) = \min\big\{\ell \in \N:
\mathcal{L}_{i-1}^\infty(t) > \mathcal{L}_i^{\infty,\ell}(t) > \mathcal{L}_{i+1}^\infty(t)
\,\, \forall \, t \in [a,b] \big\}$. If we can show that almost surely
$$\lim_{N\rightarrow \infty} \ell(N) = \ell(\infty),$$
then we are done. This is because we know both that $\{\mathcal{L}_j^N\}_{j=1}^{k}$ converges to $\{\mathcal{L}_j^\infty\}_{j=1}^{k}$ and that $\{\mathcal{L}_j^{N,\rm{re}}\}_{j=1}^{k}$ converges to $\{\mathcal{L}_j^{\infty,\rm{re}}\}_{j=1}^{k}$ (all in the uniform topology). Since by (\ref{restarstar}) the laws of $\{\mathcal{L}_j^N\}_{j=1}^{k}$ and $\{\mathcal{L}_j^{N,\rm{re}}\}_{j=1}^{k}$ coincide, then so must the laws of their almost sure uniform limits. This proves that the laws of $\{\mathcal{L}_j^\infty\}_{j=1}^{k}$ and $\{\mathcal{L}_j^{\infty,\rm{re}}\}_{j=1}^{k}$ coincide and hence proves the Brownian Gibbs property for $\{\mathcal{L}_j^\infty\}_{j=1}^{k}$.

Thus, it remains to prove that $\ell(N)$ converges almost surely to $\ell(\infty)$. We prove two lemmas which together lead to the desired conclusion.
\begin{lemma}
The sequence $\big\{ \ell(N): N \in \N \big\}$ is bounded almost surely.
\end{lemma}
\begin{proof}
By Proposition \ref{proptightnessH}, we know that $\{\mathcal{L}_j^\infty\}_{j=1}^{k}$ is composed of non-intersecting continuous functions. Therefore, $\delta : = \inf_{t \in [a,b], j \in \{ i,i+1\}} \big\vert \mathcal{L}_j^\infty -  \mathcal{L}_{j-1}^\infty \big\vert$ is almost surely positive.
Uniform convergence implies that there exists an $N_0$ such that, for all $N \geq N_0$, the following conditions hold:
$$\big\vert \mathcal{L}_i^N(a)-\mathcal{L}_i^\infty(a)\big\vert <\delta/4, \qquad \big\vert \mathcal{L}_i^N(b)-\mathcal{L}_i^\infty(b)\big\vert<\delta/4, \qquad \sup_{t\in [a,b]} \big\vert\mathcal{L}_{i\pm 1}^N(t)- \mathcal{L}_{i\pm 1}^\infty(t)\big\vert<\delta/4.$$

Therefore, for $N \geq N_0$, there exists a continuous curve $f$ from $\mathcal{L}_i^\infty(a)$ to $\mathcal{L}_i^\infty(b)$ which has a neighborhood of radius $\delta/2$ which does not intersect any curve $\mathcal{L}_j^{N}$ for $j\neq i$ and $N>N_0$ or $N=\infty$.
A basic property of Wiener measure (recorded in the present paper as Corollary~\ref{coropen}) is that any $L^\infty$-norm neighbourhood of a continuous function has positive Wiener measure. Thus  there is an almost surely  finite $\ell$ such that $\mathcal{L}_i^{N,\ell}$ stays in a neighborhood of $\delta/4$ about the function $f$. By construction, for $N \geq N_0$, $\mathcal{L}_i^{N,\ell}$ will not intersect $\mathcal{L}_{i\pm 1}^{N}$. This shows that almost surely $\ell(N)$ stays bounded as $N \to \infty$.
\end{proof}

\begin{lemma}\label{lemuniquel}
There exists a unique limit point for $\{\ell(N)\}_{N=1}^{\infty}$.
\end{lemma}
\begin{proof}
Clearly, since $\big\{ \ell(N): N \in \N \big\}$ is bounded, it has some limit points. Let $\ell^*$ be the minimum of these (finitely many) limit points. To see uniqueness observe that the only way that there could be an infinite sequence of values of $N$ for which $\ell(N)>\ell^*$  would be if $\mathcal{L}_i^{\infty,\ell^*}$ touches, but never crosses, one of its neighboring lines $\mathcal{L}_{i\pm 1}^{\infty}$. We claim that this event has probability zero. A basic property of Wiener measure (recorded in the present paper as Corollary~\ref{cor:touch}) is that a Brownian bridge and an independent almost surely continuous function which start and end a positive distance apart will almost surely either cross or never touch -- and hence will have probability zero of touching without crossing. Observe then that, up to affine shift, $\mathcal{L}_i^{\infty,\ell^*}$ is just a Brownian bridge that is independent of $\mathcal{L}_{i\pm 1}^{\infty}$ and, moreover, that it starts and ends a positive distance away from its neighboring lines. Thus the event of touching without crossing has probability zero, and so the lemma is proved.
\end{proof}

Having established that $\ell^\infty : = \lim_{N\rightarrow\infty}\ell(N)$ exists, all that remains is to prove that $\ell^\infty=\ell(\infty)$. Clearly, $\ell^\infty\geq \ell(\infty)$, because the only way that $\mathcal{L}_{i}^{N,\ell^\infty}$ could intersect $\mathcal{L}_{i\pm 1}^{\infty}$ is if it touches this line but does not cross it. However, as we saw in the proof of Lemma \ref{lemuniquel}, this event happens with probability zero. Finally, we argue that $\ell^\infty\leq \ell(\infty)$ by establishing a contradiction otherwise. Assume for the moment that there is an $\ell<\ell^\infty$ for which $\mathcal{L}_{i}^{\infty,\ell}$ does intersect $\mathcal{L}_{i\pm 1}^{\infty}$.  Then there is a positive distance between these lines. This  implies that, for some sufficiently large $N_0$, $\mathcal{L}_{i}^{N,\ell}$ intersects $\mathcal{L}_{i\pm 1}^{N}$ for all $N \geq N_0$. Hence, $\lim_{N\rightarrow\infty}\ell(N)\leq \ell<\ell^\infty$, in contradiction to the definition of $\ell^\infty  = \lim_{N\rightarrow\infty}\ell(N)$. Hence, we find that $\ell^\infty = \ell(\infty)$ and so complete the proof of the proposition.
\end{proof}

\subsection{Line ensembles satisfying the general hypotheses}\label{airylikeSEC}
We now introduce a two-parameter family of Brownian bridge line ensembles which satisfy the hypotheses laid out in Definition \ref{hypDef}. We will mention shortly some of the contexts in which line ensembles obtained as scaling limits of this family arise. 

\begin{definition}
Fix $m_1\geq 0$ and $m_2\geq 0$ as well as two vectors $p=\{p_1,\ldots,p_{m_1}\}$ and $q=\{q_1,\ldots, q_{m_2}\}$ such that $p_i\geq p_j$ and $q_i\geq q_j$ for all $i<j$. For each $N>\max(m_1,m_2)$ define the $(p,q)$-perturbed Brownian bridge line ensemble $B^{N;p,q}:\{1,\ldots,N\}\times [-N,N]\to \R$ such that $B^{N;p,q}_i$ are distributed as Brownian bridges conditioned not to intersect with starting points
\begin{equation*}
B^{N;p,q}_i(-N) = \begin{cases}
2^{1/2}N+N^{2/3}p_i & \textrm{for } 1\leq i\leq m_1\\
 0 & \textrm{otherwise},
\end{cases}
\end{equation*}
and ending points
\begin{equation*}
B^{N;p,q}_i(N) = \begin{cases}
2^{1/2}N+N^{2/3}q_i & \textrm{for } 1\leq i\leq m_2\\
 0 & \textrm{otherwise}.
\end{cases}
\end{equation*}
As before when starting points coincide this is understood as the weak limit of the path measures as the starting points converge together. Clearly these line ensembles are continuous, non-intersecting and have the Brownian Gibbs property.

We define the $N$-th $(p,q)$-perturbed edge-scaled Dyson line ensemble $\mathcal{D}^{N;p,q}$ on this probability space to be the ordered list $(\mathcal{D}^{N;p,q}_1,\ldots, \mathcal{D}^{N;p,q}_N)$, where $\mathcal{D}^{N;p,q}_i:[-N^{1/3}, N^{1/3}]\to \R$ is given by
\begin{equation*}
\mathcal{D}^{N;p,q}_i (t) = N^{-1/3}\left(B^{N;p,q}_i(N^{2/3}t) - 2^{1/2}N\right),
\end{equation*}
and, as in the unperturbed case, define the scaled and parabolically shifted line ensemble
$$\tilde{\mathcal{D}}^{N;p,q}_i(t) = 2^{1/2}\mathcal{D}^{N;p,q}_i(t) +t^2.$$
\end{definition}

It is clear that, in the case where $m_1=m_2=0$,  $\mathcal{D}^{N;p,q}$ is simply the edge-scaled Dyson line ensemble given in Definition \ref{Dn}. When only one of $m_1$ or $m_2$ is non-zero, the resulting ensembles have been studied in \cite{ADvM} under the name of Dyson's non-intersecting Brownian motions with {\it a few outliers}. When both $m_1=m_2>0$, these  {\it wanderer} ensembles have been studied in \cite{AFvM}. In the latter case, \cite{AFvM} requires the restriction on the values of $p$ and $q$
that the straight lines connecting $(-N,p_i)$ to $(N,q_i)$ intersect the $y$ axis at points of the form $(0,r_i)$ where $r_i\leq 2N$.
We say that the line ensembles $\tilde{\mathcal{D}}^{N;p,q}$ are {\it Airy-like}. From these results, we find the following.

\begin{proposition}\label{hypsat}
Consider $m_1,m_2\geq 0$ and $p,q$ (depending on $N$) such that either $m_1=0$ or $m_2=0$, or $m_1=m_2>0$ and the values of $p$ and $q$ are such that the straight lines connecting $(-N,p_i)$ to $(N,q_i)$ intersect points $(0,r_i)$ where $r_i\leq 2N$.
Then, for all $k\geq 1$ and $T>0$, the sequence of line ensembles $\mathcal{D}^{N;p,q}$ satisfy Hypotheses $(H)_{k,T}$.
\end{proposition}

\begin{proof}
Hypothesis $(H1)_{k,T}$ follows immediately from the definition of the line ensembles. The other two hypotheses utilize the determinantal nature of the finite dimensional marginals of the line ensembles, as developed in \cite{ADvM,AFvM}. In \cite[Proposition 3.3]{AFvM},  it is proved that for any set of times $A=\{t_1,\ldots, t_k\}\subseteq [-N^{1/3},N^{1/3}]$, the $A$-indexed collection of point processes
$$
\mathcal{D}^{N;p,q}(A) := \bigcup_{t\in A} \Big\{ \{\mathcal{D}_i^{N;p,q}(t)\}_{i=1}^{N}\Big\}
$$
are determinantal. Likewise is true for $\tilde{\mathcal{D}}^{N;p,q}(A)$ and since the two differ by a deterministic shift and scaling, all conclusions below (which are stated with regards to the second ensemble, holds true for the first as well). What this means is that all of the correlation functions\footnote{The $r^{th}$ correlation function is given essentially by the probability of finding points in small neighborhoods of $(s_i,x_i)$ for $s_i\in A$, $x_i\in \R$ and $i=1,\ldots, r$.} for $\tilde{\mathcal{D}}^{N;p,q}(A)$ can be written in terms of determinants of a single symmetric kernel $K^{N;p,q}:(A\times \R)^2 \to \R$. In the proof of \cite[Theorem 1.2]{AFvM}, it is shown that the kernel $K^{N;p,q}$ converges to a perturbation of the Airy kernel $K^{\rm{Airy};p,q}$. By \cite[Theorem 5]{Sosh}, this implies that the point processes $\tilde{\mathcal{D}}^{N;p,q}(A)$ converge weakly on cylinder sets to a point process $\tilde{\mathcal{D}}^{\rm{Airy};p,q}(A)$ which is still determinantal.

From the determinantal structure one also shows that $\{\tilde{\mathcal{D}}_1^{N;p,q}(t)\}_{t\in A}$ has a limiting joint distribution described via a Fredholm determinant involving the kernel $K^{N;p,q}$. This means that $\tilde{\mathcal{D}}^{\rm{Airy};p,q}(t)$ has a top-most particle for each $t\in A$. It is easy to check (using \cite[Theorem 4.a]{Sosh}) that $\tilde{\mathcal{D}}^{\rm{Airy};p,q}(t)$ has an infinite number of particles. This implies in fact that $\big\{\tilde{\mathcal{D}}_i^{N;p,q}(t): 1\leq i \leq k, t\in A \big\}$ has a limit in distribution as $N$ goes to infinity, which shows that hypothesis $(H2)_{k,T}$ holds.

In Theorem 4.d of \cite{Sosh} it is shown that in a determinantal point process, the probability of two particles coinciding is 0. Since \cite{AFvM} proved that $\tilde{\mathcal{D}}^{\rm{Airy};p,q}(A)$ is such a point process and since $\{\tilde{\mathcal{D}}_i^{N;p,q}(t)\}_{i=1}^{k}$ converges weakly to $\{\tilde{\mathcal{D}}_i^{\rm{Airy};p,q}(t)\}_{i=1}^{k}$, it follows that for all $\e>0$ and $t\in [-T,T]$, there exists $\delta>0$ and $N_1>0$ such that for all $N>N_1$,
$$\PP^N\big[\min_{1\leq i\leq k-1} |\tilde{\mathcal{D}}_i^{N;p,q}(t)-\tilde{\mathcal{D}}_{i+1}^{N;p,q}(t)|<\delta\big] <\e.$$
This proves hypothesis $(H3)_{k,T}$.
\end{proof}


\begin{rem}
The ensemble of non-intersecting Brownian exclusions \cite{TWex} also satisfies $(H)_{k,T}$ Theorem \ref{mainthm2} therefore applies and shows, in particular, uniform convergence of the edge scaled line ensemble to the Airy line ensemble.

There exists a variety of other finite line ensembles which do not display the Brownian Gibbs property, yet which (under appropriate scaling) converge in finite-dimensional distribution to Airy-like line ensembles. These include the line ensembles associated with the polynuclear growth (PNG) model, last passage percolation (LPP) and tiling problems. The finite line ensembles do, however, display Gibbs properties with respect to different underlying path measures. Via an invariance principle, these Gibbs properties converge, under the Airy rescaling, to the Brownian Gibbs property. In fact, if so inclined, one could carry out the program orchestrated in this paper in these non-Brownian settings. However, since the limiting objects (the Airy-like line ensembles) are the same as the ones that we consider, the only benefit of doing so would be to prove tightness of the finite PNG, LPP or tiling line ensembles.

Similarly to the Brownian bridge case, there are simple ways to perturb the PNG and LPP models (for instance by modifying boundary conditions or external sources) which result in perturbations to the line ensembles and hence to their limits \cite{IS,BP,BFP,CFP1}. The above-mentioned modifications to our program provide a means to prove that the limiting processes associated with these perturbations can actually be realized as the marginals of continuous, non-intersecting line ensembles with the Brownian Gibbs property. In fact, a number of these limiting processes coincide with the finite-dimensional distributions of the Airy-like ensembles for which Proposition \ref{hypsat} applies -- hence by Theorem \ref{mainthm2} the existence and Brownian Gibbs property for these limiting line ensembles is immediate.

\end{rem}

\section{Consequences of main results}\label{absconsec}

Theorem \ref{mainthm} (and its more general counterpart Theorem \ref{mainthm2}) has several interesting consequences which are stated here.

\subsection{Local Brownian absolute continuity}\label{subsubabs}
The next result holds for any line ensemble with the Brownian Gibbs property, although for concreteness we state it for the Airy line ensemble.

\begin{proposition}[Local absolute continuity]\label{propabscon}
For $k\in \N$, let $\mathcal{A}_k(\cdot) := \mathcal{A}(k,\cdot)$ denote the $k^{\rm{th}}$ line of the Airy line ensemble. For any $s,t \in \R$, $t>0$, the measure on functions from $[0,t]\to \R$ given by  $$\mathcal{A}_k(\cdot  + s) - \mathcal{A}_k(s)$$ is absolutely continuous with respect to Brownian motion (of diffusion parameter 2) on $[0,t]$. Alternatively, the measure on functions from $[0,t]\to \R$ given by $$\mathcal{A}_k(\cdot  + s) - \left(\frac{t-\cdot}{t} \mathcal{A}_k(s) + \frac{\cdot}{t} \mathcal{A}_k(s+t)\right)$$ is absolutely continuous with respect to Brownian bridge (of diffusion parameter 2) on $[0,t]$.
\end{proposition}

\begin{proof}
The absolute continuity with respect to Brownian bridge (of diffusion parameter 2) follows immediately from the Gibbs property. The reason for the parameter 2 is due to the $2^{1/2}$ factor which arises in relating the Airy line ensemble to the ensemble $\mathcal{L}$ in Theorem \ref{mainthm}. When the right endpoint is not fixed, then the absolute continuity is with respect to Brownian motion (of diffusion parameter 2). To show this, we need only show that the distribution of $\mathcal{A}_k(s+t)$ is absolutely continuous with respect to Lebesgue measure. This can be seen, however, by applying the above shown Brownian bridge absolute continuity result for a slightly larger interval $[s,s+t+\e]$ for some $\e>0$. The Brownian bridge absolute continuity on this larger interval then implies that $\mathcal{A}_k(s+t)$ has a density with respect to the Lebesgue measure.
\end{proof}

Proposition \ref{propabscon} was anticipated in the paper of Pr\"{a}hofer and Spohn \cite{PS} (page 23) where the multi-line Airy process was introduced in terms of finite-dimensional distributions. The focus of that paper was not on studying line ensembles in their own right, but rather on the relationship between these ensembles and particular solvable growth models. Shortly thereafter, Johansson \cite{KJPNG} studied a closely related growth model and strengthened the convergence results of \cite{PS} in that setting by showing weak convergence of the growth interface to the top line of the multi-line Airy process (called just the Airy process).

Pr\"{a}hofer and Spohn's prediction that locally the Airy process looks Brownian was based on the observation that short time Airy process increments had nearly linear variance. More evidence for this claim was provided by \cite{Hagg} which showed that the finite-dimensional distributions of the Airy process converge, as the time differences go to zero and the process is scaled diffusively, to those of Brownian motion. An immediate corollary of our Proposition \ref{propabscon} is a stronger version of these results.

\begin{corollary}[Functional central limit theorem for local Airy]
Let $T > 0$ and $k\geq 1$. Define $B_\e: [-T,T] \to \R$ as the process $t\mapsto \epsilon^{-1/2} \big( \mathcal{A}_k(\epsilon t) - \mathcal{A}_k(0) \big)$. Then, as $\e$ goes to zero,  $B_{\e}$ converges weakly (with respect to the uniform topology on $C([-T,T],\R)$) to two-sided Brownian motion (of diffusion parameter 2) $B:[-T,T] \to \R$, $B(0) = 0$.
\end{corollary}

\subsection{Uniqueness of the maximizing location of the Airy line ensemble minus a parabola}\label{uniquenesssec}

Johansson conjectured that there exists a unique maximization point for the Airy process minus a parabola and explained how such a result would lead to a characterization of the transversal fluctuations of directed polymers. As a corollary of our main theorem, we prove that conjecture.

\begin{theorem}[Johansson \cite{KJPNG}, Conjecture 1.5]\label{propjoh}
Let $H(t)=\mathcal{A}_1(t)-t^2$. Then, for each $T>0$, $H(t)$ attains its maximum at a unique point in $[-T,T]$, almost surely. Likewise, $H(t)$ attains its maximum at a unique point on all of $\R$, almost surely.
\end{theorem}

\begin{proof}
The first statement follows immediately from three facts: (1) $\mathcal{A}'(t):=\mathcal{A}_1(t)-\mathcal{A}_1(-T)$ is absolutely continuous with respect to Brownian motion (of diffusion parameter 2) $B(\cdot)$ on $[-T,T]$ (starting at zero at time $-T$);  (2) the law of $B(t) - t^2$ on $[-T,T]$ is absolutely continuous with respect to the law of $B(t) - T^2$ on $[-T,T]$; (3) Brownian motion attains its maximum at a unique point on any given interval. (This uniqueness also plays an important role in Williams' decomposition of one-dimensional diffusions \cite{Williams, PY}.)

Fact $1$ follows from Proposition~\ref{propabscon} and fact $2$ from Girsanov's theorem. We furnish a brief proof of fact $3$. Observe that for any two non-overlapping deterministic intervals, the maximum of a Brownian motion (of diffusion parameter 2) $B$ on one interval almost surely differs from the maximum on the other. This is because, conditioned on the values of $B$ at the endpoints of an interval, the maximum attained on the interval has a continuous distribution.
For each $m \in \N$ let $\{I_i^m\}_{i=1}^{m}$ denote the partition of $[-T,T]$ into $m$ non-overlapping intervals of equal length. Were the maximal value of $B$ on $[-T,T]$ achieved more than once, there would exist some value of $m \in \N$ and two indices $i_m\neq j_m$ such that the maxima of $B$ on the two intervals $I_{i_m}^m$ and $I_{j_m}^m$ are equal. That this event has probability zero implies fact $3$.

To prove the second statement in Theorem~\ref{propjoh}, we must show that the argmax stays bounded as $T$ goes to infinity. This result is proved below and recorded as Corollary \ref{corneglim}.
\end{proof}

The remainder of this subsection is devoted to proving Corollary \ref{corneglim} below. Recall that we write the Airy line ensemble minus a parabola and scaled by $2^{1/2}$ as $\mathcal{L}$ (i.e. $\mathcal{L}_i(x) = 2^{-1/2} \big(\mathcal{A}_i(x) - x^2\big)$ for all $i \in \N$ and $x \in \R$). Note that in the notation above, $H(\cdot) = \mathcal{L}_1(\cdot)$. In Theorem \ref{mainthm} we showed the existence of this line ensemble and that it has the Brownian Gibbs property. The one-point distribution of the random variable $\mathcal{A}_1(t)$ has the Tracy-Widom (GUE) distribution. In \cite{TW} precise bounds are proved for the tails of this probability distribution. We do not need the full strength of these bounds and thus record a weaker version here. For some $c > 0$, all $x > 0$ and each $t \in \R$,
\begin{equation}\label{eqtracywidom}
  \PP \Big(  2^{1/2}\mathcal{L}_1(t) + t^2 \leq - x \big) \leq \exp \big\{ - c x^3 \big\} \, , \, \,  \textrm{and} \, \, \, \, \,
 \PP \Big(  2^{1/2}\mathcal{L}_1(t) + t^2 \geq x \big) \leq \exp \big\{ - c x^{3/2} \big\}.
\end{equation}
The factor of $2^{1/2}$ tags along throughout the following argument, but plays no particular consequence.

\begin{proposition}\label{propargmax}
For all $\alpha\in (0,1)$, there exists $c > 0$  and $C(\alpha)>0$ such that for all $t \geq C(\alpha)>0$ and $x \geq - \alpha t^2$,
$$
 \mathbb{P} \Big( \sup_{s \not\in [-t,t]} 2^{1/2}\mathcal{L}_1(s) > x \Big)
 \leq \exp \big\{ -c (t^2 + x)^{3/2} \big\}.
$$
\end{proposition}

\begin{proof}
The proof follows the same strategy that we will use later in establishing Lemma~\ref{lemnobigmax}: if the top curve is ever high, resampling shows that it is likely to be high at some certain fixed time; but, in the present case, the top curve at any given time has the Tracy-Widom (GUE) distribution. In this way, the upper tail in (\ref{eqtracywidom}) applies not only to each  $2^{1/2}\mathcal{L}_1(t) + t^2$ but also to the supremum of this function over a bounded interval of $t$-values.

For $s, t \in \R$, $s < t$,
set $\maxone_{s,t} = \sup_{s \leq r \leq t} 2^{1/2}\mathcal{L}_1(r)$.

Let $M >0$. For $i \in \N$,
set $V_{i,M} = \big\{ 2^{1/2}\mathcal{L}_1(i) + i^2 \geq -M \big\}$. For each $t \in \R$,
$2^{1/2}\mathcal{L}_1(t) + t^2$ has the Tracy-Widom (GUE) distribution; hence, there exists $c > 0$ such that, for each $i \in \N$,
\begin{equation}\label{eqprobvinew}
  \PP\big( V_{i,M}^c \big) \leq \exp \big\{ - c M^3 \big\}.
\end{equation}

Let $x \in \R$ and fix  $i \in \N$; take $i > 0$ for convenience.
 Let $\chi_i =  \chi_{i,x} = \inf \big\{ t \in [i,i+1]: 2^{1/2}\mathcal{L}_1(t) \geq x  \big\}$; if the infimum is taken over the empty-set, we set $\chi_i = \infty$. Of course, $\chi_i < \infty$ precisely when  $\maxone_{i,i+1} \geq x$.

On the event $\chi_i < \infty$, note that $(\chi_i,i+2)$ forms a stopping domain in the sense of Definition~\ref{defstopdom}.
By the strong Gibbs property (Lemma \ref{stronggibbslemma}), almost surely, if $\big\{ \chi_i < \infty \big\} \cap V_{i+2,M}$ occurs,
the conditional distribution of $2^{1/2}\mathcal{L}_1:[\chi_i,i+2] \to \R$, with respect to $\fext(1,\chi_i,i+2)$, is given by Brownian bridge (of diffusion coefficient 2) $B:[\chi_i,i+2] \to \R$, $B(\chi_i) = 2^{1/2}\mathcal{L}_1(\chi_i)$, $B(i+2) = 2^{1/2}\mathcal{L}_1(i + 2)$, conditioned to remain above the curve $2^{1/2}\mathcal{L}_2$.
If $\big\{ \chi_i < \infty \big\} \cap V_{i+2,M}$ occurs then $2^{1/2}\mathcal{L}_1(\chi_i) = x$ and $2^{1/2}\mathcal{L}_1(i+2) \geq -(i+2)^2 - M$, and hence the monotonicity Lemmas~\ref{monotonicity} and \ref{lemmonotonetwo} imply that the conditional distribution of $2^{1/2}\mathcal{L}_1$ stochastically dominates the distribution of a Brownian bridge on $[\chi_i,i+2]$ with endpoint values $x$ and $-(i+2)^2-M$. Noting that $i \leq \chi_i \leq i+1$ implies that $i+1$ lies to the left of the midpoint of the interval $[\chi_i,i+2]$, we see that such a Brownian bridge exceeds $(x - (i+2)^2 - M)/2$ at $i+1$ with probability at least $1/2$.
The conclusion of the argument which we have presented in this paragraph is that
\begin{equation}\label{eqmaxonenew}
 \tfrac{1}{2}  \PP \Big( \big\{ \maxone^N_{i,i+1} > x \big\} \cap V_{i + 2,M} \Big) \leq
    \PP \Big( 2^{1/2}\mathcal{L}_1(i+1) + (i+1)^2  > \tfrac{x - M - (i+2)^2}{2} + (i+1)^2   \Big).
\end{equation}
Now set $M = \tfrac{x + (i+2)^2}{2}$. We require that $M\geq 0$ in order to apply the tail bounds, so this translates into $x\geq -(i+2)^2$. Given this value of $M$,
\begin{equation*}
\tfrac{x- M -(i+2)^2}{2}  + (i+1)^2  = \tfrac{x}{4} -\tfrac{3}{4} (i+2)^2 + (i+1)^2,
\end{equation*}
and given that we have assumed $x\geq -\alpha (i+2)^2$ for some fixed  $\alpha\in (0,1)$, it follows that the above expression exceeds 0 for $i$ larger than some explicit constant $C(\alpha)>0$. Thus we can bound the right hand side in (\ref{eqmaxonenew}) by using (\ref{eqtracywidom}). Also using (\ref{eqprobvinew}) and the fact that $M$ is positive, we find that, for $x > -\alpha (i+2)^2$ and $i>C(\alpha)$,
\begin{equation}\label{eqmaxoneplusnew}
  \PP \Big( \maxone_{i,i+1} > x   \Big) \leq \exp \big\{ - c \big( \tfrac{x}{4} -\tfrac{3}{4} (i+2)^2 + (i+1)^2 \big)^{3/2} \big\}
 \, + \, \exp \big\{ - c \big( \tfrac{x + (i+2)^2}{2} \big)^3 \big\}.
\end{equation}

The identical bound for negative values of $i$ may be similarly obtained. Summing (\ref{eqmaxoneplusnew}) and its negative-$i$ counterpart over $i \in \N$,  $i \geq  \lfloor t \rfloor$,
yields the statement of the proposition with a decrease in the value of $c > 0$.
\end{proof}

The next result is immediate from Proposition \ref{propargmax}.
\begin{corollary}\label{corneglim}
We have that
$$
 \lim_{t \to -\infty} \mathcal{L}_1(t)  =  \lim_{t \to \infty} \mathcal{L}_1(t) = - \infty
$$
$\PP$-almost surely.
\end{corollary}
\begin{proof}[Proof of Proposition \ref{propjoh}.]
Write $\setm$ for the set at which $\mathcal{L}_1$ attains its maximum value.
By Corollary \ref{corneglim},
$\PP \big( \setm \cap [-t,t]^c = \emptyset \big) \to 1$ as $t \to \infty$.  Hence, it suffices to show that
the supremum of $\mathcal{L}_1$ on each $[-t,t]$ is attained at a unique point almost surely. This statement is easy to see for Brownian bridge $B:[-t,t] \to \R$ pinned at arbitrary entrance and exit locations. Hence, the required statement follows from Proposition \ref{propabscon}.
\end{proof}
Let $K \in \R$ denotes the $x$ value at which the supremum of $x \to \mathcal{L}_1(x)$ is attained.
\begin{corollary}\label{Ktailcor}
There exists $c > 0$ such that, for all $x > 0$,
$$
\mathbb{P} \big( \vert K \vert \geq x \big) \leq \exp \big\{ - c x^3 \big\}.
$$
\end{corollary}
\begin{proof}
Note that the occurrence of $\sup_{t \not\in [-x,x]} 2^{1/2}\mathcal{L}_1(t) \leq - x^2/2$ and
 $2^{1/2}\mathcal{L}_1(0) > -x^2/2$ are sufficient to ensure that $\vert K \vert \leq x$.
However, each of these events has probability at least $1 - \exp \big\{ - c x^3 \big\}$: in the first case and for $x$ sufficiently high, this follows from Proposition \ref{propargmax}; in the second, from the lower bound in (\ref{eqtracywidom}). A decrease in $c > 0$ renders the desired inequality valid for all $x > 0$.
\end{proof}

\subsection{Law of the transversal fluctuations of a directed polymer}\label{transec}
We now explain the context and significance of Theorem \ref{propjoh}. In order to do this, we introduce directed polymers, last passage percolation and the PNG model \cite{PS,KJPNG}. For simplicity, we focus only on a certain lattice-based version of these models.

A {\it directed polymer in a random environment} \cite{HH,FH,HHZ,Holl} is a (quenched) probability measure $\PP^n_{\beta}$ on nearest-neighbor random walks $x(\cdot)$ of length $n$ with $x(0)=0$ (we call this set $\Pi_n$):
$$\PP^n_\beta(x(\cdot)) = \frac{1}{Z^n_{\beta}}e^{\beta T(x(\cdot))}\PP^n(x(\cdot))$$
where $\PP^n(x(\cdot))= 2^{-n}$ is the uniform measure on $\Pi_n$, $\beta\geq 0$ is the inverse temperature, and for a given path $x(\cdot)$,
$$T(x(\cdot)) = \sum_{i=1}^{n} w(i,x(i))$$ with $w(i,j)$ i.i.d. random variables. The functional $T$ models the passage time of the path $x(\cdot)$ through a random environment, or, think of $-w(i,j)$ as a potential then and $-T(x(\cdot))$ is the total potential energy of a path. The factor $\exp(\beta T(x(\cdot)))$ is known as a Boltzmann weight, and the necessary normalization $Z^n_{\beta}$ is the partition function.

Taking $\beta$ to infinity, $\PP^n_{\infty}$ concentrates all of its probability measure on the path $x(\cdot)$ with minimal energy or maximal $T(x(\cdot))$. This is called the last passage path and the resulting model is called {\it last passage percolation}.

Two exponents describe the energy and path properties of the directed polymer. Defining the free energy by  $F^{n}_{\beta}=\beta^{-1}\log Z^n_{\beta}$, the exponent $\chi$ represents the size of the free energy fluctuations, i.e., $\var(F^{n}_{\beta}) \approx n^{2\chi}$; the exponent $\xi$ represents the size of transversal fluctuations of the path $x(\cdot)$ away from the line $x(\cdot)\equiv 0$, i.e., $\max(|x(\cdot)|)\approx n^{\xi}$. It has long been predicted in the physics literature that for a wide class of weights (having sufficient moments),
$\chi=1/3$ and $\xi=2/3$, and in particular that $\chi = 2\xi -1$. The belief that the exponents, and also the limiting statistics for these fluctuations, are independent (up to certain centering and scaling) of temperature and of weight distribution is termed KPZ universality, after Kardar, Parisi and Zhang \cite{KPZ} who used earlier (highly non-rigorous) calculations of Forster, Nelson, and Stephen \cite{FNS} to predict the exponents $1/3$ and $2/3$; (they did not predict the limiting statistics associated with these fluctuations -- see the review \cite{CorwinReview} for more recent works).

The first set of mathematical steps towards proving KPZ university was made regarding the zero-temperature ($\beta=\infty$) polymer with specific weight distributions. Following \cite{KJPNG}, fix $w(i,j)$ to be geometric with parameter $q$ (so that $\PP[w(i,j)=m] = (1-q)q^m$) and consider the random variable
$$L(n,y)=\max_{x\in \Pi_n:x(0)=0,x(n)=y} T(x(\cdot))$$
which represents the point-to-point last passage time (or ground-state energy) over paths $x(\cdot)$ in $\Pi_n$ which are pinned to end at $x(n)=y$. Define the random process $t\mapsto H_n(t)$ by linearly interpolating the values set by the relation
$$L(n,y) = c_1 n + c_2 n^{1/3} H_n(c_3 y n^{-2/3}),$$
where
$$
c_1=\frac{2\sqrt{q}}{1-\sqrt{q}},\qquad c_2 =\frac{(\sqrt{q})^{1/3}(1+\sqrt{q})^{1/3}}{1-q},\qquad c_3 = c_2 \frac{1-\sqrt{q}}{1+\sqrt{q}}.
$$
Johansson then showed that, for any fixed $T$, the random function $H_N(t):[-T,T]\rightarrow \R$ converges weakly (as a probability measure on $C([-T,T],\R)$ under the uniform topology) to the random function
\begin{equation}\label{Hstarstar}
t\mapsto H(t)=\mathcal{A}_1(t)-t^2,
\end{equation}
where $\mathcal{A}_1$ is the top line of the Airy line ensemble that we have defined in this paper. (Johansson showed that this top line was a well-defined continuous function, but did not treat other lines of the ensemble.) The Airy process had been introduced a year or so earlier by Pr\"{a}hofer and Spohn (see also the independent and parallel work of Okounkov and Reshetikhin \cite{OkResh}) who considered a related model, for which they proved convergence of finite-dimensional distributions. These results also extended the seminal works of Baik, Deift and Johansson \cite{BDJ} and Johansson \cite{KJ} which only dealt with one-point convergence.

The upshot of this work is that (at least for $\beta=\infty$ and geometric weights) the point-to-point free energy of a directed polymer fluctuates on the scale $n^{\chi}$ for $\chi=1/3$, and has a non-trivial transversal correlation $n^{\xi}$ for $\xi=2/3$ -- and moreover that this entire correlation structure can be described in terms of the Airy process. This picture is believed to hold for all $\beta>0$ and for all distributions with sufficiently many moments. This has only been proved (at the level of the exponent) for one type of distribution \cite{S} or for $\beta \to 0$ as $n \to \infty$  \cite{ACQ,AKQ} in which case one encounters the so-called {\it continuum directed polymer} and the related KPZ stochastic PDE (see also \cite{SaSp}).

Returning to the original problem of the directed polymer (with unconstrained endpoint), define
$$K_n = \inf\big\{u: \sup_{t\leq u}H_n(t)= \sup_{\R}H_n(t)\big\},$$
which is in essence the location of the maximizing path up to time $n$ (the infimum and supremum are due to the discreteness of our weights). Likewise, define $K \in \R$ according to
\begin{equation}\label{defnk}
H(K)= \sup_{t \in \R } H(t),
\end{equation}
where $H(t)$ is as in (\ref{Hstarstar}). This is well-defined in light of Theorem \ref{propjoh}.

Johansson was interested in showing that the limit of the law of $K_n$ would coincide with the law of $K$ (a result that he phrased a little differently, because it was not proven that $K$ was well defined). This convergence implies that the law of the endpoint of the directed polymer (at least at zero temperature) converges to the law of the point at which the maximum of the Airy process minus a parabola is achieved. In this direction, he showed that $\{K_n\}_{n=1}^{\infty}$ is a tight sequence and formulated a theorem which states that, under Conjecture 1.5 of \cite{KJPNG} (which is now proved as Theorem \ref{propjoh}), this convergence indeed holds. Thus, we have shown the next theorem.
\begin{theorem}\label{maxdist}
The random variables $K_n$ converge to $K$ in distribution as $n \to \infty$.
\end{theorem}

Theorem \ref{maxdist} shows that the polymer endpoint fluctuates in the transversal direction on the scale of $n^{2/3}$ (i.e. $\xi=2/3$) and has a limiting location determined by the location of the maximum of the Airy process minus a parabola. This theorem holds also for lower lines in the Airy line ensemble (so that each curve has a unique maximizing point after a parabola is subtracted). The interpretation of this result in terms of polymers is slightly more obtuse. Johansson considered a multi-line PNG line ensemble (as introduced in \cite{PS}). The second line actually records information about the last passage time for two polymer paths which cannot touch (and the $k^{\rm{th}}$ line corresponds to $k$ such paths). Thus one sees that these multi-path polymers have endpoints whose law can be described in terms of the Airy line ensemble.

Given that $K$ is a well-defined random variable, one might ask what its distribution is, or, for that matter, what is the joint distribution of $(K,H(K))$. From Corollary \ref{Ktailcor} one knows that as $x$ grows, the probability $|K|>x$ decays at least as fast as $\exp\{-c x^3\}$ for some constant $c>0$. The study of the distribution of $(K,H(K))$ has received a lot of attention -- in particular with a number of new results coming after the present paper was posted. For $N$ non-intersecting Brownian bridges (as well as some of the other variants discussed earlier), the question had previously been studied in physics \cite{RS1,RS2,feierl2,FMS} and more recently in \cite{Sch} and mathematically in \cite{Lie}. A direct approach to computing this joint distribution for the Airy line ensemble \cite{MQR} relies on taking Fr{\'e}chet derivatives of the continuum statistics formulas for the top line of the Airy line ensemble (as developed in \cite{CQR}) and also uses the Brownian absolute continuity results provided in this paper in Proposition \ref{propabscon}. These approaches yield different formulas which are shown to coincide in \cite{BLS}. These exact formulas have led to exact values of the tail decay constant $c=4/3$.

\subsection{Positive association and the Adler-Moerbeke conjecture}

Theorem 1.6 of \cite{AvM} computes the large time asymptotics for the joint distribution of the Airy process (to fourth order). As the authors mention, the theorem was contingent upon a claimed, though not proved result -- here stated as Corollary~\ref{propadmoconj}. In fact, a complete proof of this theorem was anyway furnished shortly afterwards by Widom \cite{W} using Fredholm determinants;  more recently, it has been extended to higher-order precision in \cite{ST}.

\begin{corollary}[Positive association]\label{propposass}
Let $T > 0$ and $x \in \R$. The conditional distribution of $\mathcal{A}_1(T)$ given that $A_1(0) \geq x$ stochastically dominates the unconditioned distribution of $\mathcal{A}_1(T)$.
\end{corollary}
\begin{proof}
Recall from Definition \ref{Dn} the edge-scaled Dyson line ensembles $\mathcal{D}^N_i:[-N^{1/3},N^{1/3}] \to \R$, $1 \leq i \leq N$. Let $x > 0$, and write $E_N(x)$ for the event that
$\mathcal{D}^N_1(0) \geq x$.

We claim that $\mathcal{D}^N$, $1 \leq i \leq N$, given $E_N(x)$, converges weakly as a line ensemble to $\mathcal{L}$ given the event that $\mathcal{L}_1(0) \geq x$.  To verify this, recall from the proof of Proposition~\ref{propBrownianGibbsH} that a simultaneous construction of these line ensembles may be made so that, for each $k \in \N$, $\mathcal{D}^N_k$ converges to $\mathcal{L}_k$ locally uniformly.
Note that $\PP\big( E_N(x) \big)$ has limit $\PP \big( \mathcal{L}_1(0) \geq x  \big)$ whose value is given by the Tracy-Widom (GUE) distribution. Therefore, for $N$ large enough, we know that the event $E_N(x)$ on which we are conditioning has strictly positive probability (with a non-zero limit). This means we may couple the conditioned systems as $N$ variables so as to enjoy the same local uniform convergence of the lines. This in turn implies the weak convergence of the entire ensemble and verifies the claim made at the start of this paragraph.

Note that  $\mathcal{D}^N$ given $E_n(x)$ stochastically dominates $\mathcal{D}^N$ without conditioning. This is not a consequence of the monotonicity lemmas we have proved, but can easily be proved via the same method as those lemmas (see Section \ref{proofslemmas}). We forgo repeating the proof here.

The coupling of conditioned systems shows that $\mathcal{L}$ given $\mathcal{L}_1(0) \geq x$ inherits this stochastic domination property. Hence, after putting this back into a statement about $\mathcal{A}_1$, the corollary follows.
\end{proof}

We obtain the next result as an immediate consequence of Corollary \ref{propposass} by noting that the positive association implies $\PP\Big(\mathcal{A}_1(t)>a \Big \vert \mathcal{A}_1(0)<b\Big) \leq \PP\big(\mathcal{A}_1(T)>a\big).$
\begin{corollary}[Conjecture 1.21 of \cite{AvM}]\label{propadmoconj}
For any $T > 0$,
$$
 \lim_{M \to \infty}  \PP \Big( \mathcal{A}_1(T) > M  \Big\vert  \mathcal{A}_1(0) < - M \Big) = 0
$$
almost surely.
\end{corollary}

\section{Uniform bounds on acceptance probabilities, \\ minimal line gaps and maximal height}\label{technicalSec}
This section is devoted to the proof of Proposition \ref{propacceptprob}. Let us briefly outline how we proceed.

The claim of equation (\ref{maxmineqn}) is the easy part of the proof. In order to prove the bound of equation (\ref{maxmineqn}), we prove Lemmas \ref{lemnobigmax} and \ref{lemnobigmin}.

We then turn to the claims of equations (\ref{APpropeqn}) and (\ref{GAPpropeqn}) which are much more substantial. The proof strategy begins by noting that, in light of the uniform control provided by equation (\ref{maxmineqn}) on the $(k+1)$-st curve on $[-T,T]$, with high probability the convex majorant of this curve may have high derivative only on short intervals close to the endpoints $-T$ and $T$. Restricting attention to a slightly shorter domain, we can be assured that with high probability this concave majorant does not have a large slope; however, the two $k$-decreasing lists which describe the entrance and exit data of the top $k$ curves at the ends of this new interval could be very poorly behaved, with clustering of adjacent points, for example. In Proposition~\ref{propconcave}, we will show how the mutual avoidance constraint causes this entrance and exit data to improve under a further slight shortening of the domain on which the curves are considered. With well-spaced boundary data and a well-behaved concave majorant for the lower curve boundary data, we are then able to obtain the regularity claimed by equations (\ref{APpropeqn}) and (\ref{GAPpropeqn}).

As an aside, we mention that the effect identified by Proposition~\ref{propconcave} -- that in a system of mutually avoiding processes, low quality boundary data is rare because it entails that the avoidance condition is likely to be quickly violated -- is vaguely reminiscent of the separation of arms in planar critical percolation (see the appendix of \cite{GPS}) and separation of level set contours in two-dimensional random surfaces (see Section 3.5 of \cite{schrammsheffield}).

In the proofs in this section, three devices to which we will often  appeal are (i) the strong Gibbs property (Lemma \ref{stronggibbslemma}), to argue that we may consider random intervals whose size is measurable with respect to the external sigma-fields (called stopping domains) as if there were deterministic; (ii) monotone couplings of non-intersecting Brownian bridge ensembles (Lemmas~\ref{monotonicity} and \ref{lemmonotonetwo}), to understand the behavior of the full $N$ line ensembles in terms of that of $k$ line ensembles; and (iii) the Brownian Gibbs property, to translate regularity of boundary data into regularity of the internal collection of lines.

\subsection{Bounds on the maximal height}
The proof of the following lemma already illustrates in a small way the guiding theme of the paper: the Brownian Gibbs property is not merely an expression for the conditional distributions of the line ensemble; rather, the property offers a valuable probabilistic resampling technique by which to prove regularity of
line ensembles which enjoy it. The method of proof is similar to Proposition \ref{propargmax}'s, which we encountered earlier. In this section, we will often suppress the superscript $N$ in the index of the line ensemble.

\begin{lemma}\label{lemnobigmax}
Fix $k\geq 1$ and $T>0$. Consider a sequence of line ensembles $\{\mathcal{L}^N\}_{N=1}^{\infty}$ satisfying Hypothesis $(H')_{k,T}$. Write $\maxdyson^N(i,T) = \sup_{s \in [-T,T]} \mathcal{L}^N_i(s)$.  Then for all $\e>0$, there exists an $x > 0$ and $N_0>0$ such that, for all $N \geq N_0$ and $i\in \{1,\ldots, k\}$,
\begin{equation}\label{maxdysonprob}
\PP^N \Big(\maxdyson^N(i,T) > x \Big) \leq \e.
\end{equation}
\end{lemma}
\begin{proof}
It suffices to prove the case $i=1$. For $-T \leq s < t \leq T$, set $\maxone^N_{s,t} = \sup_{s \leq r \leq t} \mathcal{L}^N_1(r)$. We will use a resampling argument to show that the event $\big\{\maxone^N_{i,i+1} > x\big\}$ is improbable. If this event takes place, it means a large maximum is achieved at a random time; but a resampling will show that in this case it is also likely that the top curve is high at one of a fixed number of deterministic times. However, the distribution of $\mathcal{L}^N_1$ at fixed times is controlled by Hypothesis  $(H2')_{k,T}$, which will show that this outcome is unlikely.

For $j \in \N$, $-T \leq j \leq T$, set
$$V_j = V_{j,N,M} = \big\{  \mathcal{L}^N_1(j) \geq -M \big\}.$$
By Hypothesis  $(H2')_{k,T}$, for $\e > 0$, there exists $M > 0$ such that, for all $N \in \N$ and for each such $j$,
\begin{equation}\label{eqprobvi}
  \PP^N \big( V_j^c \big) \leq \frac{\e}{4T}.
\end{equation}

Now consider $x > 0$ and restrict $j$ further so that $-T \leq j \leq T - 2$. Let
$$\chi_j =  \chi_{j,x,N} = \inf \big\{ t \in [j,j+1]: \mathcal{L}^N_1(t) \geq x \big\},$$
where, if the infimum is taken over the empty-set, we set $\chi_j = \infty$. Of course, $\chi_j < \infty$ precisely when  $\maxone^N_{j,j+1} \geq x$.

On the event $\chi_j < \infty$, note that $(\chi_j,j+2)$ forms a stopping domain in the sense of Definition~\ref{defstopdom}.
By the strong Gibbs property of Lemma \ref{stronggibbslemma}, $\PP^N$-almost surely, if $\big\{ \chi_j < \infty \big\} \cap V_{j+2}$ occurs,
the conditional distribution of $\mathcal{L}^N_1:[\chi_j,j+2] \to \R$ is given by Brownian bridge $B:[\chi_j,j+2] \to \R$, $B(\chi_j) = \mathcal{L}^N_1(\chi_j)$,
$B(j+2) = \mathcal{L}^N_1(j + 2)$, conditioned to remain above the curve $\mathcal{L}^N_2$. Noting that $\mathcal{L}^N_1(\chi_j) = x$ and that $\mathcal{L}^N_1(j+2) \geq -M$, the monotonicity Lemmas~\ref{monotonicity} and \ref{lemmonotonetwo} imply that this conditional distribution stochastically dominates that of an independent Brownian bridge on $[\chi_j,j+2]$ with endpoint values $x$
and $-M$. Noting that $\chi_j \leq j+1$ implies that $j+1$ lies to the left of the midpoint of the interval $[\chi_j,j+2]$, we see that such a Brownian bridge exceeds $(x-M)/2$ at $j+1$ with probability at least $1/2$.
The conclusion of the argument which we have presented in this paragraph is thus
\begin{equation}\label{eqmaxone}
 \tfrac{1}{2}  \PP^N \Big( \big\{ \maxone^N_{j,j+1} > x \big\} \cap V_{j + 2} \Big) \leq
    \PP^N \Big( \mathcal{L}^N_1(j+1) > \frac{x- M }{2}  \Big)
\end{equation}
for all $N \in \N$.
Invoking again Hypothesis  $(H2')_{k,T}$, we choose $x > 0$ large enough that the right-hand term in (\ref{eqmaxone}) is at most $\e/(8T)$ for each $N \in \N$;  applying (\ref{eqprobvi}), we find that
\begin{equation}\label{eqmaxoneplus}
  \PP^N \Big( \maxone^N_{j,j+1} > x  \Big) \leq \frac{\e}{2T}.
\end{equation}
Technically, we must also justify (\ref{eqmaxoneplus}) in the case where  $\maxone^N_{j,j+1}$ is replaced by
$ \maxone^N_{-T,\lfloor -T \rfloor}$ and by $ \maxone^N_{\lfloor T \rfloor - 1,T}$. Treating
$-T$ and $T$ as deterministic times to which we apply Hypothesis $(H2')_{k,T}$, the derivation of these two extra bounds proceeds very similarly to the argument that we have given.

Summing (\ref{eqmaxoneplus}) over $j \in \N$,  $-T \leq j \leq T - 2$, and adding its counterparts for the two miscellaneous cases, yields the statement of the lemma.
\end{proof}

Lemma \ref{lemnobigmax} has an analog for the minumum value.
\begin{lemma}\label{lemnobigmin}
Fix $k\geq 1$ and $T>0$. Consider a sequence of line ensembles $\{\mathcal{L}^N\}_{N=1}^{\infty}$ satisfying Hypothesis $(H')_{k,T}$. Write $\mindyson^N(i,T) = \inf_{s \in [-T,T]} \mathcal{L}^N_i(s)$.  Then for all $\e>0$, there exists an $x > 0$ and $N_0>0$ such that, for all $N \geq N_0$ and $i\in \{1,\ldots, k\}$,
\begin{equation}\label{mindysonprob}
\PP^N \Big(\mindyson^N(i,T) < - x \Big) \leq \e.
\end{equation}
\end{lemma}
\begin{proof}
We prove the assertion by induction on $k$. Let $k \geq 1$ and assume the inductive hypothesis at $i=k-1$. (If $k=1$, define a zeroth curve $\mathcal{L}^N_0 \equiv \infty$ so that (\ref{mindysonprob}) trivially holds.) We must show that (\ref{mindysonprob}) now holds for $i=k$ as well.

Define events $E_x$ and $\tilde{E}_{\tilde x}$ as
\begin{eqnarray*}
E_x&=& \big\{ \mathcal{L}^N_{k-1}(t) \geq -x \, \forall t\in [-T,T]\big\}, \\
\tilde{E}_{\tilde{x}} &=&  \big\{ \mathcal{L}^N_{k}(t) \geq -\tilde x \, t=-T,T\big\}.
\end{eqnarray*}

There exists $x_{k-1}>0$ and $\tilde{x}_k>x_{k-1}$ such that for large $N$, $\PP\big(E_{x_{k-1}}\cap \tilde{E}_{\tilde{x}_k}\big)\geq 1-\tfrac{\e}{2}$. This fact follows from the inductive hypothesis, which shows that $\PP(E_{x_{k-1}})\geq 1-\tfrac{\e}{4}$, and Hypothesis $(H2')_{k,T}$, which shows that for $\tilde{x}_k$ large enough, $\PP(\tilde{E}_{\tilde{x}_k})\geq 1-\frac{\e}{4}$.

From now on, we will restrict attention to the portion of the probability space in which the event $E=E_{x_{k-1}}\cap \tilde{E}_{\tilde{x}_k}$, $\mathcal{L}_k(t)$, $t\in[-T,T]$ is distributed as a Brownian bridge from $\mathcal{L}_{k}(-T)$ to $\mathcal{L}_{k}(T)$ conditioned to not intersect $\mathcal{L}_{k-1}$ and $\mathcal{L}_{k+1}$. If we remove the condition to avoid $\mathcal{L}_{k+1}$ and call the resulting conditioned (only on $\mathcal{L}_{k-1}$-avoidance) Brownian bridge $B$, then it follows from Lemma~\ref{monotonicity} that $B{\cdot}$ and $\mathcal{L}_k(\cdot)$ can be coupled so that $B(\cdot)\leq \mathcal{L}_k(\cdot)$. If we further replace the starting and ending heights for $B(\cdot)$ by $-\tilde{x}_k$, then, since this lowers the starting and ending points, we may use Lemma \ref{lemmonotonetwo} to couple the resulting Brownian bridge $B'$ (which is still conditioned to avoid $\mathcal{L}_{k-1}$ but now starts and ends at $\tilde{x}_k$) so that $B'(\cdot)\leq B(\cdot)$. Finally, if we replace the condition of avoiding $\mathcal{L}_{k-1}$ by the condition of staying below $-x_{k-1}$, the resulting Brownian bridge $B''$ can be coupled (again by Lemma \ref{monotonicity}) so that $B''(\cdot)\leq B'(\cdot)$.

The output of this string of deductions is that, on the event $E$, there is a coupling of $\mathcal{L}_k:[-T,T]\to \R$ with $B'':[-T,T]\to \R$ such that $B''(\cdot) \leq \mathcal{L}_k(\cdot)$ where $B''$ is distributed as Brownian bridge with $B''(\pm T) = -\tilde{x}_k$ conditioned to stay below $-x_{k-1}$. By possibly choosing $\tilde{x}_k$ larger, we can ensure that $\tilde{x}_k-x_{k-1}\geq T^{1/2}$. Given this, the conditioned Brownian bridge $B''$ is not likely to go low. In particular, using Lemma~\ref{lembridgemod}, we readily show that
$$\PP\big(B''(t) \leq -\tilde{x}_k - yT^{1/2} \textrm{ for some } t\in [-T,T]\big) \leq e^{2} e^{-2y^2/T}.$$
By the monotonicity between $\mathcal{L}_k$ and $B''$ on the event $E$, and the high probability of $E$, we conclude that
$$\PP\Big(\mathcal{L}_k(t) \leq -\tilde{x}_k - yT^{1/2} \textrm{ for some } t\in [-T,T]\Big) \leq \e/2 + e^{2} e^{-2y^2/T}.$$
Choosing $y=T^{1/2} \sqrt{1-\log\sqrt{\e/2}}$ and setting $x_k = -\tilde{x}_k -y T^{1/2}$, we find that
$$\PP\big(\mathcal{L}_{k}(t)<-x_k \textrm{ for some } t\in [-T,T]\big) \leq \frac{\e}{2}+\frac{\e}{2},$$
which is the result desired to prove the inductive step for $i=k$.
\end{proof}

\subsection{Proof of Proposition \ref{propacceptprob}}\label{proof:propacceptprob}

\begin{figure}[ht]
\begin{center}
\includegraphics[height=8cm]{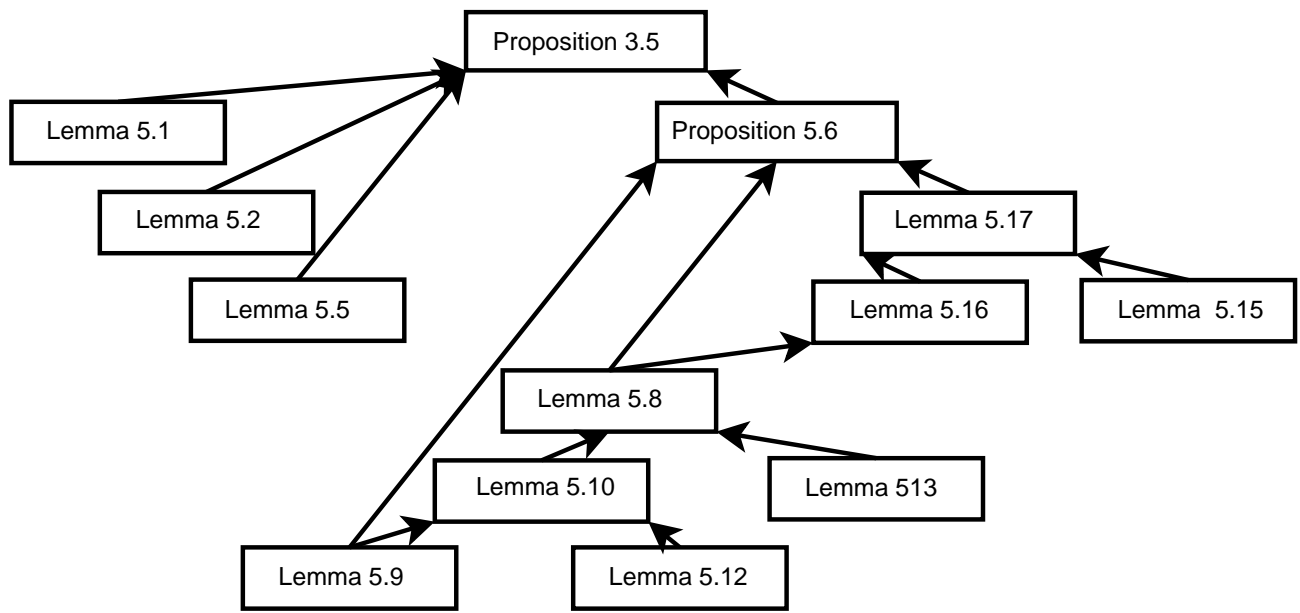}
\caption{Flow chart of the proof of Proposition \ref{propacceptprob}. The left-hand side represents Step 1 and the right-hand side represents Step 2.}
\label{flowchart}
\end{center}
\end{figure}

First observe that equation (\ref{maxmineqn}) follows immediately from Lemmas \ref{lemnobigmax} and \ref{lemnobigmin}. Turning to the proofs of equations (\ref{APpropeqn}) and (\ref{GAPpropeqn}), beyond Lemma \ref{lemnobigmax}, the central element used here is Proposition \ref{propconcave}, a result which we will shortly state. In Proposition \ref{propconcave}, we consider a $k$-curve $(\infty,f)$-avoiding line ensemble defined on an interval $[-T,T]$. The result shows that, although the line ensemble may have poorly behaved entrance and exit data (including points clustered together very closely), the curves in the ensemble typically become well-separated on the slightly shorter interval $[-(T-1),T-1]$, provided only that the lower boundary condition $f$ has a well-behaved convex hull.

A few definitions are needed.

\begin{definition}
For $K> 0$ and for $\ell<r$, let $\CMK$ be the set of functions $f:[\ell,r]\to \R$ such that its least concave majorant $\CM^f:[\ell,r]\to \R$ satisfies
\begin{equation*}
\frac{\vert \CM^f(x) - \CM^f(y) \vert}{\vert x - y \vert} \leq K
\end{equation*}
for all $x,y \in [\ell,r]$. For $M > 0$, let $\XYfM$ be the set of pairs
$\bar{x}=(x_1,\ldots,x_k), \bar{y}=(y_1,\ldots,y_k) \in \Rkle$ such that
\begin{equation*}
x_k > f(\ell), \quad y_k > f(r) \qquad \textrm{ and } \quad -M \leq \min \{x_k,y_k\} < \max \big\{ x_1,y_1 \big\} \leq M.
\end{equation*}
\end{definition}

\begin{definition}
Fix $k\geq 1$, $N\geq k+1$ and $T>0$. Consider a line ensemble
$\mathcal{L}:\{1,\ldots, N\}\times [-T,T]\to \R$. Define $\CM:[-T,T] \to \R$ so that
$\big\{ \big( t,\CM(t) \big): -T\leq t \leq T \big\} \subseteq \R^2$ is the
least concave majorant of $\mathcal{L}_{k+1}$ restricted to $[-T,T]$.
Write $\CM'_{+}:[-T,T) \to \R$ for the right derivative of
$\CM: [-T,T] \to \R$.
Recalling Definition~\ref{defstopdom}, for $K>0$,
define a stopping domain $(\tilde{\mathfrak{l}}_K,\tilde{\mathfrak{r}}_K)$ by
\begin{eqnarray*}
\tilde{\mathfrak{l}}_K &=& \inf \big\{ t \in \big[ -T,T \big] : \CM'_{+}(t) \leq    K \big\}, \\
\tilde{\mathfrak{r}}_K &=& \sup \big\{ t \in  \big[ -T,T \big] : \CM'_{+}(t) \geq - K \big\},
\end{eqnarray*}
adopting the convention that $\inf \emptyset = T$ and $\sup \emptyset = -T$. Note that  $(\tilde{\mathfrak{l}}_K,\tilde{\mathfrak{r}}_K)$ is indeed a stopping domain for lines $\mathcal{L}_1,\ldots, \mathcal{L}_k$ because this random variable is measurable with respect to $\mathcal{L}_{k+1}$.
\end{definition}

\begin{lemma}\label{lemellrbd}
Fix $k\geq 1$, $N\geq k+1$ and $T>0$. Consider a line ensemble
$\mathcal{L}:\{1,\ldots, N\}\times [-T,T]\to \R$.
For $M>0$, on the event
\begin{equation}\label{Eevent}
E_{k,T,M}=\Big\{ \sup_{t \in [-T,T]} \mathcal{L}_1(t)  \leq M \Big\} \cap \big\{ \mathcal{L}_{k+1}(-T) \geq -M \big\} \cap
\big\{ \mathcal{L}_{k+1}(T) \geq -M \big\},
\end{equation}
we have that, almost surely,
\begin{equation*}
- T \leq \tilde{\mathfrak{l}}_{K} \leq - T + 2M/K,\qquad {\textrm and} \qquad  T - 2M/K \leq \tilde{\mathfrak{r}}_K \leq T.
\end{equation*}
\end{lemma}

\begin{proof} The line segment that connects $\big(-T,\CM(-T) \big)$ and $\big( \tilde{\mathfrak{l}}_{K}, \CM(\tilde{\mathfrak{l}}_K)  \big)$ has slope at least the left-derivative of $\CM(\tilde{\mathfrak{l}}_K)$, since $\CM$ is concave. This left-derivative is at least $K$, by the definition of $\tilde{\mathfrak{l}}_K$. Hence,
\begin{equation}\label{eqctk}
 \frac{\CM(\tilde{\mathfrak{l}}_K) - \CM(-T)}{\tilde{\mathfrak{l}}_K - (-T)} \geq K.
\end{equation}
Note that $\CM(-T) = \mathcal{L}_{k+1}(-T)$ and  $\CM(\tilde{\mathfrak{l}}_K)= \mathcal{L}_{k+1}(\tilde{\mathfrak{l}}_K)$, the latter since the concave region $\big\{ ( t , y ): t \in [-T,T], y \leq \CM(t) \big\}$ has an extreme point at $\big(\tilde{\mathfrak{l}}_K,\CM(\tilde{\mathfrak{l}}_K)\big)$.
Hence, (\ref{eqctk}) holds also when $\mathcal{L}_{k+1}$ replaces $\CM$. Applying the bounds provided by the occurrence of $E_{k,T,M}$, we find that $\tilde{\mathfrak{l}}_K \leq - T + 2M/K$. Likewise, the second bound.
\end{proof}

Recall that $\bxyflr(\cdot)$ denotes the conditional distribution of $k$ Brownian bridges on $[\ell,r]$ with entrance and exit data $\bar{x}$ and $\bar{y}$ conditioned to meet neither one another nor the curve $f$.

\begin{proposition}\label{propconcave}
Fix $K,M>0$ and two finite intervals $I_1,I_2 \subseteq \R$ for which $\sup I_1< \inf I_2 + 3$.
Then, for all $\epsilon > 0$, there exists $\delta=\delta(\e,K,M,I_1,I_2)> 0$ such that, for $(\ell,r)\in I_1\times I_2$,  $f\in \CMK$ and $(\bar{x},\bar{y})\in \XYfM$,
\begin{equation}\label{APeqn}
\bxyflr \Big(\AP\big(\ell+1,r-1,\{B_{i}(\ell+1)\}_{i=1}^{k},\{B_{i}(r-1)\}_{i=1}^{k}, f \big) < \delta \Big) < \epsilon
\end{equation}
and
\begin{equation}\label{GAPeqn}
\bxyflr \Big(\min_{1 \leq i \leq k - 1} \inf_{s \in [\ell+1,r-1]} \big\vert B_i(s)  - B_{i+1}(s)  \big\vert < \delta \Big) < \epsilon.
\end{equation}
\end{proposition}

This result is the key to completing our proof of Proposition \ref{propacceptprob}. We will use it with $(\ell,r)$ replaced by stopping domains (as is justified by the strong Gibbs property). The proof of this result, which will be given in Section \ref{proof:propconcave}, relies fundamentally on the resampling technique that is at the heart of the Brownian Gibbs definition.

Returning to the task of finishing the proof of Proposition \ref{propacceptprob}, by a combination of Lemma \ref{lemnobigmax} and the one-point convergence ensured by Hypothesis $(H2')_{k,T}$, we deduce the following: for all $\e>0$, there exists an $N_0>0$ and $M_0>0$ such that, for all $N \geq N_0$ and $M \geq M_0$,
\begin{equation}\label{Eprob}
\PP^N(E_{k,T,M})\geq 1-\e/2,
\end{equation}
where $E_{k,T,M}$ is given in equation (\ref{Eevent}) (this also follows immediately from the now proved equation (\ref{maxmineqn})).
On this event, Lemma \ref{lemellrbd} ensures that by choosing $K=2M$ we have
\begin{equation}\label{lrbound}
\tilde{\mathfrak{l}}_{K} \leq -T+1, \qquad \tilde{\mathfrak{r}}_{K} \geq T-1.
\end{equation}
By the strong Gibbs property given in Lemma \ref{stronggibbslemma}, we may apply Proposition \ref{propconcave} with $(\ell,r)$ replaced by the stopping domain $(\tilde{\mathfrak{l}}_k,\tilde{\mathfrak{r}}_k)$. Also, replace $\e$ by $\e/2$, and let $K$ and $M$ be specified as above, and $I_1=[-T,-T+1]$, $I_2=[T-1,T]$ and $f=\mathcal{L}^N_{k+1}$.
The conclusion (\ref{GAPeqn}), along with (\ref{lrbound}), shows that there exists $\delta$ (depending on $\e$ through $K,M$, and also depending on $k$ and $T$) such that
\begin{equation}\label{condGap}
\PP^N\Big(\{M_{k,-T+2,T-2}^N <\delta\} \cap E_{k,T,M}\Big) < \e/2.
\end{equation}
We now obtain (\ref{GAPpropeqn}) from (\ref{Eprob}) and (\ref{condGap}) where $T$ is replaced by $T+2$.

To obtain (\ref{APpropeqn}), set $\mathfrak{l}^N = \min \big\{ \tilde{\mathfrak{l}}_K +1  , -( T - 2) \big\}$
and  $\mathfrak{r}^N = \max \big\{ \tilde{\mathfrak{r}}_K - 1 , T - 2 \big\}$; certainly,
$\big( \mathfrak{l}^N, \mathfrak{r}^N \big)$ forms a stopping domain. Taking $K = 2M$, note that,
by~(\ref{lrbound}),
the occurrence of $E_{k,T,M}$ ensures that $\mathfrak{l}^N = \tilde{\mathfrak{l}}_K + 1$
and $\mathfrak{r}^N = \tilde{\mathfrak{r}}_K - 1$.
Note also that, by the combination of Lemma \ref{stronggibbslemma} and Proposition \ref{propconcave},
for all $\e>0$, there exists $\delta$ (with the same dependencies as before) such that
\begin{equation*}
\PP^N \Big(\Big\{\AP\Big(\tilde{\mathfrak{l}}_K+1,\tilde{\mathfrak{r}}_K-1,\{\mathcal{L}^N_{i}(\tilde{\mathfrak{l}}_K+1)\}_{i=1}^{k},\{\mathcal{L}^N_{i}(\tilde{\mathfrak{r}}_K-1)\}_{i=1}^{k}, \mathcal{L}^N_{k+1}(\cdot)\Big) < \delta \Big\} \cap
E_{k,T,M} \Big) < \epsilon/2.
\end{equation*}
These last two observations combine with~(\ref{Eprob}) to give (\ref{APpropeqn}).
This completes the proof of Proposition~\ref{propacceptprob}. \qed

\subsection{Precursor lemmas, and the proof of Proposition \ref{propconcave}}\label{proof:propconcave}
In this section we will state and prove two Lemmas, \ref{lemellplusone} and \ref{eqheps}, which will, at the end of the section, serve as key inputs in our proof of Proposition \ref{propconcave}.

Fix $K,M$ and intervals $I_1,I_2$ as in the statement of the proposition and let $\e>0$. To study the law  $\bxyflr = \wxylr \big( \cdot \big\vert \ncf_{[\ell,r]} \big)$ specified in Definition~\ref{WBdef}, we will condition $\wxylr$ in two steps: first, on $\ncf_{[\ell,\ell + 1] \cup [r-1,r]}$, and afterwards, also  on $\ncf_{[\ell +1,r-1]}$.
The principal conclusions of the respective steps will be Lemma \ref{lemellplusone} and Lemma \ref{eqheps}.

\subsubsection{Step 1} We start by analysing the law $\wxylr$ given $\ncf_{[\ell,\ell + 1] \cup [r-1,r]}$, where $f \in \CMK$ and $\bar{x},\bar{y} \in \XYfM$. Although
 the entrance and exit data $\bar{x},\bar{y}$ at times $\ell$ and $r$ may include pairs of adjacent points that are extremely close, we will explain how, under the measure in question, this data improves in the interior of $[\ell,r]$; we will show that, with a conditional probability which is positive uniformly in such
$\bar{x}$, $\bar{y}$ and $f$,  the top $k$ curves at times $\ell + 1$ and $r-1$ have risen significantly from the underlying boundary condition and have separated from each other.

\begin{definition}  Fix $k\geq 1$, $K>0$, $\ell<r$ and $f\in \CMK$. A $k$-decreasing list $\bar{x} = \big(x_1,\ldots,x_k \big)  \in \R^k_>$ will be called {\it $\e$-well-spaced} at $\ell+1$ (or at $r-1$) if $x_k > f(\ell)+K$ (or $x_k > f(r)-K$) and if $\min_{1 \leq i \leq k-1}\big\vert x_{i-1} - x_{i} \big\vert > \epsilon$.
We write $\ws_\epsilon$ for the event that
$\big( B_1(t),\ldots,B_k(t) \big)$ is $\epsilon$-well-spaced at both $t \in \big\{ \ell + 1, r - 1 \big\}$.
\end{definition}

The main result of step 1 is the following.
\begin{lemma}\label{lemellplusone}
There exists an $\e_0=\e_0(K,M,I_1,I_2)$ such that for  $0 < \e <\e_0$, $\ell\in I_1$, $r\in I_2$, $f\in \CMK$ and $(\bar{x},\bar{y})\in \XYfM$,
\begin{equation*}
\wxylr \big( \ws_\epsilon \big\vert \ncf_{[\ell,\ell + 1] \cup [r-1,r]} \big) > \epsilon.
\end{equation*}
\end{lemma}

\begin{proof}
Inverting the relationship in Lemma \ref{lemellr} shows that for some $\delta_0>0$, for all $\delta<\delta_0$ there exists $\e(\delta)$ (which depends also on $K,M,I_1$ and $I_2$) such that the event of the minimal gap between consecutive lines being less than $\delta$ at times $t=\ell+1$ or $r-1$, has probability at least $1-\e(\delta)$. Note that as $\delta\to 0$, so too can $\e(\delta)\to 0$. Lemma \ref{lembkeps} shows that there is an $\tilde{e}_0>0$ such that the event that $B_k(\ell+1)> f(\ell)+K$ and $B_k(r-1)>f(r)+K$ holds with probability at least $\tilde{e}_0$.

Now let $\e_0$ be chosen such that for all $\delta<\e_0$, $\tilde{e}_0-\e(\delta)>\delta$. The existence of such an $\e_0$ is guaranteed by the decay of $\e(\delta)\to 0$ as $\delta\to 0$. Combining the two bounds given above (via the union bound) shows that, for all $\delta<\e_0$,
\begin{equation*}
\wxylr \big( \ws_\delta \big\vert \ncf_{[\ell,\ell + 1] \cup [r-1,r]} \big) > \tilde{e}_0-\e(\delta) > \delta.
\end{equation*}
Hence by replacing $\delta$ by $\e$, the above bound yields the claimed result.
\end{proof}

The remainder of step 1 is devoted then to stating and proving Lemmas \ref{lemellr} and \ref{lembkeps}. This requires some initial understanding of the vectors $\big( B_1(t),\ldots,B_k(t) \big)$ under the conditioning in the first step. The next result shows that at times $\ell+1$ and $r-1$ it is unlikely that the conditioned value of $B_1$ significantly exceeds a certain value. Not only will it be useful in the proof of Lemma \ref{lemellr} but also in the ultimate proof of Proposition \ref{propconcave}.

\begin{lemma}\label{lemconmax}
There exists a constant $c = c_k > 0$ such that for all $R>0$ and all $\ell\in I_1$, $r\in I_2$, $f\in \CMK$ and $(\bar{x},\bar{y})\in \XYfM$, for $t \in \big\{ \ell + 1, r -1 \big\}$,
\begin{equation*}
\wxylr \Big( B_1(t) > M + K(r - \ell ) + (k + R)(r-\ell)^{1/2} \Big\vert \ncf_{[\ell,\ell + 1] \cup [r-1,r]} \Big) < \exp \big\{ - c R^2 \big\}.
\end{equation*}
\end{lemma}

\begin{proof}
For $1 \leq j \leq k$, set
\begin{equation*}
x'_j = y_j' = \max \big\{ x_1, y_1 \big\}  + K(r - \ell) + \big(2k+ 1 - 2j \big)(r-\ell)^{1/2}.
\end{equation*}
These points have been selected so that
$x_j' \geq x_j$ and $y_j' \geq y_j$
for each $1 \leq j \leq k$ and so that
\begin{equation}\label{fmaxprime}
\sup_{t \in [\ell,r]} f(t) \leq x_k' - \big(r-\ell\big)^{1/2} = y_k' - \big(r-\ell\big)^{1/2}.
\end{equation}
This bound is due to the fact that $f \in \CMK$, $x_k \geq f(\ell)$ and $y_k \geq f(r)$, and is explained in Figure~\ref{figconmax}.

\begin{figure}
\begin{center}
\includegraphics[width=0.7\textwidth]{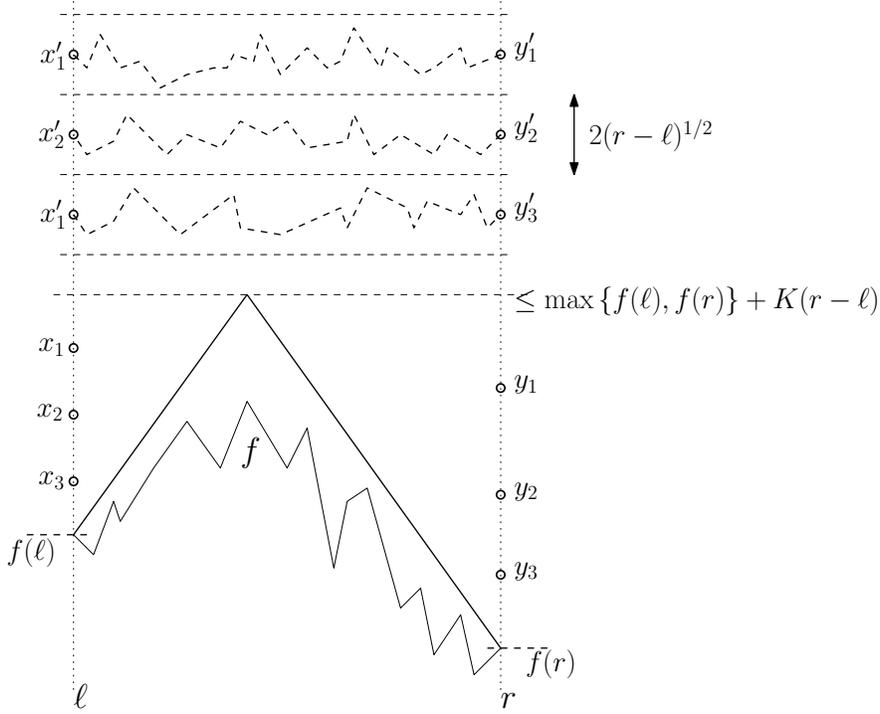}
\end{center}
\caption{Illustrating the proof of Lemma \ref{lemconmax} for the case $k=3$. The bold lines emanating from $\big(\ell,f(\ell)\big)$ and  $\big(r,f(r)\big)$ have slopes $K$ and $-K$, and thus lie above $f \in \CMK$. The lines' point of intersection, should it have $x$-coordinate in $[\ell,r]$, lies at most $K$ units above either of this pair of points. As such, the three dashed rectangular regions in the upper part of the figure are all disjoint from the curve $f$; each has positive probability of containing a Brownian bridge such as those drawn here.}\label{figconmax}
\end{figure}

We claim now that
\begin{equation}\label{eqnclowerbd}
 \wxylrprime \Big( \ncf_{[\ell,r]} \Big) \geq \Big( 1 - 2e^{-2} \Big)^k.
\end{equation}
To verify (\ref{eqnclowerbd}), note that (\ref{fmaxprime}) implies that if the curve $B_k$ does not leave the interval of width $2(r - \ell)^{1/2}$ centered at $x_k'=y_k'$, then it will not meet $f$. Similarly, if each curve $B_i:[\ell,r] \to \R$ does not leave the interval of width $2(r - \ell)^{1/2}$ centered at $x_i'=y_i'$, then mutual self-avoidance of these curves will be achieved. Calling $c$ the probability that the supremum of the modulus of standard Brownian bridge on $[0,r - \ell]$ does exceed $(r-\ell)^{1/2}$ (noter that $c$ is a constant by virtue of  Brownian scaling). The left-hand side of (\ref{eqnclowerbd}) is at least $(1-c)^k$. By Lemma~\ref{lembridgemod}, $c \leq 2 e^{-2}$ and thus we establish the right-hand side of (\ref{eqnclowerbd}).

We now apply the monotonicity Lemma \ref{lemmonotonetwo} with the choice $A = [\ell,\ell+1] \cup [r-1,r]$
to conclude the proof. Indeed, due to the coupling provided by that lemma,
\begin{eqnarray*}
 & & \wxylr \Big( B_1(\ell + 1) > M + K(r - \ell ) + (2k + R)(r-\ell)^{1/2} \Big\vert \ncf_A \Big) \nonumber \\
 & \leq & \wxylrprime \Big( B_1(\ell + 1) > M + K(r - \ell ) + (2k + R)(r-\ell)^{1/2} \Big\vert \ncf_A \Big). \nonumber
\end{eqnarray*}
However,
\begin{eqnarray}
& & \wxylrprime \Big( B_1(\ell + 1) > M +  K(r - \ell ) + (2k + R)(r-\ell)^{1/2} \Big\vert \ncf_A \Big) \nonumber \\
 & \leq &
\frac{\wxylrprime \Big( B_1(\ell + 1) > M + K(r - \ell ) + (2k + R)(r-\ell)^{1/2} \Big)}{\wxylrprime \big( \ncf_A \big)}
  \label{eqwlinetwo} \\
 & \leq &  \big( 1 - 2e^{-2} \big)^{-k} \exp \big\{ -2 k (R + 1)^2 \big\}, \nonumber
\end{eqnarray}
To go from the second line to the third involves bounding the numerator from above and the denominator from below. The denominator is bounded by noting that $\wxylrprime(\ncf_{A})\geq \wxylrprime(\ncf_{[\ell,r]})$ and then using (\ref{eqnclowerbd}). For the numerator, recall that under the measure $\wxylrprime$, the path $B_1:[\ell,r]\to \R$ satisfies $B_1(\ell)=x_1'$ and $B_1(r)=y_1'$. Noting that $x'_1 =y'_1  \leq M + K(r - \ell) + (2k - 1)(r-\ell)^{1/2}$,  the event in question may be realized only if the maximum value attained by $B_1$ exceeds its common endpoint value of $x_1'$ by at least $(R + 1)(r - \ell)^{1/2}$; hence, Lemma~\ref{lembridgemod} provides the bound on the numerator.
\end{proof}
The next lemma was one of the main pieces necessary for the proof of Lemma \ref{lemellplusone}.
\begin{lemma}\label{lemellr}
For all $\e > 0$, there exists $\delta=\delta(\e,K,M,I_1,I_2)>0$ such that, for all $\ell\in I_1$, $r\in I_2$, $f\in \CMK$, $(\bar{x},\bar{y})\in \XYfM$, and $t\in \{\ell+1,r-1\}$,
\begin{equation*}
\wxylr \Big( \min_{1 \leq i \leq k-1}\vert B_{i}(t) - B_{i+1}(t)\vert > \delta  \Big\vert \ncf_{[\ell,\ell+1] \cup [r-1,r]}  \Big) > 1 - \e.
\end{equation*}
\end{lemma}
\begin{figure}
\begin{center}
\includegraphics[width=0.7\textwidth]{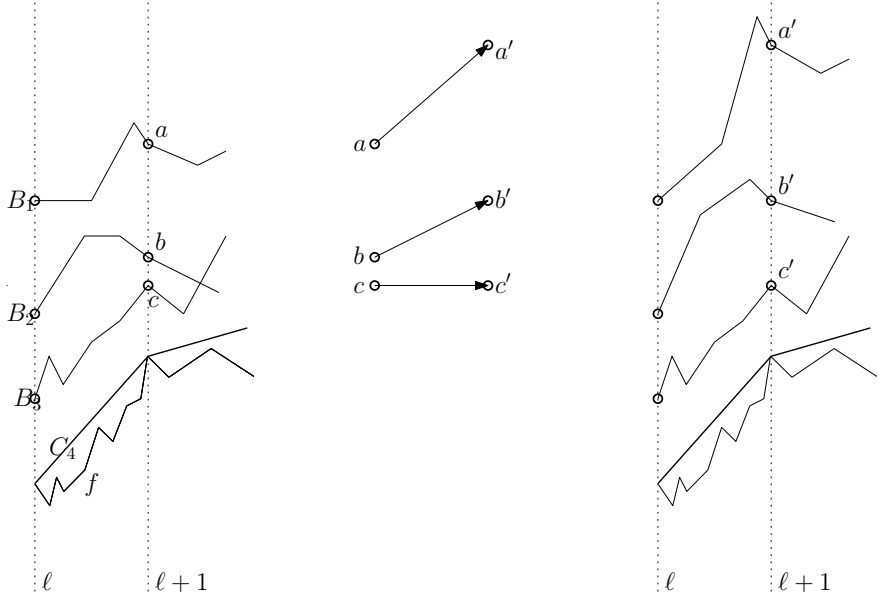} \\
\end{center}
\caption{Illustrating the proof of Lemma \ref{lemellr} in the case that $k=3$. Given the event $\nc_{[\ell,\ell+1] \cup [r-1,r]}$ and conditionally on all the data specifying the top three curves except for their values at $\ell + 1$, the scenario that two among these three values are very close to one another will be shown to be very unlikely.
We will establish this by showing that, for each possible way that two points in the triple may be very close, there is a much probable alternative where this is not the case. This alternative will be found by pushing some elements in the triple upwards, as in the depicted transformation $(a,b,c) \to (a',b',c')$; the pushing map will be chosen to dilate space, so that close-pair regions in the domain of the map are associated with regions in the range in such a way that the range is much more probable than the domain.}\label{figwssketch}
\end{figure}
The proof of Lemma \ref{lemellr} will need some further notation, and another result (Lemma~\ref{lemqn}), but the main idea of the proof is  straightforward to grasp (and is illustrated in Figure~\ref{figwssketch} for the case $t = \ell +1$). Using the decomposition in Lemma \ref{browniandecomp}, we will condition $\wxylr$ on all the information specifying the curves $B_i:[\ell,r] \to \R$, $1 \leq i \leq k$, except for the vector $\big(B_1(\ell+1),\ldots,B_k(\ell + 1)\big)$. The conditioned data generates a sigma-field which we call $\signc$.  This data consists of $k$ pairs of functions,
 $\tilde{B}^i_1:[0,1] \to \R$ and $\tilde{B}^i_2:[0,r-\ell - 1] \to \R$, $1 \leq i \leq k$, with each of these $2k$ functions equal to zero at each endpoint; $\tilde{B}^i_1$ and $\tilde{B}^i_2$
are the affine (and time-translated) shifts of $B_i$ on $[\ell,\ell+1]$ and on $[\ell+1,r]$ chosen so that the endpoint values vanish; as such, the data that remains obscure to a witness of $\signc$
is precisely the $k$-vector  $\big(B_1(\ell+1),\ldots,B_k(\ell + 1)\big)$.
Under $\wxylr$ conditionally on $\signc$, this $k$-vector is normally distributed. We are interested in whether the event $\ncf_{[\ell,\ell+1] \cup [r-1,r]}$ occurs.  There is a certain set (which we will call $\tilde{Q} \subseteq \R^k$ and formally define in the actual proof of Lemma \ref{lemellr}) of choices for this Gaussian $k$-vector which realize  $\ncf_{[\ell,\ell+1] \cup [r-1,r]}$; $\tilde{Q}$ is random and $\signc$-measurable. We will prove Lemma \ref{lemellr} by arguing that, for typical choices of the information recorded by $\signc$, the set $\tilde{Q}$ does not usually contain pairs of points that are extremely close. To establish this, we need a definition.
\begin{definition}\label{closed-disp}
Write $\Rklezero = \big\{ \bar{y} \in \R^k: y_1 > \ldots > y_k > 0 \big\}$, and note that $\Rklezero \subseteq \Rkle$. A subset $Q\subseteq \Rkle$ is {\it closed under $\Rklezero$-displacement} if whenever $\bar{x}\in Q$ and $\bar{y} \in \Rklezero$, then $\bar{x} + \bar{y}\in Q$ as well.
\end{definition}
Note that, given the bridge data $\tilde{B}^i_1$ and $\tilde{B}^i_2$ for $1\leq i\leq k$,
the set $\tilde{Q}$ of acceptable choices for the vector $\big(B_1(\ell+1),\ldots,B_k(\ell + 1)\big)$ is indeed closed under  $\Rklezero$-displacement. For this reason, the next lemma is the general statement about $k$-tuples of i.i.d. Gaussian random variables that we will invoke to show that clustering of the elements in $\tilde{Q}$ does not usually occur.
\begin{lemma}\label{lemqn}
Let $N=\{N_1,\ldots,N_k\}$ denote a Gaussian random variable with independent entries each of variance $\sigma^2$, and let $Q \subseteq \Rkle$ be closed under $\Rklezero$-displacement. For $\bar{a}\in \R^k_>$, let $\mu_{\bar{a},\sigma, Q}$ denote the conditional distribution of $N$ given that $\bar{a} + N \in Q$.

Fix $M>0$ and $c\in (0,1)$. Then for all $\e>0$, there exists $\delta= \delta(\e,M,c)$ such that, for  $\bar{a}\in \R^k_>$ satisfying $||\bar{a}||_{\infty}\leq M$,  $\sigma\in (c,c^{-1})$ and  $Q \subseteq \Rkle$ closed under $\Rklezero$-displacement,
\begin{equation}\label{mueqn}
\mu_{\bar{a},\sigma, Q} \Big\{ \bar{x} \in \R^k : \big\vert (x_i + a_i) - (x_j + a_j) \big\vert < \delta, x_1 \leq M  \Big\} < \e \textrm{ for each } 1 \leq i < j \leq k.
\end{equation}
\end{lemma}
\begin{proof}
For $j \in \{2,\ldots,k\}$ define $\phi_j: \Rkle \to \Rkle$ by
\begin{equation*}
\phi_j \big( x_1,\ldots,x_k \big) = \big( x_1 + \alpha, \ldots, x_{j-1} + \alpha,x_j,\ldots,x_k \big),
\end{equation*}
where $\alpha = 1 + \delta^{-1} \big((x_{j-1} + a_{j-1}) - (x_j + a_j) \big)$, with $\delta > 0$ to be specified.

Let $\nu$ denote the distribution of $\bar{a} + N$. Thus, for $B \subseteq Q$,
\begin{equation*}
\mu_{\bar{a},\sigma, Q}(B) = \nu(B)/\nu(Q).
\end{equation*}

For $1 \leq i < j \leq k$, set
\begin{equation*}
B^\delta_{i,j} = \Big\{ \bar{x} \in \R^k: \big\vert (x_i + a_i) - (x_j + a_j) \big\vert < \delta, x_1\leq M \Big\} \cap Q.
\end{equation*}
Then $\phi_j(B^\delta_{i,j}) \subseteq Q$. Hence,
\begin{equation}\label{eqmubeps}
\mu_{\bar{a},\sigma, Q}\big(B^\delta_{i,j}\big) \leq \frac{\nu(B^\delta_{i,j})}{\nu \big( \phi_j(B^\delta_{i,j}) \big)}.
\end{equation}
We estimate the right-hand side of (\ref{eqmubeps}) by finding a lower bound for $\nu(\phi_j(B^\delta_{i,j}))$ in terms of $\nu(B^\delta_{i,j})$. To do so, we write $p:\R^k \to [0,\infty)$ for the density of the Gaussian vector $N$, and note that
\begin{equation}\label{eqnu}
\nu\big( \phi_j(B^\delta_{i,j}) \big) = \int_{\phi_j(B^\delta_{i,j})} p(\bar{x} - \bar{a}) d^k \bar{x} = \int_{B^\delta_j}
  p \big( \phi_j(\bar{x}) - \bar{a} \big) J_{\bar{x}} d^k \bar{x},
\end{equation}
where $J_{\bar{x}} \geq 0$ is the Jacobian of $\phi_j:\R^k \to \R^k$ at $\bar{x}$. Note that $J_{\bar{x}} = 1+\delta^{-1}$.
Observe that
\begin{equation*}
 \frac{p \big( \phi_j(\bar{x}) - \bar{a} \big)}{p(\bar{x} - \bar{a})}
=  \prod_{l = 1}^{j-1} \exp \bigg\{ -\frac{1}{2\sigma^2} \bigg( \Big( x_l - a_l
 + 1 + \delta^{-1}\big( (x_i + a_i) - (x_j + a_j) \big) \Big)^2 - \big(x_l - a_l \big)^2 \bigg)\bigg\}.
\end{equation*}
Note also that, if $\bar{x} \in B^\delta_{i,j}$, then $x_l - a_l \leq 2M$ for each $1 \leq l \leq k$ and $\big\vert (x_i + a_i) - (x_j + a_j) \big\vert \leq \delta$. Hence, for such $\bar{x}$,
\begin{equation*}
 \frac{p \big( \phi_j(\bar{x}) - \bar{a} \big)}{p(\bar{x} - \bar{a})}
     \geq \exp \Big\{ - \tfrac{1}{2\sigma^2}(j-1)( 8 M + 4 ) \Big\}.
\end{equation*}
We find from (\ref{eqnu}) that
\begin{eqnarray*}
\nu\big( \phi_j(B^\delta_{i,j}) \big) & \geq &  (1+\delta^{-1}) \exp \Big\{ - \tfrac{1}{2\sigma^2}(j-1)( 8 M + 4 ) \Big\} \int_{B^\delta_j} p\big( \bar{x} - \bar{a} \big) d^n \bar{x} \\
& = & (1+ \delta^{-1}) \exp \Big\{ - \tfrac{1}{2\sigma^2}(j-1)( 8 M + 4 ) \Big\} \nu \big( B^\delta_j \big).
\end{eqnarray*}
Using (\ref{eqmubeps}), we obtain
\begin{equation*}
\mu_{\bar{a},\sigma, Q} \big( B^\delta_{i,j} \big) \leq \frac{1}{1+\delta^{-1}} \exp \Big\{ \tfrac{1}{2\sigma^2}(j-1)( 8 M + 4 ) \Big\}.
\end{equation*}
We see that, by choosing $\delta = \e \exp \Big\{ \tfrac{1}{2\sigma^2}(j-1)( 8 M + 4 ) \Big\}$, (\ref{mueqn}) holds as desired.\end{proof}
\begin{proof}[Proof of Lemma \ref{lemellr}.]
We will consider only the case $t = \ell + 1$ as that of $t=r-1$ follows similarly.
We begin by making precise the notion of conditioning on the data in the top $k$ curves except for their values at $\ell + 1$.
Recall that, under the law $\wxylr$, the curve $B_i:[\ell,r] \to \R$ is distributed as a Brownian bridge that travels from $B_i(\ell) = x_i$ to $B_i(r) = y_i$. Write $\tilde{B}_i:[0,r-\ell] \to \R$ for the affine shift of $B_i$ that is zero at the endpoints, given by the formula
\begin{equation*}
\tilde{B}_i(s) = B_i(s + \ell) - x_i -  s \frac{y_i - x_i}{r - \ell}.
\end{equation*}
We may apply the decomposition of Lemma \ref{browniandecomp} for $\tilde{B}_i$ with the choice $j = 2$ and $t_1  = 1$ (and $t_0=0$).  As such, we write
\begin{equation*}
\tilde{B}_i(s) = {\bf 1}_{s\leq 1}\Big( s \tilde{N}_i + \tilde{B}_{1}^i(s)\Big) + {\bf 1}_{s\geq 1}\Big(\frac{r-\ell-s}{r-\ell-1} \tilde{N}_i + \tilde{B}_{2}^i(s-1)\Big).
\end{equation*}
Here, $\tilde{N}_{i}$ are independent Gaussian random variables with
\begin{equation}\label{sigmaeqn}
\var(\tilde{N}_{i}) = \sigma^2 = \frac{r-\ell-1}{r-\ell},
\end{equation}
and the curves $\tilde{B}_{j}^{i}$, $j=1,2$ are independent standard Brownian bridges of respective durations $1$ and $r-\ell-1$.

Let $\signc$ denote the sigma-field generated by $\big\{\tilde{B}^i_j: 1\leq i\leq k, 1\leq j\leq 2\big\}$. Informally, the information in $\signc$ is the data described by $B_i:[\ell,r]\to \R$, $1\leq i\leq k$, lacking the data which determines $B_i(\ell+1)$, $1\leq i\leq k$.

Define the conditional measure $\wxylr\big( A \big\vert  \signc \big)$ for an event $A$ as the conditional expectation of ${\bf 1}_A$ with respect to the sigma-field $\signc$.
With respect to the conditional measure  $\wxylr\big( \cdot \big\vert  \signc \big)$,
define the $\signc$-measurable random subset $\tilde{Q} \subseteq \Rkle$ as the set of points $\bar{z} \in \Rkle$ such that
if $\big(\tilde{N}_1^1,\ldots,\tilde{N}_1^k \big) = \bar{z}$ then the event $\ncf_{[\ell,\ell+1] \cup [r-1,r]}$ occurs.
Note that almost surely with respect to  $\wxylr\big( \cdot \big\vert  \signc \big)$, $\tilde{Q}$ is closed under $\Rklezero$-displacement (Definition \ref{closed-disp}).

Let $\bar{a} \in \Rkle$ denote the successive values at $\ell + 1$ of the linear interpolations between $\big(\ell,B_i(\ell)\big)$ and $\big(r,B_i(r)\big)$ for $1 \leq i \leq k$:
\begin{equation}\label{eqaxyi}
a_i = x_i + \frac{y_i - x_i}{r - \ell}.
\end{equation}

These values are such that $\tilde{N}_i + a_i = B_{i}(\ell+1)$.

Let $\mu_{\bar{a},\sigma,\tilde{Q}}$ denote the measure defined in the statement of Lemma \ref{lemqn} with parameters specified by $\bar{a}$ in (\ref{eqaxyi}), $\sigma$ in (\ref{sigmaeqn}), and $\tilde{Q}$ defined above. Note that  $\mu_{\bar{a},\sigma,\tilde{Q}}$ is a $\signc$-measurable random measure that specifies the set of values for the Gaussian vector $\big( \tilde{N}_1,\ldots,\tilde{N}_k \big)$ which would cause the event $\ncf_{[\ell,\ell+1] \cup [r-1,r]}$ to occur.

To summarise, the distribution of $\big( B_1(\ell + 1),\ldots,B_k(\ell + 1) \big)$ under  $\wxylr\big( \cdot \big\vert \ncf_{[\ell,\ell+1] \cup [r-1,r]}  \big)$ may be constructed by a two-step procedure. First, the random measure
 $\mu_{\bar{a},\sigma,\tilde{Q}}$ is obtained by realizing the data that generates $\signc$; and then a random variable $Z$ having the law
 $\mu_{\bar{a},\sigma,\tilde{Q}}$ is constructed. The distribution  of $\big( B_1(\ell + 1),\ldots,B_k(\ell + 1) \big)$ under  $\wxylr\big( \cdot \big\vert \ncf_{[\ell,\ell+1] \cup [r-1,r]}  \big)$ has the law of
$Z+ \bar{a}$ averaged over the two steps.

We will apply Lemma~\ref{lemqn} to  $\mu_{\bar{a},\sigma,\tilde{Q}}$. We need to treat separately those parameter choices for this measure such that the measure assigns significant probability to $k$-vectors with high first component, because Lemma~\ref{lemqn} is not useful for such measures. To do this, we introduce a variable $M'\geq M$ to be fixed later.
Let $\Theta$ denote the $\signc$-measurable event
$$
\Theta =\Big\{\mu_{\bar{a},\sigma,\tilde{Q}} \big\{ x_1 > M' \big\} \leq \e \Big\}.
$$

Observe then that we can bound the probability of the event of interest in Lemma~\ref{lemellr} as follows:
\begin{equation}\label{eqwxy}
 \wxylr \Big( \min_{1 \leq i \leq k-1} \big\vert B_{i+1}(\ell + 1) - B_i(\ell + 1) \big\vert \leq \delta \, \Big\vert \, \nc^f_{[\ell,\ell + 1] \cup [r-1,r]} \Big)  \leq  A_1 + A_2,
\end{equation}
where
$$
A_1 = \wxylr \Big( \min_{1 \leq i \leq k-1} \big\vert B_{i+1}(\ell + 1) - B_i(\ell + 1) \big\vert \leq \delta , \Theta \, \Big\vert \, \nc^f_{[\ell,\ell + 1] \cup [r-1,r]} \Big)
$$
and
$$
A_2 =  \wxylr \Big( \Theta^c \, \Big\vert \,  \nc^f_{[\ell,\ell + 1] \cup [r-1,r]} \Big).
$$
Let us first bound $A_2$. Write $\mathcal{\tilde E}$ for the expectation operator of $\wxylr \big(  \cdot \big\vert  \nc^f_{[\ell,\ell + 1] \cup [r-1,r]} \big)$. Then $\mathcal{\tilde E}(\cdot \vert \signc)$ represents the conditional expectation with respect to the sigma-field $\signc$. As such, we may write
$$
\wxylr \Big(  \tilde{N}_1 > M' \, \Big\vert \,  \nc^f_{[\ell,\ell + 1] \cup [r-1,r]} \Big)
=  \mathcal{\tilde E}\big({\bf 1}_{\tilde{N}_1 >M'} \big)
= \mathcal{\tilde E} \Big( \mathcal{\tilde E}\big({\bf 1}_{\tilde{N}_1 >M'} \vert \signc \big) \Big).
$$
The random variable $\mathcal{\tilde E}\big({\bf 1}_{\tilde{N}_1 >M'} \big\vert \signc \big)$ is measurable with respect to $\signc$ and is bounded between $0$ and $1$. Multiplying it by the indicator function of $\Theta^{c}$ serves only to decrease the full expectation. However, from the definition of $\mu_{\bar{a},\sigma,\tilde{Q}}$ and of $\Theta$, we have that, $\wxylr \big(  \cdot \big\vert  \nc^f_{[\ell,\ell + 1] \cup [r-1,r]} \big)$ almost surely,
$$
{\bf 1}_{\Theta^{c}} \mathcal{\tilde E} \big( {\bf 1}_{\tilde{N}_1 >M'} \big\vert \signc \big) = {\bf 1}_{\Theta^{c}} \mu_{\bar{a},\sigma,\tilde{Q}}\big(x_1> M'\big).
$$
On the event $\Theta^c$, $\mu_{\bar{a},\sigma,\tilde{Q}}(x_1> M') \geq \e$ and hence
$$
\mathcal{\tilde E}\Big( {\bf 1}_{\Theta^{c}} \mathcal{\tilde E}\big({\bf 1}_{\tilde{N}_1 >M'}  \big\vert \signc \big) \Big) \geq \mathcal{\tilde E}\big({\bf 1}_{\Theta^{c}}\e\big) =\e A_2.
$$

On the other hand,
\begin{equation}\label{eqwxyr}
\wxylr \Big(  \tilde{N}_1 > M' \, \Big\vert \,  \nc^f_{[\ell,\ell + 1] \cup [r-1,r]} \Big) =  \wxylr \Big(  B_1(\ell + 1) > M' + a_1  \, \Big\vert \,  \nc^f_{[\ell,\ell + 1] \cup [r-1,r]} \Big).
\end{equation}

Recalling that the absolute value of each component of $\bar{x}$ and $\bar{y}$ is assumed to be at most $M$, it follows that the same is true of the components of $\bar{a}$.
Hence, Lemma \ref{lemconmax} implies that the right-hand side of (\ref{eqwxyr}) is at most $\e^2$, provided that $M'>M$ is chosen to be high enough (depending on $\e,K,M,I_1,I_2$). Hence, $\e A_2 \leq \e^2$ or $A_2\leq \e$.


Note that the bound on which we insist for $M'$ enters into the definition of the event $\Theta$,
so that this choice is relevant as we now turn to show that $A_1\leq \e$.
First observe that, in light of (\ref{sigmaeqn}) and the conditions placed on the intervals $I_1$ and $I_2$ by Proposition \ref{propconcave}, there exists a constant  $c\in (0,1)$ such that $\sigma\in(c,c^{-1})$ for all $\ell\in I_1$ and $r\in I_2$. Note also that $(\bar{x},\bar{y})\in \XYfM$ implies that $||\bar{a}||_{\infty}\leq M<M'$. Thus we may apply Lemma \ref{lemqn}, with the choice $Q = \tilde{Q}$ and with $\bar{a}$ and $\sigma$ as just specified, to find that, $\wxylr\big( \cdot \big\vert  \signc \big)$ almost surely,
\begin{equation*}
\mu_{\bar{a},\sigma, \tilde Q} \big\{ \bar{x} \in \Rkle : \vert  (x_i + a_i) - (x_j + a_j) \vert < \delta, x_1 \leq M'  \big\}
 \leq \e.
\end{equation*}
The choice of $\delta$ here depends on $\e, M',c$ and hence depends on the parameters $\e, K,M,I_1,I_2$ as necessary.
By the definition of $\Theta$, it follows that likewise that, $\wxylr\big( \cdot \big\vert  \signc \big)$ almost surely,
$$
\mu_{\bar{a},\sigma, \tilde Q} \big\{ \bar{x} \in \Rkle : \vert  (x_i + a_i) - (x_j + a_j) \vert < \delta \big\} \leq 2\e {\bf 1}_{\Theta} + {\bf 1}_{\Theta^c}.
$$
Summing over $1\leq i<j\leq k$ shows that, restricted to $\Theta$,
$$
\mu_{\bar{a},\sigma, \tilde Q} \big\{ \bar{x} \in \Rkle : \vert  (x_i + a_i) - (x_j + a_j) \vert < \delta \textrm{ for all }i\neq j \big\} \leq k^2 \e.
$$
We again appeal to conditional expectations:
$$
A_1 = \mathcal{\tilde E}\Bigg(\mathcal{\tilde E}\bigg({\bf 1}_\Theta {\bf 1} \Big\{ \min_{1 \leq i \leq k-1}\big\vert B_{i}(\ell + 1) - B_{i+1}(\ell + 1) \big\vert \leq \delta \Big\} \bigg\vert \signc\bigg)\Bigg).
$$
Note then that
\begin{eqnarray*}
& & \mathcal{\tilde E}\Big({\bf 1}_\Theta {\bf 1} \big\{ \min_{1 \leq i \leq k-1} \big\vert B_{i}(\ell + 1) - B_{i+1}(\ell + 1) \big\vert \leq \delta \big\} \Big\vert \signc\Big) \nonumber \\
& = & {\bf 1}_\Theta \mu_{\bar{a},\sigma, \tilde Q} \big\{ \bar{x} \in \Rkle : \big\vert  (x_i + a_i) - (x_j + a_j) \big\vert \leq \delta \textrm{ for all }i\neq j \big\} \leq {\bf 1}_\Theta k^2 \e. \nonumber
\end{eqnarray*}
The equality holds due to the relationship between  $\mu_{\bar{a},\sigma, \tilde Q}$ and
the distribution of $\big( B_1(\ell + 1),\ldots,B_k(\ell + 1) \big)$ under  $\wxylr\big( \cdot \big\vert \ncf_{[\ell,\ell+1] \cup [r-1,r]}  \big)$ explained in the third paragraph after (\ref{eqaxyi}).

In this way, we see that $A_1 \leq \mathcal{\tilde E}({\bf 1}_\Theta \e) \leq k^2 \e$.  Rescaling $\e$ to absorb the constants, we complete the proof of Lemma \ref{lemellr} in the case that $t = \ell + 1$; the case $t=r-1$ is analogous.
\end{proof}

The other component of the proof of Lemma \ref{lemellplusone} is the following result.
\begin{lemma}\label{lembkeps}
There exists an $\e_0=\e_0(K,I_1,I_2)$ such that for each $\ell\in I_1$, $r\in I_2$,
 $f\in \CMK$ and $(\bar{x},\bar{y})\in \XYfM$,
\begin{equation}\label{eqellb}
\wxylr \Big( B_{k}(\ell+1)>f(\ell)+K, B_{k}(r-1)>f(r)+K \, \Big\vert \, \ncf_{[\ell,\ell + 1] \cup [r-1,r]} \Big) > \e_0.
\end{equation}
\end{lemma}
\begin{proof}
Define $\bar{x}',\bar{y}' \in \Rkle$ as $x'_i = f(\ell) - 1/2 - (i-1)$ and $y'_i = f(r) - 1/2 - (i - 1)$ for
$1 \leq i \leq k$. Note that
$x_i \geq f(\ell)$ and
$y_i \geq f(r)$ for such $i$, because
$(\bar{x},\bar{y})\in
\XYfM$; hence, each $x'_i$ (or $y'_i$) is less than its $x$- (or $y$-)counterpart.
\begin{figure}
\begin{center}
\includegraphics[width=0.8\textwidth]{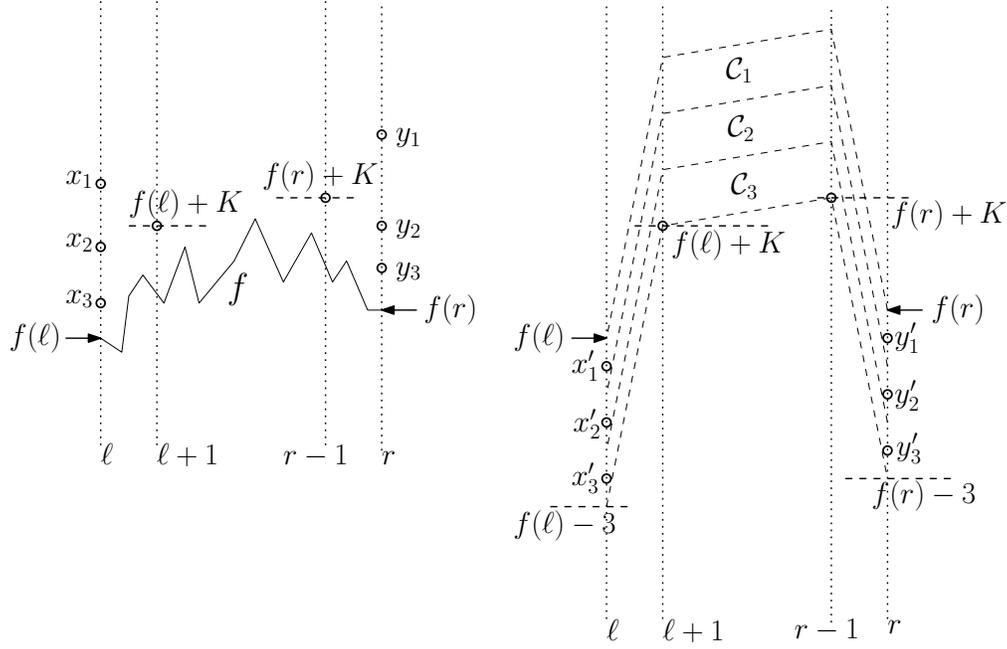} \\
\end{center}
\caption{In the case $k=3$ depicted, Lemma \ref{lembkeps} analyses
three independent Brownian bridges on $[\ell,r]$ from $x_i$ to $y_i$ conditioned to
avoid each other and the curve $f$; this boundary data is shown in the left-hand sketch. (Labels indicate $y$-coordinates except the four lowermost labels in each sketch.) Monotonicities reduce the problem to studying a system of such bridges from $x'_i$ to $y'_i$, as shown in the right-hand sketch; considering the event that these bridges remain in the respective corridors $\mathcal{C}_i$ then provides a lower bound on the probability in (\ref{eqwxi}).}\label{figcorridorsagain}
\end{figure}
By applications of the monotone couplings of Lemmas \ref{monotonicity} and \ref{lemmonotonetwo}, the probability in (\ref{eqellb})
will not rise if we replace the function $f$ in the conditioning by $-\infty$, nor if we consider the law $\wxylrprime$ in place of $\wxylr$. We must therefore bound below
$$
\wxylrprime \Big(  B_{k}(\ell+1)>f(\ell)+K, B_{k}(r-1)>f(r)+K \, \Big\vert \,  \nc_{[\ell,\ell + 1] \cup [r-1,r]}^{-\infty} \Big).
$$
We begin by noting that this quantity is at least
\begin{equation}\label{eqwxi}
\wxylrprime \Big(  B_{k}(\ell+1)>f(\ell)+K, B_{k}(r-1)>f(r)+K,  \nc_{[\ell,r]}^{-\infty} \Big).
\end{equation}
The desired lower bound for (\ref{eqwxi}) is now easily seen by consulting Figure~\ref{figcorridorsagain}.
The unit-width corridors $\mathcal{C}_i$, $1 \leq i \leq k$ have disjoint interiors, so that the event in (\ref{eqwxi})
will happen if each Brownian bridge
$B_i:[\ell,r] \to \R$, $B_i(\ell) = x'_i$, $B_i(r) = y'_i$ remains in the corridor $\mathcal{C}_i$. We claim that each bridge does so
with probability at least $c^{r - \ell}$, for some $c>0$ whose only dependence on the parameters is on $K$.
This follows readily by
decomposing  $B_i:[\ell,r] \to \R$ by splitting its domain $[\ell,\ell+1] \cup [\ell+1,r-1] \cup [r-1,r]$
(using Lemma~\ref{browniandecomp} with a coordinate change of the domain) and
applying a bound (from Lemma~\ref{lembridgemod}) on the maximum modulus of standard Brownian bridge; the argument relies on
each of the three slopes of either boundary curve of $\mathcal{C}_i$ being bounded above in absolute value by a constant which depends only on $K$. The first and the third of these slopes are $K$ and $-K$; the second, which is $\tfrac{f(r) - f(\ell)}{r - \ell - 2}$, is in absolute value at most $3K$,
because $\vert f(r) - f(\ell) \vert \leq K(r - \ell)$ is implied by $f \in \CMK$, while $r - \ell - 2 \geq 1$
by assumption on $I_1$ and $I_2$. This means that the probability of interest (\ref{eqwxi}) is at least $c_K^{k(r - \ell)}$.
The difference $r - \ell$ being at most a constant depending on the pair $(I_1,I_2)$, we see that (\ref{eqwxi}) is bounded below by some $\e_0>0$ having dependencies as desired in the statement of the lemma we are presently proving. This bound, in turn, proves the desired lower bound to complete the proof of this lemma.
\end{proof}

\subsubsection{Step 2} We now begin the second step of the groundwork for proving Proposition~\ref{propconcave}.
In the first step, we learnt in Lemma~\ref{lemellplusone} that there is a uniformly positive probability that the top $k$ curves at times $\ell + 1$ and $r-1$ are $\e$-well-spaced under the law  $\wxylr\big(\,\cdot\, \big\vert \ncf_{[\ell,\ell + 1] \cup [r-1,r]}\big)$.
In the second step, we will exploit  the as-yet-unused conditioning on $\nc_{[\ell + 1,r-1]}$ to improve this inference, finding that, under  $\wxylr\big(\,\cdot\, \big\vert \ncf_{[\ell,r]}\big)$,
these curves are in fact well-separated throughout $[\ell +1,r-1]$ with high probability. This result is achieved by first finding a lower bound (valid with high probability) on the acceptance probability of resampling the line ensemble on the interval $[\ell + 1, r-1]$.

\begin{definition}
Let $\bprime$ be the marginal distribution of $\wxylr\big(\,\cdot\, \big\vert \ncf_{[\ell,r]}\big)$ on $k$ curves from $B':[\ell,\ell+1]\cup [r-1,r]\to \R$ induced by setting $B' = B$ restricted to $[\ell,\ell+1]\cup[r-1,r]$. Likewise, let $\bstar$ be the marginal distribution of $\wxylr\big(\,\cdot\, \big\vert \ncf_{[\ell,\ell+1]\cup[r-1,r]}\big)$ on $k$ curves $B^*_i:[\ell,\ell+1]\cup [r-1,r]\to \R$.
\end{definition}

We prove Lemmas \ref{lemradnik} and \ref{lemfirststep} which together lead to Lemma \ref{eqheps} -- the main result of step 2.
\begin{lemma}\label{lemradnik}
Let $\bar{x},\bar{y} \in \R^k_>$ and let  $f:[\ell,r]\to \R$ be measurable with $x_{k}>f(\ell)$ and $y_k>f(r)$. Then $\bprime$ and $\bstar$ can both be considered as measures on the same probability space of continuous curves from $[\ell,\ell+1]\cup [r-1,r]\to \R$ and can be coupled in such a way that the Radon-Nikodym derivative of $\bprime$ with respect to $\bstar$ is given by
\begin{equation*}
\frac{d\bprime}{d\bstar}(\omega) = \zxyf^{-1} \AP\Big(\ell+1,r-1,\{B^{*}_{i}(\ell+1)\}_{i=1}^{k},\{B^{*}_{i}(r-1)\}_{i=1}^{k}, f \Big)(\omega).
\end{equation*}
The normalization constant $\zxyf$ is given by
\begin{equation*}
\zxyf = \int \AP\Big(\ell+1,r-1,\{B_{i}^{*}(\ell+1)\}_{i=1}^{k},\{B^{*}_{i}(r-1)\}_{i=1}^{k}, f \Big)(\omega) d\bstar(\omega).
\end{equation*}
\end{lemma}

\begin{proof}
This follows immediately from definitions.
\end{proof}

\begin{lemma}\label{lemfirststep}
There exists an $\e_0=\e_0(K,M) > 0$ such that, for all $\e<\e_0$ and all $\ell\in I_1$, $r\in I_2$, $f\in \CMK$ and $(\bar{x},\bar{y})\in \XYfM$,
\begin{equation}
 \bstar \bigg( \AP\Big(\ell+1,r-1,\{B^{*}_{i}(\ell+1)\}_{i=1}^{k},\{B^{*}_{i}(r-1)\}_{i=1}^{k}, f \Big) > \epsilon \bigg) > \epsilon.
\end{equation}
\end{lemma}
\begin{proof}
We claim that, on the event $\ws_\epsilon$,
\begin{equation}\label{eqacprob}
\AP\Big(\ell+1,r-1,\{B_{i}(\ell+1)\}_{i=1}^{k},\{B_{i}(r-1)\}_{i=1}^{k}, f \Big) > c^{k \epsilon^{-2}}
\end{equation}
for some constant $c \in (0,1)$.
Figure~\ref{figcorridors} illustrates the derivation of this claim.
To bound below the acceptance probability in (\ref{eqacprob}), we will construct a disjoint collection of  width-$\e$ corridors $\mathcal{D}_i$, $1 \leq i \leq k$, all of which lie above the curve $f$.
Set $E_i$ to be the event that $\mathcal{D}_i$ contains the range of an independent Brownian bridge starting at $\ell+1$ with value $B_i(\ell+1)$ and ending at $r-1$ with value $B_i(r-1)$, for all $1\leq i \leq k$. We will find a lower bound on the probability of $E_i$ valid for each $1 \leq i \leq k$.

To define the corridor $\mathcal{D}_i$, we consider two cases. For $1 \leq i \leq k$, the index $i$ is called {\it close} if $\vert B_i(\ell + 1) - B_i(r-1) \vert < K(r - \ell - 2)$. If an index $i$ is close, we will define the domain $\mathcal{D}_i$ to be a corridor with a single kink located so as to insure that it stays strictly above the function $f$; whereas if $i$ is not close, then $\mathcal{D}_i$ will be an unkinked corridor.

For $1 \leq i \leq k$, we define the corridor-centre function $l_i:[\ell + 1,r] \to \R$.
If index $i$ is close, set
$l_i(t) = \min \big\{ B_i(\ell + 1) + K(t- \ell +1), B_i(r-1) + K(r-1 - t) \big\}$ for
$t \in [\ell + 1,r-1]$,
and note that the range of $l_i$ is a polygonal path comprising two segments in this case.
If index $i$ is not close, then set $l_i(t) =  B_i(\ell + 1) + \tfrac{B_i(r-1) - B_i(\ell + 1)}{r - \ell - 2}(t- \ell  - 1)$.

In either case, set $\mathcal{D}_i \subseteq \R^2$,
$$
 \mathcal{D}_i = \Big\{ (x,y) \in \R^2: \ell + 1\leq x \leq r-1, \, l_i(x) -\e/2  \leq y \leq l_i(x) + \e/2  \Big\}.
$$

It is readily confirmed that, for $\e < 2$, the occurrence of $\ws_\e$ ensures that the $k$ corridors $\mathcal{D}_i$ are pairwise disjoint and that their union does not meet the curve $f$. That $f$ is avoided is insured by the bound of $K$ on the convex majorant of $f$ and the kinking of corridors associated with close $i$'s.

If $i$ is not close, then the probability of $E_i$ is simply the probability that the maximum modulus of a standard Brownian bridge of duration $r - \ell - 2$ is at most $\e$.
If $i$ is close then, writing $t \in [\ell +1,r-1]$ for the $x$-coordinate of the kink in the path $l_i$,
this lower bound may be expressed using the Brownian decomposition
Lemma~\ref{browniandecomp} in terms of a normal random variable of variance $\tfrac{(t-\ell-1)(r-1-t)}{r-\ell -2}$ and two Brownian bridges of durations $t - \ell - 1$ and $r - 1 -t$. Lemma~\ref{lembridgemod} is used to control the deviation of these bridges from linear motion. In either case, we find that the probability of $E_i$ is at least $c^{\epsilon^{-2}}$, with $c=c(K,M,I_1,I_2) \in (0,1)$.

\begin{figure}
\begin{center}
\includegraphics[width=0.7\textwidth]{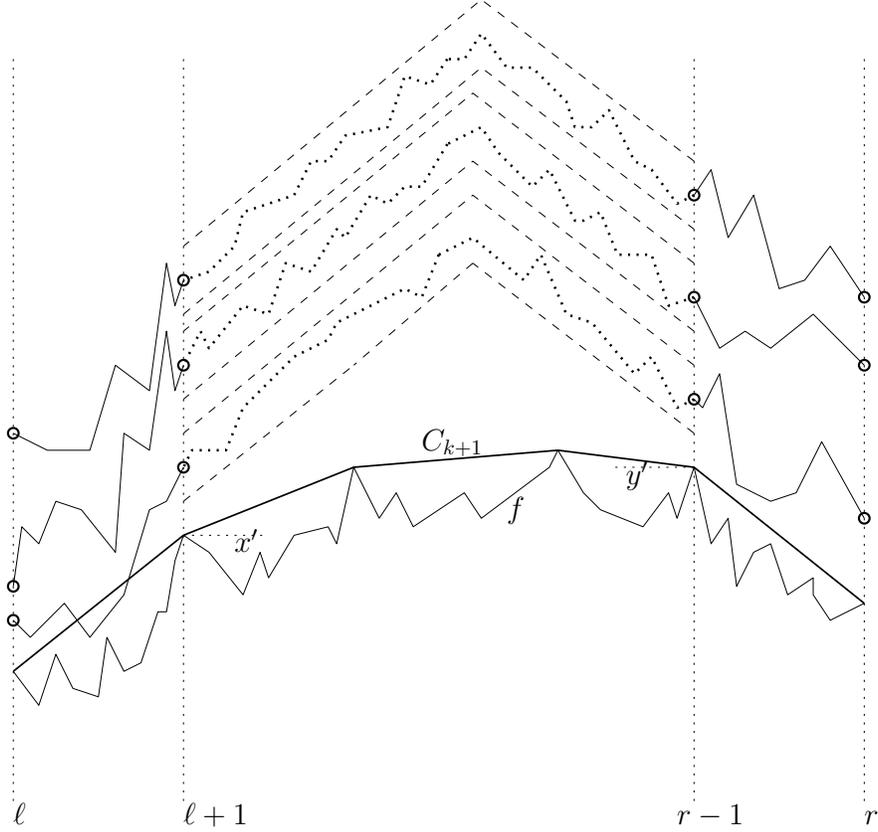} \\
\end{center}
\caption{Illustrating the proof of Lemma \ref{lemfirststep} in the case that $k=3$, and each index $1 \leq i \leq 3$ is close. The bold lines indicate curves realizing the event $\ws_\epsilon$. The three corridors are of width $\epsilon$, and, with conditional probability at least  $c^{k \epsilon^{-2}}$, each will contain a Brownian bridge between the midpoints of the intervals that form the ends of each corridor, as illustrated by the dotted curves in the figure.}\label{figcorridors}
\end{figure}

 To complete the argument, observe that, by Lemma \ref{lemellplusone} and the definition of $\bstar$, there exists $\e_0 > 0$ such that for all $0<\e<\e_0$, $\bstar(\ws_\epsilon)>\e$. This implies that
\begin{equation*}
 \bstar \bigg( \AP\Big(\ell+1,r-1,\{B^{*}_{i}(\ell+1)\}_{i=1}^{k},\{B^{*}_{i}(r-1)\}_{i=1}^{k}, f \Big) > c^{k\e^{-2}} \bigg) > \epsilon
\end{equation*}
and hence completes the proof.
\end{proof}

We now state the main result of step 2.
\begin{lemma}\label{eqheps}
There exists an $\e_0=\e_0(K,M) > 0$ such that for all $0<\e<\e_0$ and all $\ell\in I_1$, $r\in I_2$,  $f\in \CMK$ and $(\bar{x},\bar{y})\in \XYfM$,
\begin{equation}\label{eqbxyf}
\bxyflr \bigg( \AP\Big(\ell+1,r-1,\{B_{i}(\ell+1)\}_{i=1}^{k},\{B_{i}(r-1)\}_{i=1}^{k}, f \Big) > \epsilon^3 \bigg) > 1 - \epsilon.
\end{equation}
\end{lemma}
\begin{proof}
The probability in question is the same as under $\bxyf'$, because the acceptance probability in question is measurable with respect to the curves $B_i:[\ell,\ell+1] \cup [r-1,r] \to \R$, $1 \leq i \leq k$. We use Lemma~\ref{lemradnik} to find that
\begin{eqnarray*}
&&\bxyflr \Big( \AP\big(\ell+1,r-1,\{B_{i}(\ell+1)\}_{i=1}^{k},\{B_{i}(r-1)\}_{i=1}^{k}, f \big) \leq \epsilon^3 \Big)\\
&\leq& \zxyf^{-1} \epsilon^3 \bstar\Big( \AP\big(\ell+1,r-1,\{B^{*}_{i}(\ell+1)\}_{i=1}^{k},\{B^{*}_{i}(r-1)\}_{i=1}^{k}, f \big) \leq \epsilon^3 \Big),
\end{eqnarray*}
where
\begin{equation}\label{eqzxyf}
\zxyf =  \int   \AP\Big(\ell+1,r-1,\{B^{*}_{i}(\ell +1)\}_{i=1}^{k},\{B^{*}_{i}(r -1)\}_{i=1}^{k}, f \Big) d  \bstar(\omega) \geq \epsilon^2.
\end{equation}

The inequality in (\ref{eqzxyf}) is due to Lemma \ref{lemfirststep}. Hence, as desired,
$$
\bxyflr \Big( \AP\big(\ell+1,r-1,\{B_{i}(\ell)\}_{i=1}^{k},\{B_{i}(r)\}_{i=1}^{k}, f \big) \leq \epsilon^3 \Big) \leq \epsilon.
$$
\end{proof}

We may now present the proof of Proposition \ref{propconcave} which, when shown, completes the proof of Proposition \ref{propacceptprob}.

\begin{proof}[Proof of Proposition \ref{propconcave}]
Choosing $\delta=\e^{3}$ in Lemma \ref{eqheps}, we obtain the assertion (\ref{APeqn}) regarding acceptance probabilities.

Obtaining the assertion (\ref{GAPeqn}) about the minimal gap requires a little more effort but is closely related to (\ref{APeqn}).
Roughly speaking, we wish to show that, if the minimal gap becomes very small, so must the acceptance probability, so as to obtain (\ref{GAPeqn}) from (\ref{APeqn}). To argue this, we will use Lemma~\ref{lembrownbridge} to bound hitting probabilities for a Brownian bridge.

Note that the law $\bxyflr$ can be formed in two steps: first by realizing its marginal $\bprime$ on
$[\ell,\ell + 1] \cup [r-1,r]$, and second by realizing its conditional marginal on $[\ell +1,r-1]$.
For $M' > 0$, define the event $\aone$ (in terms of data available after the first step -- i.e., the sigma-field generated by the projection onto $[\ell,\ell + 1] \cup [r-1,r]$):
$$
\aone  = \Big\{ \AP\big(\ell+1,r-1,\{B_{i}(\ell+1)\}_{i=1}^{k},\{B_{i}(r-1)\}_{i=1}^{k}, f \big) > \epsilon^3 \Big\}
 \cap \bigcap_{t \in \{ \ell +1,r-1 \}, 1 \leq i \leq k}
\Big\{ \vert B_i(t) \vert  \leq M' \Big\}.
$$

Recalling that we have assumed that $\ell\in I_1$, $r \in I_2$ (with $\sup I_1< \inf I_2 + 3$), $f\in \CMK$ and $(\bar{x},\bar{y})\in \XYfM$,  note that Lemmas  \ref{lemconmax} and \ref{eqheps} imply that, for any given $\e > 0$, there exists  $M'=M'(\e,K,M,I_1,I_2) > 0$ large enough such that
\begin{equation}\label{eqprobaone}
\bxyflr \big(  \aone  \big) > 1 - 2\epsilon.
\end{equation}
In fact, Lemma  \ref{lemconmax} only bounds the probability of $\{B_i(t)\geq M'\}$. We claim that the probability of the event $\{B_i(t)\leq -M'\}$ can easily be bounded as going to zero as $M'$ increases by appealing to the monotonicity result of Lemma \ref{monotonicity}. Indeed, if the function $f$ were replaced by $f\equiv -\infty$, the claim would be immediately true, and by monotonicity, it is likewise true for any $f\geq -\infty$.

Let $m = (\ell + r)/2$ denote the midpoint of the interval $[\ell + 1,r-1]$. We now adopt the shorthand that, on the probability space on which $\bxyflr$ is defined, $\bar{x}'$ and $\bar{y}'$
denote $\big(B_1(\ell +1),\ldots,B_k(\ell +1) \big)$ and  $\big(B_1(r -1),\ldots,B_k(r - 1) \big)$.
Suppose now that $\bar{x}'$ and $\bar{y}'$ are such that $\aone$ is satisfied. Let $\delta > 0$ be a parameter whose value will be set later. For each $i\in \{1,\ldots k-1\}$, let $\chi_i \in [\ell +1,m]$ denote the infimum of those $t \in [\ell + 1,m]$ such that $B_i(t) - B_{i+1}(t) \leq \delta$; if no such point exists, we formally set $\chi_i = \infty$. Then on the event $\{\chi_i<\infty\}$, $(\chi_i,m)$ is a stopping domain (for the top $k$ curves) under the independent Brownian bridge law $\mathcal{W}_{k;\bar{x}',\bar{y}'}^{\ell +1,r-1}$ on the curves $\{B_i\}_{i=1}^{k}$ . Observe that, given $\chi_i < \infty$ and the trajectories $\big\{ B_i(t), B_{i+1}(t): \ell+1 \leq t \leq \chi \big\}$, the conditional probability that $B_i(t) > B_{i+1}(t)$ for all $t \in [\chi_i,r-1]$, is equal to
\begin{equation}\label{eqwdelta}
\mathcal{W}_{1;\delta,y'_i - y'_{i+1}}^{0,2(r - 1 - \chi_i)} \Big( B(s) > 0 \, \, \forall \, \,  s \in \big[0,2(r - 1 - \chi_i)\big]  \Big).
\end{equation}
This expression naturally arises by considering the difference $B_i - B_{i+1}$; the factor of two appearing in the expression for the duration arises because taking the difference of Brownian motions doubles the variance of the process.

Noting that $y_i' - y_{i+1}' \leq 2M'$ (as we assumed $y'$ is such that $A$ can be satisfied) and $r-1 - \chi_i \geq (r - \ell)/2 \geq 1/2$ by assumption,
an evident stochastic domination shows that the probability (\ref{eqwdelta}) is at most
$$
\mathcal{W}_{1;\delta,2M'}^{0,1} \Big( B(s) > 0 \, \, \forall \, \,  s \in \big[0,1\big]  \Big).
$$
Lemma~\ref{lembrownbridge} implies that this quantity is at most $8 \pi^{-1/2} \big( \delta M' \big)^{1/2}$.

Combining these bounds and invoking the strong Gibbs property of Lemma \ref{stronggibbslemma}, we learn that
$$
\mathcal{W}_{k;\bar{x}',\bar{y}'}^{\ell+1,r-1} \big( \nc_{\ell +1,r-1} \big\vert \chi_i < \infty \big) \leq 8 \pi^{-1/2} \big( \delta M' \big)^{1/2}.
$$
By the definition of $\chi_i$, this means that
\begin{equation}\label{eqwxyone}
\mathcal{W}_{k;\bar{x}',\bar{y}'}^{\ell + 1,r-1} \Big( \ncf_{\ell +1,r-1} \Big\vert
 \, \exists \, t \in [\ell +1,m] \textrm{ such that $B_i(t) - B_{i+1}(t) \leq \delta$} \Big) \leq  8 \pi^{-1/2} \big( \delta M' \big)^{1/2}.
\end{equation}
Similarly, we find that
\begin{equation}\label{eqwxytwo}
\mathcal{W}_{k;\bar{x}',\bar{y}'}^{\ell+1,r-1} \Big( \ncf_{\ell +1,r-1} \Big\vert \,
\exists \, t \in [m,r-1] \textrm{ such that $B_i(t) - B_{i+1}(t) \leq \delta$} \Big) \leq  8 \pi^{-1/2} \big( \delta M' \big)^{1/2}.
\end{equation}
Note that if $A,B,C$ are three events in a probability space $(\Omega,\mathbb{P})$ such that
$\Omega = A \cup B$ and
$\max \big\{ \mathbb{P} \big( C \big\vert A \big),  \mathbb{P} \big( C \big\vert B \big) \big\} \leq \phi$, then
$\mathbb{P} \big( C \big) \leq 2 \phi$. This fact may be employed in concert with (\ref{eqwxyone}) and (\ref{eqwxytwo}) to yield that
\begin{equation*}
\mathcal{W}_{k;\bar{x}',\bar{y}'}^{\ell+1,r-1} \Big( \ncf_{\ell +1,r-1} \Big\vert \,
\exists \,  t \in [\ell + 1,r-1] \textrm{ such that $B_i(t) - B_{i+1}(t) \leq \delta$} \Big) \leq 16 \pi^{-1/2} \big( \delta M' \big)^{1/2}.
\end{equation*}
On the event $\aone$, we have that (from the acceptance probability bound)
\begin{equation*}
\mathcal{W}_{k;\bar{x}',\bar{y}'}^{\ell+1,r-1} \big( \ncf_{\ell +1,r-1} \big) \geq \epsilon^3.
\end{equation*}
Hence,
\begin{equation*}
\mathcal{W}_{k;\bar{x}',\bar{y}'}^{\ell+1,r-1} \Big(
\, \exists \, t \in [\ell + 1,r-1] \textrm{ such that $B_i(t) - B_{i+1}(t) \leq \delta$}
 \, \Big\vert \,
\ncf_{[\ell +1,r-1]}
\Big) \leq  16 \pi^{-1/2} \big( \delta M' \big)^{1/2} \epsilon^{-3}.
\end{equation*}
Taking $\delta = (M')^{-1} \tfrac{\epsilon^8}{2^8 \pi (k-1)^2}$
 and summing over $i \in \{ 1,\ldots,k-1\}$, we conclude that, should $\aone$ occur,
\begin{equation}\label{eqwxyprime}
\mathcal{W}_{k;\bar{x}',\bar{y}'}^{\ell+1,r-1} \Big(
 \min_{1 \leq i \leq k-1}\inf_{s \in [\ell + 1,r-1]} \big( B_i(s) - B_{i+1}(s) \big)  \leq   (M')^{-1} \tfrac{\epsilon^8}{2^8 \pi (k-1)^2}
 \, \Big\vert \,
\ncf_{[\ell +1,r-1]}
\Big) \leq \epsilon.
\end{equation}

Let us conclude the proof now. The law $\bxyflr$ can be realized by first sampling from its marginal $\bxyf'$ on $[\ell,\ell+1]\cup [r-1,r]$ and then sampling the remaining path on $[\ell+1,r-1]$ from the measure $\mathcal{W}_{k;\bar{x}',\bar{y}'}^{\ell+1,r-1}\big(\cdot \big\vert \ncf_{[\ell +1,r-1]}\big)$ where $\bar{x}'$ and $\bar{y}'$ are determined from the first step as $\bar{x}' = \big(B_1(\ell+1),\ldots, B_k(\ell+1)\big)$ and $\bar{y}'=\big(B_1(r-1),\ldots, B_k(r-1)\big)$.

By (\ref{eqprobaone}), the event $A$ fails to occur under $\bxyf'$ with probability at most $2\e$. On the event $A$ we may appeal to (\ref{eqwxyprime}). Combining these two facts yields the bound
\begin{equation*}
\bxyflr \bigg( \min_{1 \leq i \leq k -1}\inf_{s \in [\ell + 1,r-1]} \big( B_i(s) - B_{i+1}(s) \big)  \leq   (M')^{-1} \frac{\epsilon^8}{2^8 \pi (k-1)^2}   \bigg) \leq 3 \epsilon.
\end{equation*}
We have obtained (\ref{GAPeqn}) and have thus completed the proof of Proposition~\ref{propconcave}.
\end{proof}

\section{Proof of monotonicity and strong Gibbs property}\label{proofslemmas}

The proofs of the monotonicity lemmas are based on an analogous result for non-intersecting random walks which can be proved by coupling Markov chains. The proof of the strong Gibbs property is a fairly straightforward extension of the usual proof of the strong Markov property.

\begin{proof}[Proof of Lemmas \ref{monotonicity} and \ref{lemmonotonetwo}]
The approach used in proving these two lemmas is identical so we will demonstrate it in the case of Lemma \ref{monotonicity} and briefly explain how it applies equally well to Lemma~\ref{lemmonotonetwo}. We demonstrate this coupling for Brownian motion by exhibiting an analogous coupling of constrained simple symmetric random walks which then converge to the Brownian motions. The result for the random walks follows from a coupling of Monte-Carlo Markov chains which converge to the constrained path measures and hence demonstrates the coupling.

Consider $\bar{x}^n=(x_1^n\ldots, x_k^n)$ and $\bar{y}^n=(y_1^n\ldots, y_k^n)$ such that $x_{i}^n, y_{i}^n \in n^{-1}(2\Z)$ for all $1\leq i\leq k$ and $n>1$ and such that $x_{i}^n \to x_i$ as well as $y_i^n\to y_i$. Write $\wxy^{a,b;n}$ for the law of $k$ independent random walks
$X^n_i:[a,b] \to n^{-1} \Z$, $1 \leq i \leq k$, that satisfy $X^n_i(a) = x_i^n$ and $X^n_i(b) = y_i^n$ and take steps of size $n^{-1}$ in discrete time increments of size $n^{-2}$.
Define $\PP^{k;n}_f$ as the law of $\wxy^{a,b;n} \big( \cdot \big\vert \nc^{f}_{[a,b]} \big)$ (i.e., the probability distribution of $k$ random walk bridges $X^{f;n}=\{X^{f;n}_i\}_{i=1}^k$,  $X^{f;n}_i:[a,b]\rightarrow n^{-1}\Z$ with $X^{f;n}_i(a)=x_i, X^{f;n}_i(b)=y_i$ conditioned on the event of non-intersection with each other and the barrier $f$).
We will now demonstrate that if $f\leq g$ then there is a coupling measure $\PP^{k;n}_{f,g}$ on which $\{X^{f;n}_i\}_{i=1}^{k}$ and $\{X^{g;n}_i\}_{i=1}^{k}$ are defined with marginal distributions $\PP^{k;n}_f$ and $\PP^{k;n}_g$ under which almost surely $X^{f;n}_i(s)\leq X^{g;n}_i(s)$ for all $i$ and all $s\in [a,b]$. As $n \to \infty$, the invariance principle for Brownian bridge shows that this coupling measure converges to the desired coupling of Brownian bridges necessary to complete our proof.

Thus it remains to demonstrate the existence of $\PP^{k;n}_{f,g}$. In order to do this, we introduce a
continuous-time Markov chain dynamic on the random walk bridges $\{X^{f;n}_i\}_{i=1}^{k}$ and $\{X^{g;n}_i\}_{i=1}^{k}$. Let us write the collection of random walk bridges at time $t$ as $(X^{f;n})_t$ and $(X^{g;n})_t$. The time $0$ configurations of $(X^{f;n})_0$ and $(X^{g;n})_0$ are chosen to be the lowest possible trajectories of random walk bridges conditioned to touch neither each other nor the barriers $f$ and $g$ (respectively). It is clear that such a lowest collection of paths exists and is unique and that these initial trajectories are ordered $(X^{f;n}_i(s))_0 \leq (X^{g;n}_i(s))_0$ for all  $s\in [a,b]$ and $i\in \{1,\ldots, k\}$. The dynamics of the Markov chain are as follows: for each $s\in n^{-2}\Z [a,b]$, each $i\in \{1,\ldots, k\}$ and each $m\in \{+,-\}$ ($+$ stands for up flip and $-$ for down flip), there are independent exponential clocks which ring at rate one. When the clock labeled $(s,i,m)$ rings, one attempts to change $X^{f;n}_i(s)$ to $X^{f;n}_i(s)+2m n^{-1}$. This is the only change at that instant attempted, and it is only successful if the change can be made without disturbing the condition of non-intersection of the random walks with themselves and $f$. Likewise, according to the {\it same} clock, one attempts to change $X^{g;n}_i(s)$ to $X^{g;n}_i(s)+2mn^{-1}$, and the same conditions apply.

The first key fact (which is readily checked) is that these time dynamics preserve ordering, so that, for all $t\geq 0$, $(X^{f;n}_i(s))_t \leq (X^{g;n}_i(s))_t$ for all  $s\in [a,b]$ and $i\in \{1,\ldots, k\}$. The second key fact is that the marginal distributions of these time dynamics converge to the invariant measure for this Markov chain, which is given by the measures we have denoted by $\PP^{k;n}_{f}$ and $\PP^{k;n}_{g}$. This fact follows since we have a finite state Markov chain which is irreducible with obvious invariant measure. Combining these two facts implies the existence of the coupling measure $\PP^{k;n}_{f,g}$ as desired.

The proof of Lemma \ref{lemmonotonetwo} also relies on the same Markov chain which preserves ordering. The ordering of the initial data is a direct consequence of the ordering of the starting and ending points, and hence the same argument carries over.
\end{proof}

\begin{proof}[Proof of Lemma \ref{stronggibbslemma}]
The proof proceeds along the same lines as the proof of the strong Markov property for Brownian motion (see Section 3 of Chapter 8 of \cite{durrett}).

The right-hand side of equation (\ref{strongeqn}) is $\mathcal{F}_{ext}(k,\ell,r)$-measurable. Consider functionals $H$ which map line ensembles $\mathcal{L}$ to $\R$. To prove our desired result it suffices to show that for all $H$ which are Borel measurable with respect to $\mathcal{F}_{ext}(k,\mathfrak{l},\mathfrak{r})$,
\begin{equation}\label{stareqn}
\EE\Big[H(\mathcal{L}_1,\ldots,\mathcal{L}_N) \cdot F \big( \mathfrak{l},\mathfrak{r},\mathcal{L}_1\big\vert_{(\mathfrak{l},\mathfrak{r})},\cdots, \mathcal{L}_k\big\vert_{(\mathfrak{l},\mathfrak{r})} \big) \Big] = \EE\Big[H(\mathcal{L}_1,\ldots, \mathcal{L}_N)\cdot \bxyf^{\mathfrak{l},\mathfrak{r}}\big[F(\mathfrak{l},\mathfrak{r},B_1,\ldots, B_k)\big]\Big].
\end{equation}

By the Monotone Class theorem, it suffices to check equation (\ref{stareqn}) for functions $F\in bC^k$ of the form
\begin{equation*}
F\big((\ell,r,f_1,\ldots, f_k)\big) := \prod_{j=1}^{n} g_j(\ell,r,f_{i_{j}}(t_j)),
\end{equation*}
where $(\ell,r,f_1,\ldots, f_k)\in C^k$ and $g_j:\R^3\to \R$ are bounded continuous functions, $i_{j}\in \{1,\ldots,k\}$, and $t_j\in [a,b]$.

We will also make use of the continuity of expectations with respect to boundary conditions. In particular, consider $f:[a,b]\to \R$ fixed and $\bar{x}_n$, $\bar{y}_n$, $\ell_n$ and $r_n$ converging to $\bar{x}$, $\bar{y}$, $\ell$ and $r$. Then for $F\in bC^k$,
\begin{equation}\label{continuityeqn}
\lim_{n\to \infty} \mathcal{B}_{\bar{x}_n,\bar{y}_n,f}^{\ell_n,r_n}\big[F(\ell_n,r_n,B_1,\ldots,B_k)\big] = \mathcal{B}_{\bar{x},\bar{y},f}^{\ell,r}\big[F(\ell,r,B_1,\ldots,B_k)\big].
\end{equation}

We will approximate the stopping domain $(\mathfrak{l},\mathfrak{r})$ via a sequence of stopping domains $(\mathfrak{l}_n,\mathfrak{r}_n)$ which are defined by
$\mathfrak{l}_n = (j+1)2^{-n}$ where $j$ is such that $j2^{-n} \leq \mathfrak{l} < (j+1)2^{-n}$ (if no such $j$ exists then $\mathfrak{l}_n=-\infty$); and $\mathfrak{r}_n =j2^{-n}$ where $j2^{-n} < \mathfrak{r} \leq (j+1)2^{-n}$ (if no such $j$ exists then $\mathfrak{r}_n=\infty$).
It is clear that $(\mathfrak{l}_n,\mathfrak{r}_n)$ form stopping domains which converge to $(\mathfrak{l},\mathfrak{r})$.

Let $A=\{-\infty<\mathfrak{l}\leq \mathfrak{r}<\infty\}$, which coincides with the intersection of the events $\{-\infty<\mathfrak{l}_n\leq \mathfrak{r}_n<\infty\}$ over all $n\geq 1$.

On the event $A$, from the convergence of the stopping domains and the continuity of the $\mathcal{L}_i$, (\ref{continuityeqn}) holds with $\ell, r$ replaced by $\mathfrak{l},\mathfrak{r}$, and with $\bar{x}_n=\{\mathcal{L}_i(\mathfrak{l}_n)\}_{i=1}^k$, $\bar{y}_n=\{\mathcal{L}_i(\mathfrak{r}_n)\}_{i=1}^k$ and $f(\cdot) = \mathcal{L}_{k+1}(\cdot)$ or $-\infty$ if $k=N$.


Using the above deduction as well as the bounded convergence theorem, we then have that

\begin{eqnarray*}
& & \EE\Big[H(\mathcal{L}_1,\ldots, \mathcal{L}_N)\cdot \bxyf^{\mathfrak{l},\mathfrak{r}}\big[F(\mathfrak{l},\mathfrak{r},B_1,\ldots, B_k)\big]\cdot{\bf 1}_{A}\Big]\\
&=& \EE\Big[H(\mathcal{L}_1,\ldots, \mathcal{L}_N)\cdot \lim_{n\to \infty}  \mathcal{B}_{\bar{x}_n,\bar{y}_n,f}^{\mathfrak{l}_n,\mathfrak{r}_n}\big[F(\mathfrak{l}_n,\mathfrak{r}_n,B_1,\ldots, B_k)\big]\cdot {\bf 1}_A\Big]\\
&=& \lim_{n\to \infty} \EE\Big[H(\mathcal{L}_1,\ldots, \mathcal{L}_N)\cdot \mathcal{B}_{\bar{x}_n,\bar{y}_n,f}^{\mathfrak{l}_n,\mathfrak{r}_n}\big[F(\mathfrak{l}_n,\mathfrak{r}_n,B_1,\ldots, B_k)\big]\cdot{\bf 1}_A\Big].
\end{eqnarray*}

By the law of total probability and Fubini's theorem,
\begin{eqnarray*}
 & & \EE\Big[H(\mathcal{L}_1,\ldots, \mathcal{L}_N)\cdot \mathcal{B}_{\bar{x}_n,\bar{y}_n,f}^{\mathfrak{l}_n,\mathfrak{r}_n}\big[F(\mathfrak{l}_n,\mathfrak{r}_n,B_1,\ldots, B_k)\big]\cdot {\bf 1}_A\Big]\\
 & = & \sum_{-\infty <i<j<\infty} \EE\Big[H(\mathcal{L}_1,\ldots, \mathcal{L}_N)\cdot \mathcal{B}_{\bar{x}_n,\bar{y}_n,f}^{\mathfrak{l}_n,\mathfrak{r}_n}\big[F(\mathfrak{l}_n,\mathfrak{r}_n,B_1,\ldots, B_k)\big] \cdot {\bf 1}  \{\mathfrak{l}_n=(i+1)2^{-n}, \mathfrak{r}_n= j2^{-n}\}\Big].
\end{eqnarray*}
Then using the Brownian Gibbs property on each term in the summation above, it follows that the above is equivalent to
\begin{eqnarray*}
& & \sum_{-\infty <i<j<\infty} \EE\Big[H(\mathcal{L}_1,\ldots, \mathcal{L}_N) \cdot F\big(\mathfrak{l}_n,\mathfrak{r}_n,\mathcal{L}_1\big\vert_{(\mathfrak{l}_n,\mathfrak{r}_n)},\cdots, \mathcal{L}_k\big\vert_{(\mathfrak{l}_n,\mathfrak{r}_n)}\big) {\bf 1}\{\mathfrak{l}_n=(i+1)2^{-n}, \mathfrak{r}_n= j2^{-n}\}\Big] \\
&=& \EE\Big[H(\mathcal{L}_1,\ldots, \mathcal{L}_N) \cdot F\big(\mathfrak{l}_n,\mathfrak{r}_n,\mathcal{L}_1\big\vert_{(\mathfrak{l}_n,\mathfrak{r}_n)},\cdots, \mathcal{L}_k\big\vert_{(\mathfrak{l}_n,\mathfrak{r}_n)}\big) {\bf 1}_A\Big].
\end{eqnarray*}
By our choice of $F$, the mapping $(\mathfrak{l},\mathfrak{r})\mapsto F\big(\mathfrak{l}_n,\mathfrak{r}_n,\mathcal{L}_1\big\vert_{(\mathfrak{l}_n,\mathfrak{r}_n)},\ldots, \mathcal{L}_k\big\vert_{(\mathfrak{l}_n,\mathfrak{r}_n)}\big)$ is continuous almost surely with respect to the measure $\PP$ on $\mathcal{L}$. Therefore, $F\big(\mathfrak{l}_n,\mathfrak{r}_n,\mathcal{L}_1\big\vert_{(\mathfrak{l}_n,\mathfrak{r}_n)},\ldots, \mathcal{L}_k\big\vert_{(\mathfrak{l}_n,\mathfrak{r}_n)}\big)$ converges to $F\big(\mathfrak{l},\mathfrak{r},\mathcal{L}_1\big\vert_{(\mathfrak{l},\mathfrak{r})},\ldots, \mathcal{L}_k\big\vert_{(\mathfrak{l},\mathfrak{r})}\big)$ on the event $A$. By the dominated convergence theorem, this implies that
\begin{eqnarray*}
& &\lim_{n\to \infty} \EE\Big[H(\mathcal{L}_1,\ldots, \mathcal{L}_N) \cdot  F\big(\mathfrak{l}_n,\mathfrak{r}_n,\mathcal{L}_1\big\vert_{(\mathfrak{l}_n,\mathfrak{r}_n)},\ldots, \mathcal{L}_k\big\vert_{(\mathfrak{l}_n,\mathfrak{r}_n)}\big) {\bf 1}_A\Big]\\
&=&  \EE\Big[H(\mathcal{L}_1,\ldots, \mathcal{L}_N) \cdot  F\big(\mathfrak{l},\mathfrak{r},\mathcal{L}_1\big\vert_{(\mathfrak{l},\mathfrak{r})},\ldots, \mathcal{L}_k\big\vert_{(\mathfrak{l},\mathfrak{r})}\big) {\bf 1}_A\Big] .
\end{eqnarray*}
This completes the proof of the strong Brownian Gibbs property.
\end{proof}

\end{document}